\documentclass[11pt]{article}
\usepackage[utf8]{inputenc}
\usepackage[T1]{fontenc}
\usepackage[a4paper, margin=1in, nohead]{geometry}
\usepackage{amssymb, amsmath, amsthm, amsfonts,amsbsy,  bbm, dsfont, mathrsfs}
\usepackage{graphicx, graphics, enumerate}
\usepackage{booktabs, caption, subcaption, multirow, float, url, yfonts}
\usepackage{comment}
\usepackage{color}
\usepackage{xcolor}
\usepackage{unravel}
\usepackage{indentfirst}
\usepackage{enumitem}

\usepackage{setspace}
\usepackage{makecell} 
\setcellgapes{5pt}

\usepackage{hyperref}
\definecolor{persianblue}{rgb}{0.11, 0.22, 0.73}
\colorlet{col_linkcolor}{persianblue}
\colorlet{col_citecolor}{persianblue}
\colorlet{col_urlcolor}{persianblue}
\hypersetup{
    colorlinks=true,
    linkcolor=col_linkcolor, 
    citecolor=col_citecolor, 
    urlcolor=col_urlcolor}

\usepackage{natbib}
\newcommand{\iid}{\textit{i.i.d.}}
\newcommand{\inid}{\textit{i.n.i.d.}}

\DeclareMathOperator{\sign}{sign}

\newcommand{\eps}{\varepsilon}

\newcommand{\Prob}{\mathbb{P}}
\newcommand{\E}{\mathbb{E}}

\newcommand{\Nstar}{\mathbb{N}^{*}}

\newcommand{\Indicator}{\mathds{1}}

\newcommand{\Rb}{\mathbb{R}}

\newcommand{\Zb}{\mathbb{Z}}

\newcommand{\DeltaB}{\Delta_{n,\text{B}}}
\newcommand{\DeltaE}{\Delta_{n,\text{E}}}

\newcommand{\resp}{\textit{resp.}}

\newcommand{\Normal}{\mathcal{N}}

\newcommand{\Ktroisntilde}{\widetilde{K}_{3,n}}
\newcommand{\Ktroisn}{K_{3,n}}
\newcommand{\Kquatren}{K_{4,n}}
\newcommand{\lambdatroisn}{\lambda_{3,n}}

\newcommand{\caracfsum}{f_{S_n}}
\newcommand{\vp}{\text{p.v.}}

\newcommand{\colRev}[1]{#1}

\newcommand*{\ExistingBoundBEiid}{0.4690}
\newcommand*{\ExistingBoundBEinid}{0.5583}

\newcommand*{\NewBoundBEinid}{0.4403}

\newcommand*{\NewBoundBEiid}{0.3990}

\newcommand*{\rninidskew}[1]{r_{#1,n}^{\textnormal{inid,skew}}}
\newcommand*{\rniidskew}[1]{r_{#1,n}^{\textnormal{iid,skew}}}
\newcommand*{\rninidnoskew}[1]{r_{#1,n}^{\textnormal{inid,noskew}}}
\newcommand*{\rniidnoskew}[1]{r_{#1,n}^{\textnormal{iid,noskew}}}

\newcommand*{\RinidInt}{\overline R_{n}^{\textnormal{inid}}}
\newcommand*{\RiidIntNoskew}{\overline R_{n}^{\textnormal{iid,noskew}}}
\newcommand*{\RiidIntSkew}{\overline R_{n}^{\textnormal{iid,skew}}}
\newcommand*{\RiidInt}{\overline R_{n}^{\textnormal{iid}}}

\newcommand*{\Rinid}{R_{n}^{\textnormal{inid}}}
\newcommand*{\Riid}{R_{n}^{\textnormal{iid}}}

\newtheorem{theorem}{Theorem}
\newtheorem{proposition}[theorem]{Proposition}%
\newtheorem{corollary}[theorem]{Corollary}%
\newtheorem{lemma}[theorem]{Lemma}%
\newtheorem{example}{Example}%
\newtheorem{remark}{Remark}%

\newtheorem{definition}{Definition}%
\newtheorem{assumption}{Assumption}%

\title{Explicit non-asymptotic bounds for the distance to the first-order Edgeworth expansion}

\author{Alexis Derumigny\thanks{Delft University of Technology, Mourik Broekmanweg 6,
2628 XE  Delft, Netherlands.
\newline E-mail address: a.f.f.derumigny@tudelft.nl},
Lucas Girard\thanks{Centre de Recherche en \'Economie et de Statistiques (CREST), CNRS, \'Ecole polytechnique, GENES, ENSAE Paris, Institut Polytechnique de Paris, 91120 Palaiseau, France.
\newline E-mail address: lucas.girard@ensae.fr},
Yannick Guyonvarch\thanks{PSAE-INRAE, 22 Place de l'Agronomie, 91120 Palaiseau, France.
\newline E-mail address:
yannick.guyonvarch@inrae.fr
\newline \indent We would like to thank professors Victor-Emmanuel Brunel and Xavier D'Haultf\oe{}uille for insightful discussions as well as seminar participants at CREST, University of Surrey, Université Paris-Saclay, and CIREQ Montreal Econometrics Conference. 
All possible errors remain ours.
Part of this article was written while A.D. was employed by the University of Twente and Y.G. was employed by the University Paris-Sud and then by Télécom Paris. No specific funding was received to assist with the preparation of this manuscript.}
}

\date{{\small This is the working paper version of the article \cite{derumigny2024explicit} published in \textit{Sankhya A 86}, 261–336 (2024). 
DOI: \url{https://doi.org/10.1007/s13171-023-00320-y}}}

\begin{document}

\maketitle

\begin{abstract}
In this article, we obtain explicit bounds on the uniform distance between the cumulative distribution function of a standardized sum \(S_n\) of \(n\) independent centered random variables with moments of order four and its first-order Edgeworth expansion. Those bounds are valid for any sample size with \(n^{-1/2}\) rate under moment conditions only and \(n^{-1}\) rate under additional regularity constraints on the tail behavior of the characteristic function of~\(S_n\). In both cases, the bounds are further sharpened if the variables involved in \(S_n\) are unskewed. We also derive new Berry-Esseen-type bounds from our results and discuss their links with existing ones. 
Following these theoretical results, we discuss the practical use of our bounds, which depend on possibly unknown moments of the distribution of~\(S_n\).
Finally, we apply our bounds to investigate several aspects of the non-asymptotic behavior of one-sided tests: informativeness, sufficient sample size in experimental design, distortions in terms of levels and p-values.
\end{abstract}

\medskip



\section{Introduction}

As the number of observations~$n$ in a statistical experiment goes to infinity, many statistics of interest have the property to converge weakly to a $\Normal(0,1)$ distribution, once adequately centered and scaled, see, e.g., Chapter~5 of \cite{vanderVaart2000} for a thorough introduction.
Hence, when little is known on the distribution of a statistic for a fixed sample size, a classical approach to conduct inference on the parameters of the statistical model amounts to approximating that distribution by its tractable Gaussian limit.
A recurring theme in statistics and probability is thus to quantify the distance between those two distributions for a given~$n$.

In this article, we present some refined results in the canonical case of a standardized sum of independent random variables.
We consider independent but not necessarily identically distributed random variables to encompass a broader range of applications.
For instance, certain bootstrap schemes such as the multiplier ones (see Chapter~9 in \cite{vanderVaartWellner1996} or Chapter~10 in \cite{Kosorok2006}) boil down to studying a sequence of mutually independent not necessarily identically distributed (\inid{}) random variables conditional on the initial sample.

More formally, let ${(X_i)_{i=1, \dots, n}}$ be a sequence of \inid{} random variables satisfying
for every~${i \in \{1,...,n\}}$,
${\E[X_i] = 0}$ and ${\gamma_i := \E[ X_i^4 ] < +\infty}$.
We also define the standard deviation $B_n$ of the sum of the $X_i$'s, i.e., ${B_n := \sqrt{\sum_{i=1}^n\E[X_i^2]}},$ so that the standardized sum can be written as ${S_n := \sum_{i=1}^n X_i/B_n}$.
Finally, we define the average individual standard deviation ${\overline{B}_n := B_n/\sqrt{n}}$ and the average standardized third raw moment ${\lambda_{3,n} := \frac{1}{n}\sum_{i=1}^n\mathbb{E}[X_i^3]/\overline{B}_n^3}$.
The main results of this article are of the form
\begin{align}
\label{eq:edg_exp_res}
    & \underbrace{\sup_{x \in \Rb}
    \left|
    \Prob(S_n \leq x) - \Phi(x) - \frac{\lambda_{3,n}}{6\sqrt{n}}(1-x^2)\varphi(x)
    \right|
    }_{\displaystyle =:\Delta_{n,\text{E}}}
    \leq \delta_n,
\end{align}
where $\Phi$ is the cumulative distribution function of a standard Gaussian random variable, $\varphi$ its density function and $\delta_n$ is a positive sequence that depends on the first four moments of $(X_i)_{i=1,\dots,n}$ and tends to zero under some regularity conditions. 
In the following, we use the notation $G_n(x) := \Phi(x) + \lambda_{3,n} (6\sqrt{n})^{-1} (1-x^2)\varphi(x)$.

The quantity $G_n(x)$ is usually called the one-term Edgeworth expansion of $\Prob\left(S_n \leq x\right)$, hence the letter E in the notation $\Delta_{n,\text{E}}$.
Controlling the uniform distance between $\Prob\left(S_n \leq \cdot \right)$ and $G_n(\cdot)$ has a long tradition in statistics and probability, see for instance \cite{esseen1945} and the books by \cite{cramer1962} and \cite{bhattacharyarao1976}.
As early as in the work of \cite{esseen1945}, it was acknowledged that in independent and identically distributed (\iid{}) cases, $\Delta_{n,\text{E}}$ was of the order $n^{-1/2}$ in general and of the order $n^{-1}$ if $(X_i)_{i=1,\dots,n}$ has a nonzero continuous component.
These results were then extended in a wide variety of directions, often in connection with bootstrap procedures,
see for instance \cite{hall1992bootstrap} and \cite{lahiri2003resampling} for the dependent case.

A one-term Edgeworth expansion can be seen as a refinement of the so-called Berry-Esseen inequality (\cite{berry1941}, \cite{esseen1942}) which goal is to bound
\begin{equation}
\label{eq:def_Delta_n_B}
    \Delta_{n,\text{B}} := \sup_{x \in \Rb}\big| \Prob(S_n \leq x) - \Phi(x) \big|.
\end{equation}
The refinement stems from the fact that in $\Delta_{n,\text{E}},$ the distance between $\Prob\left(S_n \leq \cdot \right)$ and $\Phi$ is adjusted for the presence of non-asymptotic skewness in the distribution of $S_n$.
Contrary to the literature on Edgeworth expansions, there is a substantial amount of work devoted to explicit constants in the Berry-Esseen inequality and its extensions, see, e.g., \cite{bentkus1996}, \cite{bentkus2003}, \cite{pinelis2016}, \cite{chernozhukov2017}, \cite{raivc2018multivariate}, \cite{raic2019}.
The sharpest known result in the \inid{} univariate framework is due to \cite{shevtsova2013}, which shows that for every ${n \in \Nstar}$, if ${\E[|X_i|^3]<+\infty}$ for every~${i \in \{1,...,n\}}$, then $\Delta_{n,\text{B}} \leq \ExistingBoundBEinid  \, K_{3,n} / \sqrt{n}$ where $K_{p,n} := n^{-1} \sum_{i=1}^n\E[|X_i|^p]/(\overline{B}_n)^p$, for ${p \in \Nstar}$, denotes the average standardized $p$-th absolute moment.
$K_{p,n}$ measures tail thickness, with $K_{2,n}$ normalized to 1 and $\Kquatren$ the kurtosis.
An analogous result is given in \cite{shevtsova2013} under the \iid{} assumption where $\ExistingBoundBEinid $ is replaced with $\ExistingBoundBEiid$.
A close lower bound is due to \cite{esseen1956}: there exists a distribution such that $\Delta_{n,\text{B}} = (C_B/\sqrt{n}) \left(n^{-1} \sum_{i=1}^n\E[|X_i|^3]/\overline{B}_n^3\right)$ with ${C_B \approx 0.4098}$. Another line of research applies Edgeworth expansions in order to get a bound on $\Delta_{n,\text{B}}$ that contains higher-order terms, see \cite{adell2008}, \cite{boutsikas2011} and~\cite{zhilova2020new}.

Despite the breadth of those theoretical advances, there remain some limits to take full advantage of those results even in simple statistical applications, for instance, when conducting inference on the expectation of a real random variable.\footnote{In this article, we only give results for \textit{standardized} sums of random variables, i.e., sums that are rescaled by their standard deviation. In practice, the variance is unknown and has to be replaced with some empirical counterpart, leading to what is usually called a \textit{self-normalized} sum. This is an important question in practice that we leave aside for future research. There exist numerous results on self-normalized sums in the fields of Edgeworth expansions and Berry-Esseen inequalities (\cite{hall1987}, \cite{delapena2009}). 
However, the practical limitations of existing results that we point out in our work still prevail.} 
If we focus on Berry-Esseen inequalities, we show in Section~\ref{subsec:uninformativeness_examples} shows that even the sharpest upper bound to date on $\Delta_{n,\text{B}}$ can be uninformative when conducting inference on an expectation even for $n$ larger than 59,000. 
Therefore, it is natural to wonder whether bounds derived from a one-term Edgeworth expansion could be tighter in moderately large samples (such as a few thousands). In the \iid{} case and under some smoothness conditions, \cite{senatov2011} obtains such improved bounds.
To our knowledge, the question is nevertheless still open in the \inid{} setup, as well as in the general setup when no condition on the characteristic function is assumed.
In particular, most articles that present results of the form of~\eqref{eq:edg_exp_res} do not provide a fully explicit value for $\delta_n$, that is, $\delta_n$ is defined up to some ``universal'' but unknown constant, see for instance~\cite{cramer1962} and~\cite{bentkus1997}, among others.

In this article, we derive novel inequalities of the form of~\eqref{eq:edg_exp_res} that aim to be relevant in practical applications.
Such ``user-friendly'' bounds seek to achieve two goals.
First, we provide explicit values for $\delta_n$, which are implemented in the new \texttt{R} package \texttt{BoundEdgeworth} \cite{packageBoundEdgeWorth}
using the function \texttt{Bound\_EE1} (the function \texttt{Bound\_BE} provides a bound on \(\DeltaB\)).
Second, the bounds $\delta_n$ should be small enough to be informative even with small (${n \approx}$~hundreds) to moderate (${n \approx}$~thousands) sample sizes.
We obtain these bounds in an \iid{} setting and in a more general \inid{} case only assuming finite fourth moments.

We give improved bounds on $\Delta_{n,\text{E}}$ under some regularity assumptions on the tail behavior of the characteristic function $\caracfsum$ of~$S_n$.
Such conditions are related to the continuity of the distribution of~$S_n$ and the differentiability of the corresponding density (with respect to Lebesgue's measure).
These are well-known conditions required for the Edgeworth expansion to be a good approximation of ${\Prob(S_n\leq \cdot \,)}$ with fast rates.
Our main results are summed up in Table~\ref{tab:summary_all_results}.

\begin{table}[ht]
\begin{center}
\begin{minipage}{\textwidth}
\makegapedcells 
\begin{tabular*}{\textwidth}{@{\extracolsep{\fill}}lll@{\extracolsep{\fill}}}
    \toprule
        \textit{Setup}
        & \textit{General case}
        &\textit{Under regularity assumptions on} ${\caracfsum}$
        \\
    \midrule
        \multirow{2}{*}{ \inid{} }
        & $\dfrac{0.3990 K_{3,n}}{\sqrt{n}} + O(n^{-1})$
        & $\dfrac{0.195 \, K_{4,n}
        + 0.038 \, \lambda_{3,n}^2}{n} + O(n^{-5/4}+n^{-p/2})$ \\
        & (Theorem~\ref{thm:nocont_choiceEps}) & (Corollary~\ref{cor:improvement_inid_case}) \\
    \midrule
        \multirow{2}{*}{ \iid{} }
        & $\dfrac{0.1995(K_{3,n}+1)}{\sqrt{n}} + O(n^{-1})$ 
        & $\dfrac{0.195 \, K_{4,n}
        + 0.038 \, \lambda_{3,n}^2}{n} 
        + O(n^{-5/4})$ \\
        & (Theorem~\ref{thm:nocont_choiceEps}) & (Corollary~\ref{cor:improvement_iid_case}) \\
    \bottomrule
\end{tabular*}
\caption{Summary of the new bounds on $\Delta_{n,\text{E}}$ under different scenarios.
    We use the notation $O(n^{-\alpha})$ to indicate terms that are smaller than $C n^{-\alpha}$ for some constant $C$.
    All these terms are given with explicit expressions for any sample size and most of them are significantly reduced when there is no skewness. 
    $p \geq 0$ is a constant depending on the tail decay of the characteristic function~$\caracfsum$.
    Note that the corresponding term is dominant if \(p \leq 2\) (see Section~\ref{ssec:thms_cont} for additional discussions).
    For this application of Corollary~\ref{cor:improvement_iid_case}, we impose an alternative tail decay condition, namely $\sup_n \kappa_n < 1$ (see Section~\ref{ssec:thms_cont} for the definition of $\kappa_n$).}
\label{tab:summary_all_results}
\end{minipage}
\end{center}
\end{table}

In the rest of this section, we
introduce notation used in the rest of the paper.
Section~\ref{ssec:thms_nocont} presents our bounds on $\DeltaE$ under moment conditions only in \inid{} or \iid{} settings.
In Section~\ref{ssec:thms_cont}, we develop tighter bounds under regularity assumptions on the characteristic function of~$S_n$.
They rely on an alternative control of \(\DeltaE\) that involves the integral of \(\caracfsum\),
enabling us to use additional regularity assumptions on the tails of that function.
In Section~\ref{sec:practical_considerations}, we discuss practical aspects related to our bounds: how to choose or estimate the moments of the distribution of~\(S_n\) involved in order to compute our bounds.
We also perform numerical comparisons between our and existing bounds for some particular distributions (Student and Gamma).\footnote{
The code to replicate our results is available in the Github repository \newline \url{https://github.com/AlexisDerumigny/Reproducibility-BoundsDistanceEdgeworth}. 
}
In Section~\ref{sec:application_to_testing}, we apply our results to analyze several aspects of one-sided tests based on the normal approximation of a sample mean.
In particular, based on our bounds, we propose a new method to compute sufficient sample sizes for experimental design with given effect size to be detected and nominal power.
All proofs are postponed in the appendix.
The proofs of the main results are gathered in Appendix~\ref{appendix:sec:proof_main_theorems}, relying on the computations of Appendix~\ref{appendix:sec:omega}. 
Useful lemmas are given in Appendix~\ref{sec:lemmas}.

\medskip

\textbf{Additional notation.}
$\vee$ (\resp{} $\wedge$) denotes the maximum (\resp{} minimum) operator. 
For a random variable $X$, we denote its probability distribution by $P_X$.
For a distribution $P$, let $f_P$ denote its characteristic function; similarly, for a random variable $X$, we denote by $f_X$ its characteristic function.
We recall that $f_{\mathcal{N}(0,1)}(t)=e^{-t^2/2}$.
We denote the (extended) lower incomplete Gamma function by
$\gamma(a, x) := \int_0^x |u|^{a-1} e^{-u} du$
(for $a > 0$ and $x \in \Rb$),
the upper incomplete Gamma function by
$\Gamma(a,x) := \int_x^{+\infty} u^{a-1} e^{-u} du$
(for $a \geq 0$ and $x  > 0$)
and the standard gamma function by 
$\Gamma(a) := \Gamma(a,0) = \int_0^{+\infty} u^{a-1} e^{-u} du$
(for $a > 0$).
For two sequences $(a_n),$ $(b_n),$ we write $a_n = O(b_n)$ whenever there exists $C>0$ such that ${a_n \leq C b_n}$; $a_n = o(b_n)$ whenever $a_n / b_n \to 0$; and $a_n \asymp b_n$ whenever $a_n = O(b_n)$ and $b_n = O(a_n)$.
We denote by $\chi_1$ the constant $\chi_1 := \sup_{x>0} x^{-3} |\cos(x)-1+x^2/2| \approx 0.099$
\citep{shevtsova2010}, and by $\theta_1^*$ the unique root in $(0,2\pi)$ of the equation $\theta^2+2\theta\sin(\theta)+6(\cos(\theta)-1)=0$. 
We also define $t_1^* := \theta_1^* / (2\pi) \approx 0.64$
\citep{shevtsova2010}.
For every~${i \in \Nstar}$, we define the individual standard deviation ${\sigma_{i} := \sqrt{\E[X_i^2]}}$.
Henceforth, we reason for a fixed arbitrary sample size~${n \in \Nstar}$.
Densities and continuous distributions are always assumed implicitly to be with respect to Lebesgue's measure.

\medskip

For clarity, we define below the concept of an explicit expression. In the rest of the article, the goal is to find bounds on $\DeltaE$ that are explicit expressions in the sense of Definition~\ref{def:explicit_expression}.

\begin{definition}
    An expression is called \textbf{explicit} if it can be written as a finite sequence of terms.
    A \textbf{term} is defined as
    \begin{itemize}[nolistsep]
        \item either a numerical constant (i.e. a computable real number),
        \item or one of the parameters of the framework (such as $n$, $\lambda_{3,n}$, $\Kquatren$ and so on),
        \item or one of the standard functions (rational functions, exponential functions, logarithmic functions, incomplete Gamma functions, indicator functions, absolute value, maximum or minimum) applied to a finite set of terms,
        \item or, recursively, as an explicit expression itself.
    \end{itemize}
    
    \label{def:explicit_expression}
\end{definition}

\section{Control of \texorpdfstring{$\boldsymbol{\Delta_{n,\text{E}}}$}{Delta(n,E)} under moment conditions only}
\label{ssec:thms_nocont}

We start by introducing two versions of our basic assumptions on the distribution of the variables~$(X_i)_{i=1, \dots, n}$.

\begin{assumption}[Moment conditions in the \inid{} framework]
\label{hyp:basic_as_inid}
	$(X_i)_{i=1, \dots, n}$ are independent and centered random variables 
	such that for every $i=1,\dots,n$, the fourth raw individual moment $\gamma_{i} := \E[ X_i^4 ]$ is positive and finite.
\end{assumption}

\begin{assumption}[Moment conditions in the \iid{} framework]
\label{hyp:basic_as_iid}
	$(X_i)_{i=1,\dots,n}$ are \iid{} centered random variables such that the fourth raw moment $\gamma_{n} := \E[ X_n^4 ]$ is positive and finite.
\end{assumption}

Assumption~\ref{hyp:basic_as_iid} corresponds to the classical \iid{} sampling with finite fourth moment while Assumption~\ref{hyp:basic_as_inid} is its generalization in the \inid{} framework.
Those two assumptions primarily ensure that enough moments of $(X_i)_{i=1,\dots,n}$ exist to build a non-asymptotic upper bound on $\Delta_{n,\text{E}}.$
In some applications, such as the bootstrap, it is required to consider an array of random variables $(X_{i,n})_{i=1,\dots,n}$ instead of a sequence. For example, \cite{efron1979}'s nonparametric bootstrap procedure consists in drawing $n$ elements in the random sample $(X_{1,n},...,X_{n,n})$ with replacement. Conditional on $(X_{i,n})_{i=1,\dots,n},$ the $n$ values drawn with replacement can be seen as a sequence of $n$ \iid{} random variables with distribution $\frac{1}{n}\sum_{i=1}^n\delta_{\{X_{i,n}\}}$,
denoting by $\delta_{\{a\}}$ the Dirac measure at a given point~${a \in \Rb}$.
Our results encompass these situations directly.
Nonetheless, we do not use the array terminology here as our results hold non-asymptotically, i.e., for any fixed sample size~$n$.

To state our first theorem, 
remember that ${\overline{B}_n := (1 / \sqrt{n}) \sqrt{\sum_{i=1}^n\sigma_{i}^2}}$, for $p \in \Nstar$,
$K_{p,n} := n^{-1} \sum_{i=1}^n\E[|X_i|^p]/ \overline{B}_n^p$,
and let us introduce
\(\Ktroisntilde := K_{3,n} + \frac{1}{n}\sum_{i=1}^n\mathbb{E}|X_i|\sigma_{i}^2 / \overline{B}_n^3\),
$\Delta:= (1 - 4 \chi_1 - \sqrt{K_{4,n}/n}) / 2$,
and the terms $\rninidskew{1}$, $\rninidnoskew{1}$, $\rniidskew{1}$ and $\rniidnoskew{1}$.

These remainder terms are defined by:
\begin{align}
    \rninidskew{1}
    &:= \frac{(14.1961 + 67.0415) \, \Ktroisntilde^4}{16\pi^4 n^2}
    + \frac{4.3394 \, |\lambdatroisn| \Ktroisntilde^3}{8 \pi^3 n^2}+ \frac{1.0435 K_{4,n}^{5/4}}{n^{5/4}} \nonumber \\ 
    & + \frac{1.1101 \Kquatren^{3/2} + 31.9921 |\lambdatroisn| \times \Kquatren}{n^{3/2}} +
    \frac{0.6087 \Kquatren^{7/4}}{n^{7/4}}
    + \frac{9.8197\Kquatren^{2}}{n^2} \nonumber \\
    &+ \frac{ |\lambdatroisn| \big( \Gamma( 3/2 , \sqrt{0.2} (n/K_{4,n})^{1/4} \wedge 2 \sqrt{n} / \Ktroisntilde)
    - \Gamma( 3/2 , 2 \sqrt{n} / \Ktroisntilde) \big) }{\sqrt{n}} \nonumber \\
    &+ \frac{1.0253\Ktroisn}{6\pi\sqrt{n}} \Bigg\{ 0.5|\Delta|^{-3/2}\Indicator_{\{\Delta \neq 0\}} \times \bigg|\gamma(3/2, 4 \Delta n / \Ktroisntilde^2) \nonumber \\
    & \qquad\qquad\qquad\;\; - \gamma\big(3/2, 2\Delta ( 0.1 (n/K_{4,n})^{1/2} \wedge 2 n / \Ktroisntilde^2 ) \big) \bigg| \nonumber \\
    & \qquad\qquad\qquad\;\; + \Indicator_{\{\Delta = 0\}}
    \frac{ ( 2 \sqrt{n} / \Ktroisntilde )^3
    - (\sqrt{0.2} (n/K_{4,n})^{1/4} \wedge 2 \sqrt{n} / \Ktroisntilde )^3 }{3} \Bigg\},
    \label{eq:def_r_1n_inid_eps_eq_01}
\end{align}
\begin{align}
    \rninidnoskew{1}
    &:= \frac{(14.1961 + 67.0415) \, \Ktroisntilde^4}{16\pi^4 n^2} + 
    \frac{0.6661 \Kquatren^{3/2}}{n^{3/2}} + 
    \frac{6.1361 \Kquatren^{2}}{n^{2}} \nonumber \\
    &+ \frac{1.0253\Kquatren}{6\pi n} \Bigg\{ 0.5|\Delta|^{-2}\Indicator_{\{\Delta \neq 0\}} \times \bigg|\gamma(2, 4 \Delta n / \Ktroisntilde^2) \nonumber \\
    & \qquad\qquad\qquad\;\; - \gamma\big(2, 2\Delta ( 0.1 (n/K_{4,n})^{1/2} \wedge 2 n / \Ktroisntilde^2 ) \big) \bigg| \nonumber \\
    & \qquad\qquad\qquad\;\; + \Indicator_{\{\Delta = 0\}}
    \frac{ ( 2 \sqrt{n} / \Ktroisntilde )^4
    - (\sqrt{0.2} (n/K_{4,n})^{1/4} \wedge 2 \sqrt{n} / \Ktroisntilde )^4 }{4} \Bigg\},
    \label{eq:def_r_1n_inid_eps_eq_01_noskew}
\end{align}
\begin{align}
    \rniidskew{1} & :=
    \frac{(14.1961 + 67.0415) \, \Ktroisntilde^4}{16\pi^4 n^2}
    + \frac{4.3394 \, |\lambdatroisn| \Ktroisntilde^3}{8 \pi^3 n^2}
    + \RiidIntSkew \nonumber \\
    & + \frac{1.306 \big( e_{2,n}
    - 1.006792
    \big) \lambdatroisn^2}{36 n} \nonumber \\
    &+ \frac{ |\lambdatroisn| \big( \Gamma( 3/2 , \sqrt{0.2} (n/K_{4,n})^{1/4} \wedge 2 \sqrt{n} / \Ktroisntilde)
    - \Gamma( 3/2 , 2 \sqrt{n} / \Ktroisntilde) \big) }{\sqrt{n}} \nonumber \\
    &+ 
    \frac{1.0253 \times 2^{5/2} \, \Ktroisn}{3 \pi \sqrt{n}}
    \bigg( \Gamma \big( 3/2, \big\{ \sqrt{0.2} (n/\Kquatren)^{1/4} \wedge 2\sqrt{n}/\Ktroisntilde \big\}^2/8 \big) \nonumber 
    \\
    & \qquad \qquad \qquad \qquad \qquad - \Gamma \big( 3/2,
    4 n / (8 \Ktroisntilde^2)
    \big) \bigg),
    \label{eq:def_r_1n_iid_eps_eq_01}
\end{align}
and
\begin{align}
    \rniidnoskew{1}
    &:= \frac{(14.1961 + 67.0415) \, \Ktroisntilde^4}{16\pi^4 n^2}
    + \RiidIntNoskew \nonumber \\
    &+ \frac{16 \times 1.0253\Kquatren}{3\pi n} 
    \bigg( \Gamma \big( 2 , \big\{ \sqrt{0.2} (n/\Kquatren)^{1/4}
    \wedge 2\sqrt{n}/\Ktroisntilde \big\}^2/8 \big) \nonumber \\
    & \qquad\qquad\qquad\qquad\qquad - \Gamma \big( 2 , 4 n / (8 \Ktroisntilde^2) \big) \bigg),
    \label{eq:def_r_1n_iid_eps_eq_01_noskew}
\end{align}
where
\begin{normalsize}
\begingroup \allowdisplaybreaks
\begin{align}
    \RiidIntSkew 
    &:=  \dfrac{0.06957 |\lambdatroisn|}{n^{1.5}} 
    + \dfrac{0.6661 \Kquatren}{n^{2}} 
    + \dfrac{0.4441 \lambdatroisn^{2}}{n^{2}} 
    + \dfrac{0.6087 |\lambdatroisn|  \times \Kquatren}{n^{2.5}} \nonumber \\
    & + \dfrac{0.2221 \Kquatren^{2}}{n^{3}} 
    \nonumber \\
    & + e_{2,n} \times \Bigg(\dfrac{0.1088 \Kquatren^{2}}{n^{2}} 
    + \dfrac{1.3321 \Kquatren}{n^{2}} 
    + \dfrac{0.3972 |\lambdatroisn|  \times \Kquatren^{0.75}}{n^{2.25}} \nonumber \\
    & \qquad + \dfrac{0.04441 \Kquatren^{1.5}}{n^{2.5}} 
    + \dfrac{0.02961 \Kquatren^{0.5}  \times \lambdatroisn^{2}}{n^{2.5}} 
    + \dfrac{0.006620 |\lambdatroisn|  \times \Kquatren^{1.25}}{n^{2.75}} 
    \nonumber \\
    & + \dfrac{0.0003701 \Kquatren^{2}}{n^{3}} 
    + \dfrac{4.0779 }{n^{2}} 
    + \dfrac{2.4316 |\lambdatroisn|  \times \Kquatren^{-0.25}}{n^{2.25}} 
    + \dfrac{0.2719 \Kquatren^{0.5}}{n^{2.5}} 
    \nonumber \\
    & + \dfrac{0.1813 \Kquatren^{-0.5}  \times \lambdatroisn^{2}}{n^{2.5}}
    + \dfrac{0.1216 |\lambdatroisn|  \times \Kquatren^{0.25}}{n^{2.75}} 
    + \dfrac{0.002266 \Kquatren}{n^{3}} 
    \nonumber \\ 
    & + \dfrac{0.3625 |\lambdatroisn|^{2}  \times \Kquatren^{-0.5}}{n^{2.5}} 
    + \dfrac{0.05404 |\lambdatroisn|  \times \Kquatren^{-0.75}  \times \lambdatroisn^{2}}{n^{2.75}}
    \nonumber \\
    & + \dfrac{0.01209 |\lambdatroisn|^{2}  \times \Kquatren^{0}}{n^{3}} 
    + \dfrac{0.002027 |\lambdatroisn|  \times \Kquatren^{0.75}}{n^{3.25}} 
    + \dfrac{0.004531 \Kquatren}{n^{3}} 
    \nonumber \\
    & + \dfrac{0.006042 \Kquatren^{0}  \times \lambdatroisn^{2}}{n^{3}} 
    + \dfrac{7.552 \times 10^{-5} \Kquatren^{1.5}}{n^{3.5}} 
    \nonumber \\ & + \dfrac{0.002014 \Kquatren^{-1}  \times \lambdatroisn^{4}}{n^{3}} 
    + \dfrac{0.0009006 |\lambdatroisn|  \times \Kquatren^{-0.25}  \times \lambdatroisn^{2}}{n^{3.25}} \nonumber \\
    & + \dfrac{5.035 \times 10^{-5} \Kquatren^{0.5}  \times \lambdatroisn^{2}}{n^{3.5}} 
    \nonumber \\ & + \dfrac{0.0001007 |\lambdatroisn|^{2}  \times \Kquatren^{0.5}}{n^{3.5}} 
    + \dfrac{1.126 \times 10^{-5} |\lambdatroisn|  \times \Kquatren^{1.25}}{n^{3.75}} 
    + \dfrac{3.147 \times 10^{-7} \Kquatren^{2}}{n^{4}} 
    \nonumber \\
    & + \dfrac{0.2983 |\lambdatroisn|  \times \Kquatren}{n^{1.5}} 
    \nonumber \\ 
    & + \dfrac{1.8261 |\lambdatroisn|}{n^{1.5}} 
    + \dfrac{0.5445 |\lambdatroisn|^{2}  \times \Kquatren^{-0.25}}{n^{1.75}} 
    + \dfrac{0.06087 |\lambdatroisn|  \times \Kquatren^{0.5}}{n^{2}} \nonumber \\
    & + \dfrac{0.04058 |\lambdatroisn|  \times \Kquatren^{-0.5}  \times \lambdatroisn^{2}}{n^{2}} 
    \nonumber \\ & + \dfrac{0.009074 |\lambdatroisn|^{2}  \times \Kquatren^{0.25}}{n^{2.25}} 
    + \dfrac{0.0005073 |\lambdatroisn|  \times \Kquatren}{n^{2.5}}
    \Bigg),
    \label{eq:definition_R2n_int_skew_bound_when_eps_01}
\end{align}
\endgroup
\begingroup \allowdisplaybreaks
\begin{align}
    \RiidIntNoskew  
    & := \dfrac{0.6661 \Kquatren}{n^{2}} 
    + \dfrac{0.2221 \Kquatren^{2}}{n^{3}} 
    + e_{2,n} \times
    \Bigg(\dfrac{0.1088 \Kquatren^{2}}{n^{2}} \nonumber \\
    & + \dfrac{1.3321 \Kquatren}{n^{2}} 
    + \dfrac{0.04441 \Kquatren^{1.5}}{n^{2.5}} 
    \nonumber \\
    & + \dfrac{0.0003701 \Kquatren^{2}}{n^{3}}
    + \dfrac{4.0779 }{n^{2}}
    + \dfrac{0.2719 \Kquatren^{0.5}}{n^{2.5}} 
    + \dfrac{0.002266 \Kquatren}{n^{3}}
    \nonumber \\
    & + \dfrac{0.004531 \Kquatren}{n^{3}} 
    + \dfrac{7.552 \times 10^{-5} \Kquatren^{1.5}}{n^{3.5}} 
    + \dfrac{3.147 \times 10^{-7} \Kquatren^{2}}{n^{4}} 
    \Bigg).
    \label{eq:definition_R2n_int_noskew_bound_when_eps_01}
\end{align}
\endgroup
and
\begin{align*}
    e_{2,n} &:= \exp\Bigg(
    0.0119 + 0.000071 \times
    \bigg( \frac{42.9326|\lambda_{3,n}|}{(K_{4,n}^{1/4}n^{1/4})}
    + 4.8 \left(\frac{K_{4,n}}{n}\right)^{1/2} \\
    & + \frac{3.2\lambda_{3,n}^2}{(K_{4,n}n)^{1/2}}
    + \frac{0.7156K_{4,n}^{1/4}|\lambda_{3,n}|}{n^{3/4}}
    + \frac{0.04K_{4,n}}{n} \bigg) \Bigg).
\end{align*}
\end{normalsize}

The following theorem is proved in Sections~\ref{ssec:proof:inid_nocont} (``\inid{}'' case) and~\ref{ssec:proof:iid_nocont} (``\iid{}'' case).

\begin{theorem}[Control of the one-term Edgeworth expansion with bounded moments of order four]
\label{thm:nocont_choiceEps} 
    If Assumption~\ref{hyp:basic_as_inid} (\textit{resp.} Assumption~\ref{hyp:basic_as_iid}) holds and $n \geq 3$, we have the bound
    \begin{equation}
        \DeltaE \leq  \frac{0.1995 \, \widetilde{K}_{3,n}}{\sqrt{n}}
        + \frac{0.031 \, \widetilde{K}_{3,n}^2
        + 0.195 \, K_{4,n}
        + 0.054 \, |\lambda_{3,n}|\widetilde{K}_{3,n}
        + 0.03757 \, \lambda_{3,n}^2}{n} + r_{1,n} \, ,
    \label{eq:generall_bound_nocont}
    \end{equation}
    where $r_{1,n}$ is one of the four possible remainders $\rninidskew{1}$, $\rninidnoskew{1}$, $\rniidskew{1}$ or $\rniidnoskew{1}$, depending on whether Assumption~\ref{hyp:basic_as_inid} (``\inid{}'' case) or~\ref{hyp:basic_as_iid} (``\iid{}'' case) is satisfied and whether $\E[X_i^3]=0$ for every $i = 1,\dots,n$ (``noskew'' case) or not (``skew'' case).
\end{theorem}

{\color{black}
\begin{remark}
    Assume that there exists a constant $K_4$ such that $K_{4,n} \leq K_4$ for all $n \geq 3$ (this is the case, for example, if the data is an \iid~sample from a given infinite homogeneous population).
    Then the remainder terms can be bounded in the following way:
    $|\rninidskew{1}| = O(n^{-5/4})$, 
    $|\rninidnoskew{1}| = O(n^{-3/2})$,
    $|\rniidskew{1}| = O(n^{-5/4})$, and
    $|\rniidnoskew{1}| = O(n^{-2})$.
    This can be seen directly from the previous equations, as it is always possible to find the main term, and then bound all the others by the required powers.
\end{remark}

\begin{remark}
    In the regime where $K_{4,n}$ tends to infinity faster than $\sqrt{n}$, our bounds do not tend to $0$. This is the case in particular for the term that is multiplied by $\Indicator_{\{\Delta \neq 0\}}$.
    In this case, the bounds given by Theorem~\ref{thm:nocont_choiceEps} are still valid; in some cases, the right-hand side will be larger than $1$ and therefore the inequality trivially still holds.
    This can be interpreted in the following sense: the average kurtosis of the distribution increases too fast for the distance to the first-order Edgeworth expansion to be controlled by our techniques.
\end{remark}
}

Note that it is possible to replace $\Ktroisntilde$ by the simpler upper bound $2 K_{3,n}$ under Assumption~\ref{hyp:basic_as_inid} (respectively by $K_{3,n}+1$ under Assumption~\ref{hyp:basic_as_iid}).
This theorem displays a bound of order $n^{-1/2}$ on $\DeltaE$ {\color{black}in the regime where $K_{4,n}$ is bounded by a fixed constant.}
The rate $n^{-1/2}$ cannot be improved when only assuming moment conditions on $(X_i)_{i=1,\dots,n}$ (\cite{esseen1945}, \cite{cramer1962}).
Another nice aspect of those bounds is their dependence on $\lambdatroisn$.
For many classes of distributions, $\lambda_{3,n}$ can, in fact, be exactly zero.
This is the case if for every $i = 1,\dots,n$, $X_i$ has a non-skewed distribution, such as any distribution that is symmetric around its expectation.
More generally, $|\lambda_{3,n}|$ can be substantially smaller than $K_{3,n}$, decreasing the related terms.

As mentioned in the Introduction, we are not aware of explicit bounds on $\DeltaE$ under moment conditions only.
It is thus difficult to assess how our bounds compare to the literature.
On the other hand, there exist well-established bounds on $\DeltaB$.
Using Theorem~\ref{thm:nocont_choiceEps}, the bound $(1-x^2)\varphi(x) / 6 \leq \varphi(0)/6 \leq 0.0665 $ for all $x \in \Rb$, and applying the triangle inequality, we can control \(\DeltaB\) as well.
More precisely, for every $n \geq 3$, we have
\begin{equation}
\label{eq:be_bound}
    \DeltaB \leq \frac{0.1995\widetilde{K}_{3,n}+0.0665|\lambda_{3,n}|}{\sqrt{n}} 
    + O(n^{-1}).
\end{equation}
Under Assumption~\ref{hyp:basic_as_inid}, \(\Ktroisntilde \leq 2 \Ktroisn \).
Combined with the refined inequality $|\lambda_{3,n}| \leq 0.621K_{3,n}$ \citep[Theorem~1]{pinelis2011relations}, we can derive a simpler bound that involves only $K_{3,n}$
\begin{align*}
    \frac{0.1995\widetilde{K}_{3,n}
    + 0.0665|\lambda_{3,n}|}{\sqrt{n}}
    \leq \frac{\NewBoundBEinid K_{3,n}}{\sqrt{n}}.
\end{align*}

The bound $\DeltaB \leq \NewBoundBEinid K_{3,n} / \sqrt{n} + O(n^{-1})$ 
is already tighter than the sharpest known Berry-Esseen inequality in the \inid{} framework, $\DeltaB \leq \ExistingBoundBEinid K_{3,n}/\sqrt{n}$,
as soon as the remainder term \(O(n^{-1}\) is smaller than the difference $0.118 K_{3,n}/\sqrt{n}$.
This bound is also tighter than the sharpest known Berry-Esseen inequality in the \iid{} case, $\DeltaB \leq \ExistingBoundBEiid K_{3,n}/\sqrt{n}$, up to a \(O(n^{-1}\) term.
We recall that the sharpest existing bounds \citep{shevtsova2013} only require a finite third moment while we use further regularity in the form of a finite fourth moment.
We refer to Example~\ref{ex:be_bounds_comparison_nocont} and Figure~\ref{fig:graph_nocont} for a numerical comparison, showing improvements for $n$ of the order of a few thousands.
The most striking improvement is obtained in the unskewed case when ${\E[X_i^3] = 0}$ for every integer~$i$. 
In this case, Theorem~\ref{thm:nocont_choiceEps} and the inequality ${\Ktroisntilde \leq 2 K_{3,n}}$ yield ${\DeltaB \leq \NewBoundBEiid K_{3,n} / \sqrt{n} + O(n^{-1})}$.
Note that this result does not contradict \cite{esseen1956}'s lower bound ${0.4098 K_{3,n} / \sqrt{n}}$ as the distribution he constructs does not satisfy $\E[X_i^3] = 0$ for every~$i$.

Under Assumption~\ref{hyp:basic_as_iid}, $\Ktroisntilde \leq K_{3,n}+1$ and we can combine this with~\eqref{eq:be_bound} and the inequality $|\lambda_{3,n}| \leq 0.621K_{3,n}$, so that we obtain
\begin{align*}
    \DeltaB & \leq \frac{0.1995 (K_{3,n} + 1)
    + 0.0665 \times 0.621 K_{3,n}}{\sqrt{n}}
    + O(n^{-1}) \\
    & \leq \frac{0.2408 K_{3,n} + 0.1995}{\sqrt{n}}
    + O(n^{-1}).
\end{align*}
As in the \inid{} case discussed above, the numerical constant in front of $K_{3,n}$ in the leading term is smaller than the lower bound constant ${C_B:=0.4098}$ derived in \cite{esseen1956}.
The point is addressed in detail in \cite{shevtsova2012}, where the author explains that the constant coming from \cite{esseen1956} cannot be improved only if one seeks control of $\DeltaB$ with a leading term of the form
$c_1 K_{3,n}/\sqrt{n}$ for some $c_1 > 0$.
In contrast, our bound on $\DeltaB$ exhibits a leading term of the form $(c_1 K_{3,n} + c_2)/\sqrt{n}$ for positive constants~$c_1$ and~$c_2$.

\begin{example}[Implementation of our bounds on $\DeltaB$]
\label{ex:be_bounds_comparison_nocont}
Theorem~\ref{thm:nocont_choiceEps} provides new tools to control $\DeltaB$, and we compare them with existing results.
To compute our bounds, we need numerical values for $\Ktroisntilde$, $\lambdatroisn$, and $\Kquatren$ or upper bounds thereon.
{\color{black}As discussed in Section~\ref{subsec:discussion_bounds}, controlling \(\Kquatren\) is in fact sufficient to bound~\(\DeltaE\) and~\(\DeltaB\). 
In that section, we also explain that the choice \(\Kquatren \leq 9\) is reasonable in practice as it covers a wide range of commonly encountered distributions.
Consequently, we stick to this value in our numerical examples.}

The different bounds{\color{black}, without or with the assumption of an unskewed distribution (\(\lambdatroisn = 0\)),} are plotted as a function of~\(n\) in Figure~\ref{fig:graph_nocont}:
\begin{itemize}
    
    \item \cite{shevtsova2013} \inid{}: $\frac{\ExistingBoundBEinid }{\sqrt{n}} \Ktroisn$
    
    \item \cite{shevtsova2013} \iid{}: $\frac{\ExistingBoundBEiid }{\sqrt{n}} \Ktroisn$
    
    \item Thm.~\ref{thm:nocont_choiceEps} \inid{}: $\frac{\NewBoundBEinid }{\sqrt{n}} \Ktroisn + r_{1,n}$
    
    \item Thm.~\ref{thm:nocont_choiceEps} \inid{} (unskewed): $\frac{\NewBoundBEiid }{\sqrt{n}} \Ktroisn + r_{1,n}$
    
    \item Thm.~\ref{thm:nocont_choiceEps} \iid{}:
    $\frac{0.2408 K_{3,n} + 0.1995}{\sqrt{n}} + r_{1,n}$

    \item Thm.~\ref{thm:nocont_choiceEps} \iid{} (unskewed): $\frac{0.1995 (\Ktroisn + 1)}{\sqrt{n}} + r_{1,n}$,
    
\end{itemize}
{\color{black}where the explicit expressions of~\(r_{1,n}\), according to the set-up, are given in Equations~\eqref{eq:def_r_1n_inid_eps_eq_01}, \eqref{eq:def_r_1n_inid_eps_eq_01_noskew}, \eqref{eq:def_r_1n_iid_eps_eq_01}, and~\eqref{eq:def_r_1n_iid_eps_eq_01_noskew}.}

\begin{figure}[t]
    \centering
    \includegraphics[width=0.92\textwidth]{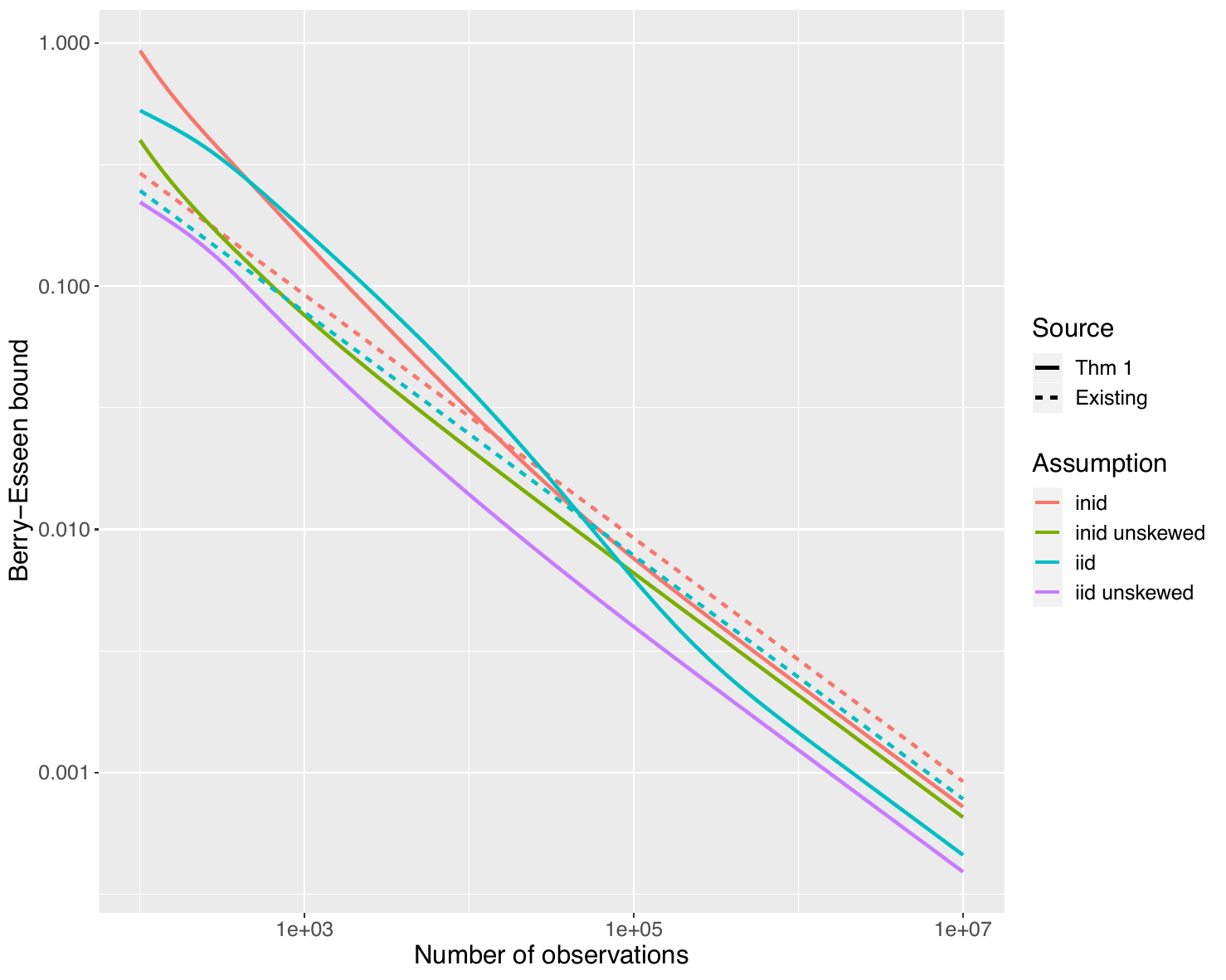}
    \caption{{\small Comparison between existing \citep{shevtsova2013} and new (Theorem~\ref{thm:nocont_choiceEps}) Berry-Esseen upper bounds on $\Delta_{n,\text{B}} := \sup_{x \in \Rb}\left| \Prob(S_n \leq x) - \Phi(x) \right|$ for different sample sizes under moment conditions only (log-log scale).
    {\color{black}As remarked by a reviewer, we note that the improvement we obtain should not come as a surprise since our results require boundedness of \(4\)th order moments while \cite{shevtsova2013}'s bounds remain valid under boundedness of $3$rd order moments only. 
    In that respect, the comparison is somewhat unfair.}
    }}
    \label{fig:graph_nocont}
\end{figure}
As previously mentioned, our bound in the baseline \inid{} case gets close to and even improves upon the best known Berry-Esseen bound in the \iid{} setup \citep{shevtsova2013} for $n$ of the order of tens of thousands. 
When ${\lambdatroisn = 0}$, our bounds are smaller, highlighting improvements of the Berry-Esseen bounds for unskewed distributions.
In parallel, the bounds are also reduced in the \iid{} framework.
\end{example}

\section{Improved bounds on \texorpdfstring{$\boldsymbol{\DeltaE}$}{Delta(n,E)} under assumptions on the tail behavior of \texorpdfstring{$\boldsymbol{\caracfsum}$}{fSn}}
\label{ssec:thms_cont}

In this section, we derive tighter bounds on~$\DeltaE$ under additional regularity conditions on the tail behavior of the characteristic function of~$S_n$.
They follow from Theorem~\ref{thm:cont_choiceEps}, which provides an alternative upper bound on~$\DeltaE$ that involves the tail behavior of $\caracfsum$.
To state this theorem, let us introduce the terms $\rninidskew{2}$, $\rninidnoskew{2}$, $\rniidskew{2}$ and $\rniidnoskew{2}$
{\color{black}
\begin{align}
    \rninidskew{2} & := \frac{1.2533 \, \Ktroisntilde^4}{16\pi^4n^2}
    + \frac{0.3334 \, \Ktroisntilde^4 \, |\lambda_{3,n}|}{16\pi^4n^{5/2}}
    + \frac{14.1961 \, \Ktroisntilde^{16}}{(2\pi)^{16}n^8}
    + \frac{4.3394 \, |\lambdatroisn| \, \Ktroisntilde^{12}}{(2\pi)^{12} n^{13/2}} \nonumber \\
    & + \frac{|\lambdatroisn|
    \big( \Gamma( 3/2 , \sqrt{0.2} (n/K_{4,n})^{1/4} \wedge 16\pi^3n^2/\Ktroisntilde^4)
    - \Gamma( 3/2 , 16\pi^3n^2/\Ktroisntilde^4) \big)}{\sqrt{n}} \nonumber \\
    & + \frac{1.0435 K_{4,n}^{5/4}}{n^{5/4}} + 
    \frac{1.1101 \Kquatren^{3/2} + 8.2383 |\lambdatroisn| \times \Kquatren}{n^{3/2}} +
    \frac{0.6087 \Kquatren^{7/4}}{n^{7/4}} \nonumber \\
    & +
    \frac{9.8197\Kquatren^{2}}{n^2} \nonumber \\
    & + \frac{1.0253\Ktroisn}{6\pi\sqrt{n}} \Bigg\{ 0.5|\Delta|^{-3/2}\Indicator_{\{\Delta \neq 0\}} \times \bigg|\gamma(3/2, 2^8\pi^6 \Delta n^4 / \Ktroisntilde^8) \nonumber \\
    & \qquad\qquad\qquad\;\; - \gamma\big(3/2, \Delta ( 0.2 (n/K_{4,n})^{1/2} \wedge 2^8\pi^6 n^4 / \Ktroisntilde^8 ) \big) \bigg| \nonumber \\
    & \qquad\qquad\qquad + \Indicator_{\{\Delta = 0\}}
    \frac{(16 \pi^3 n^2 / \Ktroisntilde^4)^3
    - (\sqrt{0.2} (n/K_{4,n})^{1/4} \wedge 16 \pi^3 n^2 / \Ktroisntilde^4)^3}{3}
    \Bigg\} \nonumber \\  
    & + \frac{1.0253}{\pi} \left( \Gamma\left( 0 , (4\pi^2n/\Ktroisntilde^2 \wedge 144\pi^8n^4/\Ktroisntilde^8)(1-4\pi\chi_1t_1^*) / (2\pi^2) \right) \right. \nonumber \\
    & \left. \qquad\qquad\quad - \Gamma\left( 0 , (4t_1^{*2}\pi^2n/\Ktroisntilde^2 \wedge 144\pi^6n^4/\Ktroisntilde^8)(1-4\pi\chi_1t_1^*) / 2 \right) \right) \nonumber \\
    & + \frac{1.0253}{\pi} \left( \Gamma\left( 0 , (4\pi^2n/\Ktroisntilde^2 \wedge 144\pi^8n^4/\Ktroisntilde^8) / (2\pi^2) \right) - \Gamma\left( 0 , 144\pi^6n^4/\Ktroisntilde^8 \right) \right),
    \label{eq:def_r_2n_inid_eps_eq_01}  
\end{align}
\begin{align}
    \rninidnoskew{2} & := \frac{1.2533 \, \Ktroisntilde^4}{16\pi^4n^2}
    + \frac{14.1961 \, \Ktroisntilde^{16}}{(2\pi)^{16}n^8}
    \nonumber \\
    & + \frac{1.0253\Kquatren}{6\pi n} \Bigg\{ 0.5|\Delta|^{-2}\Indicator_{\{\Delta \neq 0\}} \times \bigg|\gamma(2, 2^8\pi^6 \Delta n^4 / \Ktroisntilde^8) \nonumber \\
    & \qquad\qquad\qquad\;\; - \gamma\big(2, \Delta ( 0.2 (n/K_{4,n})^{1/2} \wedge 2^8\pi^6 n^4 / \Ktroisntilde^8 ) \big) \bigg| \nonumber \\
    & \qquad\qquad\quad + \Indicator_{\{\Delta = 0\}}
    \frac{(16 \pi^3 n^2 / \Ktroisntilde^4)^4 
    - (\sqrt{0.2} (n/K_{4,n})^{1/4} \wedge 16 \pi^3 n^2 / \Ktroisntilde^4)^4}{4}
    \Bigg\} \nonumber \\  
    & + \frac{1.0253}{\pi} \left( \Gamma\left( 0 , (4\pi^2n/\Ktroisntilde^2 \wedge 144\pi^8n^4/\Ktroisntilde^8)(1-4\pi\chi_1t_1^*) / (2\pi^2) \right) \right. \nonumber \\
    & \left. \qquad\qquad\quad - \Gamma\left( 0 , (4t_1^{*2}\pi^2n/\Ktroisntilde^2 \wedge 144\pi^6n^4/\Ktroisntilde^8)(1-4\pi\chi_1t_1^*) / 2 \right) \right) \nonumber \\
    & + \frac{1.0253}{\pi} \left( \Gamma\left( 0 , (4\pi^2n/\Ktroisntilde^2 \wedge 144\pi^8n^4/\Ktroisntilde^8) / (2\pi^2) \right) \right. \nonumber \\
    & \left. \qquad\qquad\quad - \Gamma\left( 0 , 144\pi^6n^4/\Ktroisntilde^8 \right) \right).
    \label{eq:def_r_2n_inid_eps_eq_01_noskew}    
\end{align}
\begin{align}
    \rniidskew{2} & := \frac{1.2533 \, \Ktroisntilde^4}{16\pi^4n^2}
    + \frac{0.3334 \, \Ktroisntilde^4 \, |\lambda_{3,n}|}{16\pi^4n^{5/2}}
    + \frac{14.1961 \, \Ktroisntilde^{16}}{(2\pi)^{16}n^8}
    + \frac{4.3394 \, |\lambdatroisn| \, \Ktroisntilde^{12}}{(2\pi)^{12} n^{13/2}} \nonumber \\
    & + \frac{|\lambdatroisn|
    \big( \Gamma( 3/2 , \sqrt{0.2} (n/K_{4,n})^{1/4} \wedge 16\pi^3n^2/\Ktroisntilde^4)
    - \Gamma( 3/2 , 16\pi^3n^2/\Ktroisntilde^4) \big)}{\sqrt{n}} \nonumber \\
    & + \RiidIntSkew \nonumber \\
    & + \frac{1.0253 \times 2^{5/2} \, \Ktroisn}{3\pi\sqrt{n}} \big| \Gamma(3/2, 2^5\pi^6n^4/\Ktroisntilde^8) \nonumber \\
    &\qquad\qquad\qquad\qquad\quad\; - \Gamma(3/2, 0.1\sqrt{n/(16\Kquatren)} \wedge 2^5\pi^6n^4/\Ktroisntilde^8) \big| \nonumber \\
    & + \frac{1.306 \big( e_{2,n}(0.1) - e_3(0.1) \big) \lambdatroisn^2}{36 n} \nonumber \\
    & + \frac{1.0253}{\pi} \left( \Gamma\left( 0 , (4\pi^2n/\Ktroisntilde^2 \wedge 144\pi^8n^4/\Ktroisntilde^8)(1-4\pi\chi_1t_1^*) / (2\pi^2) \right) \right. \nonumber \\
    & \left. \qquad\qquad\quad - \Gamma\left( 0 , (4t_1^{*2}\pi^2n/\Ktroisntilde^2 \wedge 144\pi^6n^4/\Ktroisntilde^8)(1-4\pi\chi_1t_1^*) / 2 \right) \right) \nonumber \\
    & + \frac{1.0253}{\pi} \left( \Gamma\left( 0 , (4\pi^2n/\Ktroisntilde^2 \wedge 144\pi^8n^4/\Ktroisntilde^8) / (2\pi^2) \right) - \Gamma\left( 0 , 144\pi^6n^4/\Ktroisntilde^8 \right) \right),
    \label{eq:def_r_2n_iid_eps_eq_01}    
\end{align}
and
\begin{align}
    \rniidnoskew{2}
    & := \frac{1.2533 \, \Ktroisntilde^4}{16\pi^4n^2}
    + \frac{14.1961 \, \Ktroisntilde^{16}}{(2\pi)^{16}n^8} + \RiidIntNoskew \nonumber \\
    & + \frac{16 \times 1.0253 \, \Ktroisn
        \big| \Gamma(2, 2^5\pi^6n^4/\Ktroisntilde^8) - \Gamma(2, 0.1\sqrt{n/(16\Kquatren)} \wedge 2^5\pi^6n^4/\Ktroisntilde^8) \big|}{3\pi n} \nonumber \\  
    & + \frac{1.0253}{\pi} \left( \Gamma\left( 0 , (4\pi^2n/\Ktroisntilde^2 \wedge 144\pi^8n^4/\Ktroisntilde^8)(1-4\pi\chi_1t_1^*) / (2\pi^2) \right) \right. \nonumber \\
    & \left. \qquad\qquad\quad - \Gamma\left( 0 , (4t_1^{*2}\pi^2n/\Ktroisntilde^2 \wedge 144\pi^6n^4/\Ktroisntilde^8)(1-4\pi\chi_1t_1^*) / 2 \right) \right) \nonumber \\
    & + \frac{1.0253}{\pi} \left( \Gamma\left( 0 , (4\pi^2n/\Ktroisntilde^2 \wedge 144\pi^8n^4/\Ktroisntilde^8) / (2\pi^2) \right) \right. \nonumber \\
    & \left. \qquad\qquad\quad - \Gamma\left( 0 , 144\pi^6n^4/\Ktroisntilde^8 \right) \right).
    \label{eq:def_r_2n_iid_eps_eq_01_noskew}
\end{align}
}

Recall also that $t_1^* \approx 0.64$ and let
$a_n := 2t_1^*\pi\sqrt{n}/\Ktroisntilde \wedge 16\pi^3n^2/\Ktroisntilde^4$ and
$b_n := 16\pi^4n^2/\Ktroisntilde^4$. In practice, even for fairly small~\(n\), \(a_n\) is equal to \(2t_1^*\pi\sqrt{n}/\Ktroisntilde\).

\begin{theorem}
    \noindent
    If Assumption~\ref{hyp:basic_as_inid} (\textit{resp.} Assumption~\ref{hyp:basic_as_iid}) holds and $n \geq 3$, we have the bound
    \begin{align}
        \label{eq:generall_bound_cont}
        \DeltaE \leq \frac{0.195 \, K_{4,n}
        + 0.038 \, \lambda_{3,n}^2}{n}
        + \frac{1.0253}{\pi} \int_{a_n}^{b_n} \frac{|\caracfsum(t)|}{t}dt
        + r_{2,n} \, ,
    \end{align}
    where $r_{2,n}$ is one of the four possible remainders $\rninidskew{2}$, $\rninidnoskew{2}$, $\rniidskew{2}$ or $\rniidnoskew{2}$, depending on whether Assumption~\ref{hyp:basic_as_inid} (``\inid{}'' case) or~\ref{hyp:basic_as_iid} (``\iid{}'' case) is satisfied and whether $\E[X_i^3]=0$ for every $i = 1,\dots,n$ (``noskew'' case) or not (``skew'' case).
\label{thm:cont_choiceEps}
\end{theorem}

{\color{black}
\begin{remark}
    Assume that there exists a constant $K_4$ such that $K_{4,n} \leq K_4$ for all $n \geq 0$ (this is the case, for example, if the data is an \iid~sample from a given infinite homogeneous population).
    Then the remainder terms can be bounded in the following way:
    $|\rninidskew{2}| = O(n^{-5/4})$, 
    $|\rninidnoskew{2}| = O(n^{-3/2})$,
    $|\rniidskew{2}| = O(n^{-5/4})$, and
    $|\rniidnoskew{2}| = O(n^{-2})$, for every $n \geq 3$.
\end{remark}

\medskip
}

This theorem is proved in Section~\ref{ssec:proof:inid_cont} under Assumption~\ref{hyp:basic_as_inid} (\resp{} in Section~\ref{ssec:proof:iid_cont} under Assumption~\ref{hyp:basic_as_iid}). 
The first term contains quantities that were already present in the term of order~$1/n$ in the bound of Theorem~\ref{thm:nocont_choiceEps}: $0.195\Kquatren$ and $0.038 \lambdatroisn^2$.
On the contrary, the other terms are encompassed in the integral term and in the remainder.
Indeed, a careful reading of the proofs (see notably Section~\ref{ssec:proof:outline} that outlines the structure of the proofs of all theorems) shows that the leading term $0.1995 \, \widetilde{K}_{3,n} / \sqrt{n}$ in the bound~\eqref{eq:generall_bound_nocont} comes from choosing a free tuning parameter~$T$ of the order of~$\sqrt{n}$.
Here, we make another choice for~$T$ such that this term is now negligible.
The cost of this change of $T$ is the introduction of the integral term involving $\caracfsum$.
The leading term of the bound thus depends on the tail behavior of~$\caracfsum$.

Note that the result is obtained under the same conditions as Theorem~\ref{thm:nocont_choiceEps}, namely under moment conditions only.
Nonetheless, it is mainly interesting combined with some assumptions on \(\caracfsum\) over the interval \([a_n, b_n]\), otherwise we do not have an explicit control on the integral term involving~\(\caracfsum\).
In the rest of this section, we present two possible assumptions on \(\caracfsum\) that yield such a control.

\subsection{Polynomial tail decay on \texorpdfstring{\(\boldsymbol{|\caracfsum|}\)}{|fSn|}}

As a first regularity condition on \(\caracfsum\), we can assume a polynomial rate decrease.
Corollary~\ref{cor:improvement_inid_case} presents the resulting bound in the \inid{} case.
In fact, a similar condition could be invoked with \iid{} data by requesting a polynomial decrease of the characteristic function of \(X_n / \sigma_n\).
However, we present in the next paragraph milder assumptions in the \iid{} case that remain sufficient to obtain an explicit control of the tails of $\caracfsum$.

\begin{corollary}
\label{cor:improvement_inid_case}
    Let $n \geq 3$. If Assumption~\ref{hyp:basic_as_inid} holds and if there exist some positive constants \(C_0, p\) such that for all~$|t| \geq a_n$, $|\caracfsum(t)| \leq C_0 |t|^{-p}$, then
    \begin{align*}
        \DeltaE \leq \frac{0.195 \, K_{4,n}
        + 0.038 \, \lambda_{3,n}^2}{n}
        + \frac{1.0253 \, C_0 a_n^{-p}}{\pi}
        + r_{3,n}
    \end{align*}
    where $r_{3,n} := r_{2,n} - 1.0253 \, C_0 b_n^{-p}/\pi$.
\end{corollary}

Besides moment conditions, Corollary~\ref{cor:improvement_inid_case} requires a uniform control of $\caracfsum$ outside the interval $(-a_n,a_n)$.
When $\Ktroisntilde = o(\sqrt{n})$, $a_n$ goes to infinity. In this case, the condition is a tail control of the characteristic function of~$S_n$ in a neighborhood of infinity, thus making the condition weaker to impose. 

Placing restrictions on the tails of $\caracfsum$ is not very common in statistical applications. However, this notion is closely related to the smoothness of the underlying distribution of $S_n$. Proposition~\ref{prop:charac_bound_differentiable} in the Appendix (which builds upon classical results such as~\cite[Theorem 1.2.6]{ushakov2011}) shows that the tail condition on $\caracfsum$ is satisfied with $p \geq 1$ whenever $P_{S_n}$ has a density $g_{S_n}$ that is $p-1$ times differentiable and such that its $(p-1)$-th derivative is of bounded variation with total variation $V_n := \mathrm{Vari}[g_{S_n}^{(p-1)}]$ uniformly bounded in $n$.
In such situations, we can take $C_0 = 1 \vee \sup_{n \in \Nstar}V_n$.

Although Corollary~\ref{cor:improvement_inid_case} is valid for every positive~$p$, it is only an improvement on the results of the previous section under the stricter condition~${p > 1}$, a situation in which $P_{S_n}$ admits a density with respect to Lebesgue's measure (second part of Proposition~\ref{prop:charac_bound_differentiable}).
In particular when $p=2$, $a_n^{-p}$ is exactly {\color{black}proportional to} $n^{-1}$ and we obtain
\begin{equation*}
    \DeltaE \leq \frac{0.195 \, K_{4,n}
    + 0.038 \, \lambda_{3,n}^2 + 1.0253 \, C_0 \pi^{-1} }{n}
    + O(n^{-5/4}),
\end{equation*}
for every $n \geq 3$.
When ${p > 2}$, $a_n^{-p}$ becomes negligible compared to $n^{-5/4}$ so that
\begin{equation*}
    \DeltaE \leq \frac{0.195 \, K_{4,n}
    + 0.038 \, \lambda_{3,n}^2}{n}
    + O(n^{-5/4}).
\end{equation*}
Combining these bounds on $\DeltaE$ with the expression of the Edgeworth expansion translates into upper bounds on $\DeltaB$ of the form
\begin{equation*}
    \DeltaB \leq \frac{0.0665 \, |\lambda_{3,n}|}{\sqrt{n}} 
    + O(n^{-1})
    \leq \frac{0.0413 \, \Ktroisn}{\sqrt{n}}
    + O(n^{-1}).
\end{equation*}
As soon as the previous \(O(n^{-1})\) term gets smaller than $0.0413 \Ktroisn/\sqrt{n}$, the bound on $\DeltaB$ becomes much better than $\ExistingBoundBEinid \Ktroisn/\sqrt{n}$ or $\ExistingBoundBEiid \Ktroisn/\sqrt{n}$.
This can happen even for sample sizes $n$ of the order of a few thousands, assuming that $\Ktroisn$ and $\Kquatren$ are reasonable (\textit{e.g.} $\Kquatren \leq 9$).
When $\E[X_i^3]=0$ for every $i = 1,\dots,n$, we remark that $\DeltaB = \DeltaE$, meaning that we obtain a bound on $\DeltaB$ of order $n^{-1}$.

We confirm these rates through a numerical application in Example~\ref{ex:be_bounds_comparison_cont} for the specific choices $C_0 = 1$ and $p = 2$. 
These choices are satisfied for common distributions such as the Laplace distribution (for which these values of $C_0$ and $p$ are sharp) and the Gaussian distribution. This actually opens the way for another restriction on the tails of $\caracfsum$: we could impose $|\caracfsum(t)| \leq \max_{1 \leq r \leq M}|\rho_r(t)|$ for all~$|t| \geq a_n$ and for $(\rho_r)_{r=1, \dots, M}$ a family of known characteristic functions.
This second suggestion boils down to a semiparametric assumption on $P_{S_n}$: $\caracfsum$ is assumed to be controlled in a neighborhood of $\pm\infty$ by the behavior of at least one of the $M$ characteristic functions $(\rho_r)_{r=1, \dots, M},$ but $\caracfsum$ need not be exactly one of those $M$ characteristic functions.
This semiparametric restriction becomes less and less stringent as $n$ increases since we need to control $\caracfsum$ on a region that vanishes as $n$ goes to infinity.
Since $S_n$ is centered and of variance~$1$ by definition, the choice of possible $\rho_r$ is naturally restricted to the set of characteristic functions that correspond to such standardized distributions.

\subsection{Alternative control of \texorpdfstring{\(\boldsymbol{|\caracfsum|}\)}{|fSn|} in the \iid{} case}

We state a second corollary that deals with the \iid{} framework. 
We define the following quantity
$\kappa_n := 
\sup_{t: \, |t| \geq a_n/\sqrt{n}} |f_{X_n/\sigma_n}(t)|$ 
and let $c_n := b_n/a_n$.
Under Assumption~\ref{hyp:basic_as_iid}, we remark that
$\sup_{t: \, |t|\geq a_n} |\caracfsum(t)| = \kappa_n^n.$

\begin{corollary}
\label{cor:improvement_iid_case}
    Let \(n \geq 3\).
    Under Assumption~\ref{hyp:basic_as_iid},
    \begin{align*}
        \DeltaE \leq \frac{0.195 \, K_{4,n}
        + 0.038 \, \lambda_{3,n}^2}{n}
        + \frac{1.0253 \, \kappa_n^n\log(c_n)}{\pi}
        + r_{2,n} \,. 
    \end{align*}
    Furthermore, $\kappa_n < 1$ as soon as $P_{X_n/\sigma_n}$ has an absolutely continuous component.
\end{corollary}

Note that for any given $s > 0$ and any random variable $Z$, $\sup_{t: |t| \geq s} |f_Z(t)| = 1$ if and only if $P_{Z}$ is a lattice distribution, i.e., concentrated on a set of the form $\{a + nh, n \in \Zb \}$ \cite[Theorem 1.1.3]{ushakov2011}. Therefore, $\kappa_n < 1$ as soon as the distribution is not lattice, which is the case for any distribution with an absolute continuous component.

In Corollary~\ref{cor:improvement_iid_case}, the first term on the right-hand side of the inequality as well as $r_{2,n}$ are unchanged compared to Theorem~\ref{thm:cont_choiceEps} and Corollary~\ref{cor:improvement_inid_case}.
The second term on the right-hand side of the inequality, $(1.0253/\pi)\kappa_n^n\log(c_n),$ corresponds to an upper bound on the integral term of Equation~\eqref{eq:generall_bound_cont} in Theorem~\ref{thm:cont_choiceEps}.
Imposing $K_{4,n} \leq K_4$, we can only claim that $1.0253 \, \kappa_n^n\log(c_n)/\pi
= O(\kappa_n^n\log n)$, which does not provide an explicit rate on $\DeltaE$.
If we also assume $\sup_{n\geq 3}\kappa_n<1$ then we can write
\begin{equation*}
    \DeltaE \leq \frac{0.195 \, K_{4,n}
    + 0.038 \, \lambda_{3,n}^2}{n}
    + O(n^{-5/4}),
\end{equation*}
and
\begin{equation*}
    \DeltaB
    \leq \frac{0.0665 \, |\lambda_{3,n}|}{\sqrt{n}}
    + O(n^{-1})
    \leq \frac{0.0413 \, \Ktroisn}{\sqrt{n}}
    + O(n^{-1}).
\end{equation*}

\medskip

When is the assumption
$\sup_{n\geq 3}\kappa_n < 1$ reasonable? 
First, it always holds in the \iid{} setting with a distribution of the \((X_i)_{i = 1, \ldots, n}\) independent of~$n$ and continuous.
By definition of $a_n$ and by the fact that $\Ktroisntilde \geq 1$, $a_n/\sqrt{n}$ is larger than $2 t_1^* \pi$ for $n$ large enough. 
Consequently, $\kappa_n$ is upper bounded by
$\kappa := \sup_{t: \, |t| \geq 2t_1^*\pi} |f_{X_1/\sigma_1}(t)|$
for $n$ large enough.
In this case, if $P_{X_1/\sigma_1}$ has an absolutely continuous component, \(\kappa < 1\).
For smaller~\(n\), we use the fact that \(\kappa_n < 1\) for every~\(n\) as explained right after Corollary~\ref{cor:improvement_iid_case}.
The value of $\kappa$ depends on the distribution $P_{X_1/\sigma_1}$.
The closer to one $\kappa$ gets, the less regular $P_{X_1/\sigma_1}$ is, in the sense that the latter becomes hardly distinguishable from a lattice distribution.

Second, we could impose that the characteristic function $f_{X_n/\sigma_n}$ be controlled by some finite family of known characteristic functions $\rho_1, \dots, \rho_M$ (independent of $n$) beyond $a_n / \sqrt{n}$.
This follows the suggestion mentioned after Corollary~\ref{cor:improvement_inid_case}, except that we now obtain an exponential upper bound instead of a polynomial one.
Indeed, for $n$ large enough, $\kappa_n \leq \kappa := \sup_{t:|t|\geq 2t_1^*\pi}\max_{1\leq m \leq M}|\rho_m(t)|$ and $\kappa < 1$ provided that $(\rho_m)_{m=1,\dots,M}$ are characteristic functions of continuous distributions.

In Example~\ref{ex:be_bounds_comparison_cont}, we plot our bounds on $\DeltaB$ by imposing the restriction $\kappa_n \leq 0.99$ which we argue is a very reasonable choice.
To justify this claim, we compare our restriction to the value of $\kappa_n$ we would get if $X_n/\sigma_n$ were standard Laplace, a distribution whose characteristic function has much fatter tails than the standard Gaussian or Logistic for instance.
In fact, if we were to compute $\sup_{t:|t|\geq 2t_1^*\pi}|\rho(t)|$ with $\rho$ the characteristic function of a standard Laplace distribution, we would get $\kappa_n < 0.11$.
Despite our fairly conservative bound on $\kappa_n$, we witness considerable improvements of our bounds compared to those given in Section~\ref{ssec:thms_nocont}.

\begin{example}[Implementation of our bounds on $\DeltaB$]
\label{ex:be_bounds_comparison_cont}
We compare the bounds on $\DeltaB$ obtained in Corollaries~\ref{cor:improvement_inid_case} and~\ref{cor:improvement_iid_case} to $\ExistingBoundBEinid K_{3,n}/\sqrt{n}$ and $\ExistingBoundBEiid K_{3,n}/\sqrt{n}.$ 
As in Example~\ref{ex:be_bounds_comparison_nocont}, we fix $K_{4,n} \leq 9$, which is enough to control~\(\Ktroisn\) (see Section~\ref{subsec:discussion_bounds}).
As explained above, we set \(p = 2\) and \(C_0 = 1\) to apply Corollary~\ref{cor:improvement_inid_case} and \(\kappa = 0.99\) for Corollary~\ref{cor:improvement_iid_case}.
\begin{itemize}

    \item Cor.~\ref{cor:improvement_inid_case} \inid{}:
    $\DeltaB \leq \frac{0.0413 \, \Ktroisn}{\sqrt{n}}
    + \frac{0.195 \, K_{4,n}
    + 0.0147 \, \Ktroisn^2}{n}
    + \frac{1.0253}{\pi} a_n^{-2} + r_{3,n}$
    
    \item Cor.~\ref{cor:improvement_inid_case} \inid{} unskewed:
    $\DeltaB \leq \frac{0.195 \, \Kquatren}{n}
    + \frac{1.0253}{\pi} a_n^{-2} + r_{3,n}$
    
    \item Cor.~\ref{cor:improvement_iid_case} \iid{}:
    $\DeltaB \leq \frac{0.0413 \, \Ktroisn}{\sqrt{n}}
    + \frac{0.195 \, K_{4,n} + 
    + 0.0147 \, \Ktroisn^2}{n}
    + \frac{1.0253 \, \kappa_n^n\log(c_n)}{\pi}
    + r_{2,n}$
    
    \item Cor.~\ref{cor:improvement_iid_case} \iid{}
    unskewed:
    $\DeltaB \leq \frac{0.195 \, K_{4,n}}{n}
    + \frac{1.0253 \, \kappa_n^n\log(c_n)}{\pi}
    + r_{2,n}$
    
\end{itemize}
Figure~\ref{fig:graph_cont} displays the different bounds that we obtain as a function of the sample size $n$, alongside with the existing bounds \citep{shevtsova2013} that do not assume such regularity conditions.
The new bounds take advantage of these regularity conditions and are therefore tighter in all settings for \(n\) larger than~\(10,000\).
In the unskewed case, the improvement arises for much smaller~\(n\) and the rate of convergence gets faster from~$1/\sqrt{n}$ to~$1/n$.
\end{example}

\begin{figure}[H]
    \centering
    \includegraphics[width=0.92\textwidth]{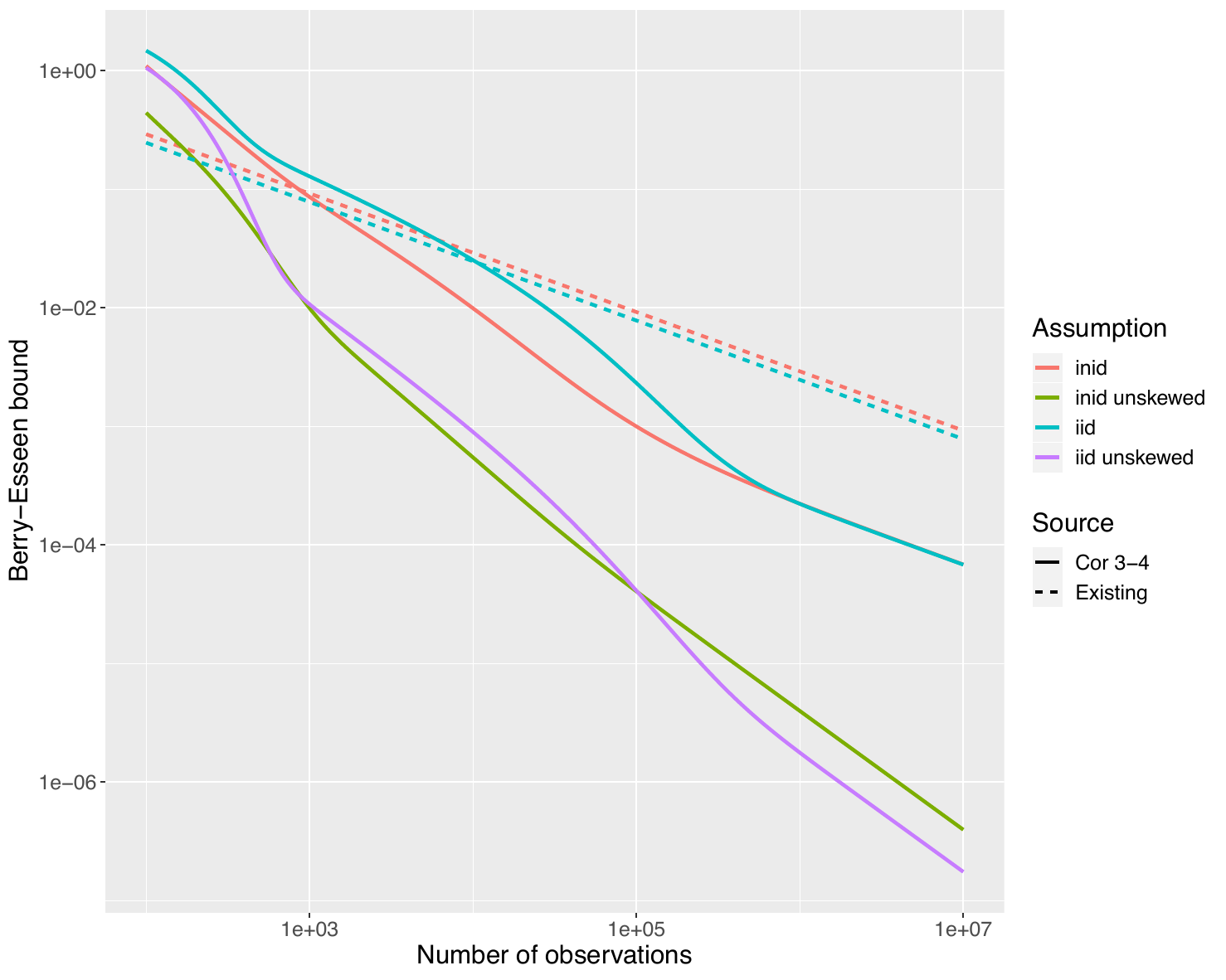}
    \caption{{\small Comparison between existing \citep{shevtsova2013} and new (Corollaries~\ref{cor:improvement_inid_case} and~\ref{cor:improvement_iid_case}) Berry-Esseen upper bounds on $\Delta_{n,\text{B}} := \sup_{x \in \Rb}\left| \Prob(S_n \leq x) - \Phi(x) \right|$ for different sample sizes with additional regularity assumption on~$\caracfsum$ (log-log scale).
    {\color{black}Note that, compared to existing ones, the new bounds make use of the regularity assumption and of the boundedness of the 4th order moments.}}}
    \label{fig:graph_cont}
\end{figure}

\section{Practical considerations}
\label{sec:practical_considerations}

\subsection{Default value \texorpdfstring{$\Kquatren \leq 9$}{K4n = 9} or ``Plug-in'' approach}
\label{subsec:discussion_bounds}

{\color{black}
As seen in the previous examples, explicit values or bounds on some functionals of $P_{S_n}$ are required to compute our non-asymptotic bounds on a standardized sample mean.
This phenomenon is not unique to our bounds, and arises for any Berry-Esseen- or Edgeworth-type bounds. 
A value or a bound on~\(\Ktroisn\) is indeed required to compute existing Berry-Esseen bounds as in the seminal works of \cite{berry1941} and~\cite{esseen1942} and its recent improvement (\textit{e.g.} \cite{shevtsova2013}).
Similar to us, recent extensions to these bounds proposed in \cite{adell2008}, \cite{boutsikas2011} and~\cite{zhilova2020new} also depend on several (potentially unknown) moments of the distributions.

\medskip 

Under moment conditions only, the main term and remainder~\(r_{1,n}\) of Theorem~\ref{thm:nocont_choiceEps} solely depend on $\lambdatroisn$, \(\Ktroisn\) or \(\Ktroisntilde\), and \(\Kquatren\).
As a matter of fact, a bound on \(\Kquatren\) is sufficient to control all those quantities:
\cite{pinelis2011relations} ensures $|\lambdatroisn| \leq 0.621 \Ktroisn$, and a convexity argument yields $\Ktroisn \leq K_{4,n}^{3/4}$ (and remember that \(\Ktroisntilde\) is lower than \(2 \Ktroisn\) in the \inid{} case and \(\Ktroisn + 1\) in the \iid{} case). 
Having access to a bound on~\(\Kquatren\) is thus crucial to compute our bounds in practice.

First, in some situations, one may rely on auxiliary information about the distribution.
In the \iid{} case in particular, we note that imposing the bound ${\Kquatren \leq 9}$ allows for a wide family of distributions used in practice: 
any Gaussian, Gumbel, Laplace, Uniform, or Logistic distribution satisfies it, as well as any Student with at least 5 degrees of freedom, any Gamma or Weibull with shape parameter at least~1.
In this case, remember that \(\Kquatren\) is the kurtosis of~\(X_n\), a natural and well-studied feature of a distribution.

In the \inid{} case, $\Kquatren$ can be rewritten as a weighted average of individual kurtosis.
In that respect, the bound ${\Kquatren \leq 9}$ indicates that, on average, the individual kurtosis are lower than~\(9\).

Second, if a bound on \(\Kquatren\) is not available, a ``plug-in'' approach remains applicable.
The idea is to estimate the moments $\lambdatroisn$, $\Ktroisn$ and $\Kquatren$ by their empirical counterparts in the data (method of moments estimation), and then compute \(\delta_n\) by replacing the unknown needed quantities with those estimates.
We acknowledge that this type of ``plug-in'' approach is only approximately valid, although somewhat unavoidable when bounds on the unknown moments are not given to the researcher.
%

In addition to the dependence on these moment bounds, Theorem~\ref{thm:cont_choiceEps} involves the integral $\int_{a_n}^{b_n} |\caracfsum(t)| / t \, dt$ that depends on the a priori unknown characteristic function of~$S_n$.
%
%
The application of the resulting 
Corollaries~\ref{cor:improvement_inid_case} and~\ref{cor:improvement_iid_case}
requires a control on the tail of this characteristic function through the quantities \(C_0\) and \(p\) in the \inid{} case (respectively \(\kappa_n\) in the \iid{} case), which 
can be given using expert knowledge of the regularity of the density of $S_n$, as discussed in Section~\ref{ssec:thms_cont}.
}
It is also possible to estimate the integral directly, for instance using the empirical characteristic function \cite[Chapter~3]{ushakov2011}.

\subsection{Numerical comparisons of our bounds on \texorpdfstring{\(\Prob(S_n \leq x)\)}{Prob(S\_n <= x)} and existing ones}

{\color{black}

To give a better sense of the accuracy of our results, we perform a comparison between our bounds on \(x \mapsto \Prob(S_n \leq x)\) and the existing ones \citep{shevtsova2013}.
Indeed, a control~\(\delta_n\) on \(\DeltaE\) (respectively \(\DeltaB\)) naturally yields upper and lower brackets on \(\Prob(S_n \leq x)\) of the form  
\(\left[\Phi(x) + \lambdatroisn / (6 \sqrt{n}) \times (1 - x^2) \varphi(x)\right] \pm \delta_n\) (respectively \(\Phi(x) \pm \delta_n\)), for any real~\(x\).
We plot those upper and lower brackets in the \iid{} framework for three distinct distributions: Student distributions with 5 (Figure~\ref{fig:graph_Prob_Sn_student_df5}) or 8 (Figure~\ref{fig:graph_Prob_Sn_student_df8}) degrees of freedom and an Exponential distribution with expectation equal to~1, re-centered to fall in our framework (Figure~\ref{fig:graph_Prob_Sn_gamma}).
These three distributions are continuous with respect to Lebesgue's measure which allows us to resort to our sharpest \iid{} bounds, namely those presented in Corollary~\ref{cor:improvement_iid_case} (compared to Figures~\ref{fig:graph_nocont} and~\ref{fig:graph_cont}, we only report those improved bounds here).
On the contrary, remember that the existing bounds \citep{shevtsova2013} assume finite third-order moments only; hence, they do not leverage the additional information about skewness and regularity of the considered distributions. 

The bound \(\delta_n\) depends on various features of the distribution of \(S_n\).
In line with Example~\ref{ex:be_bounds_comparison_cont}, we set \(\kappa = 0.99\), which happens to be a conservative choice with those distributions as \(\kappa = 0.42\) for a Student(df = 8), 0.54 for a Student(df = 5), and 0.63 for the Exponential distributions we consider.
In the following comparisons, we focus on the impact of the unknown moments \(\Kquatren\), \(\Ktroisn\), and \(\lambdatroisn\) on the accuracy of our bounds. 
}

\begin{figure}[htbp]
    \centering
    \includegraphics[width=1\textwidth]{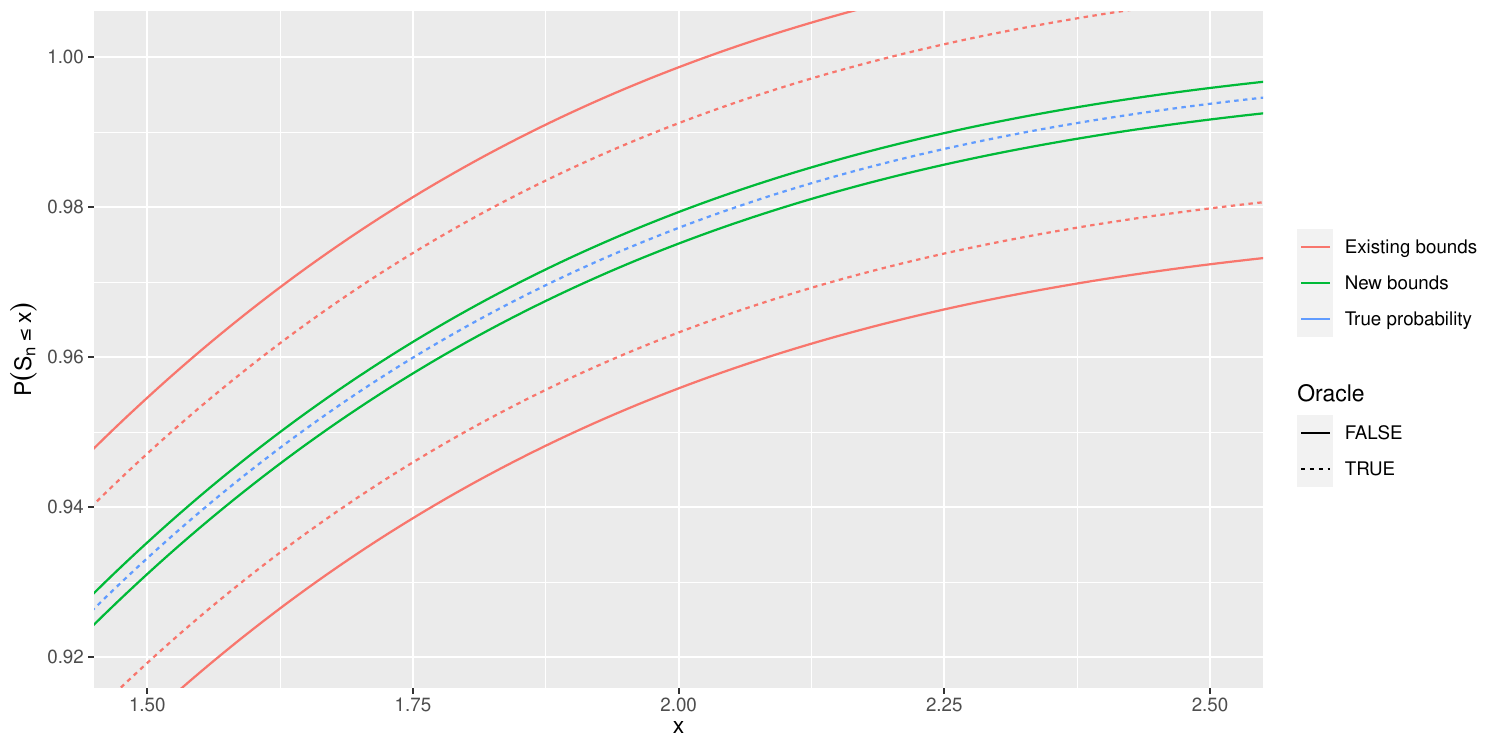}
    \caption{
    Setting: \iid{} unskewed ($\lambdatroisn = 0$)
    with $X_n \sim \textnormal{Student}(\textnormal{df} = 5)$ and $n = 5,\!000$. \\  
    {\footnotesize \textbf{Blue line:} $\Prob(S_n \leq x)$ as a function of $x$.} \\
    {\footnotesize \textbf{Continuous green lines:} bounds
    \(\Phi(x) \pm \delta_n^{\textnormal{new}}\)
    where $\delta_n^{\textnormal{new}}$ denotes the right-hand side of Corollary~\ref{cor:improvement_iid_case} with \(\kappa_n = 0.99\), $\Kquatren \leq 9$, and \(\Ktroisn \leq 9^{3/4}\). }\\
    {\footnotesize \textbf{Dashed green lines:} bounds
    \(\Phi(x) \pm  \delta_n^{\textnormal{new, oracle}}\), where $\delta_n^{\textnormal{new, oracle}}$ denotes the right-hand side of Corollary~\ref{cor:improvement_iid_case} with \(\kappa_n = 0.99\) and using the true (oracle) values of $\Kquatren = 9$ and $\Ktroisn \approx 2.1$.} \\
    {\footnotesize \textbf{Continuous red lines:} bounds $\Phi(x) \pm \ExistingBoundBEiid \Ktroisn / \sqrt{n}$ using the bound \(\Ktroisn \leq 9^{3/4} \approx 5.2\).} \\
    {\footnotesize \textbf{Dashed red lines:} bounds $\Phi(x) \pm \ExistingBoundBEiid \Ktroisn / \sqrt{n}$ using the true value \(\Ktroisn \approx 2.1\).}
    }
    \label{fig:graph_Prob_Sn_student_df5}
\end{figure}

\begin{figure}[htbp]
    \centering
    \includegraphics[width=1\textwidth]{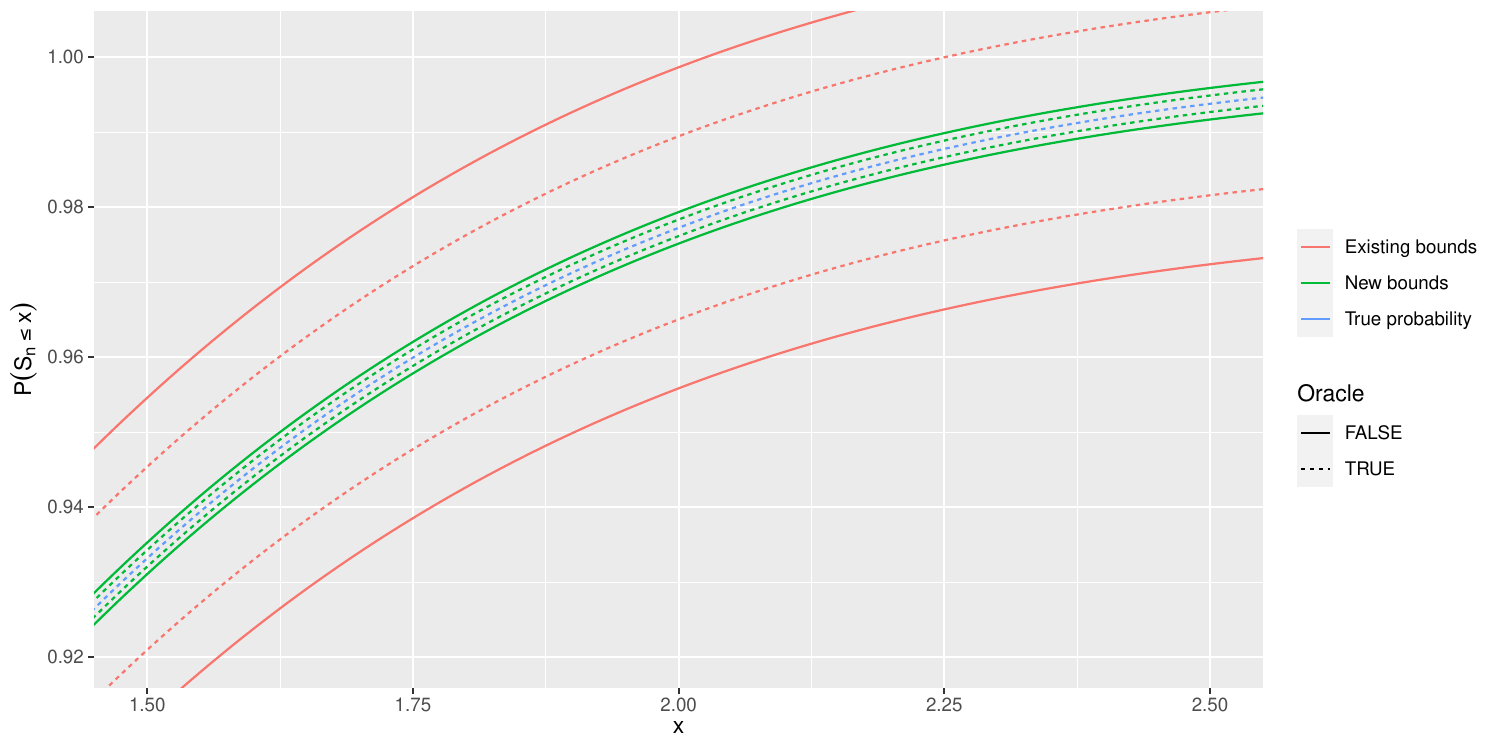}
    \caption{Setting: \iid{} unskewed ($\lambdatroisn = 0$)
    with $X_n \sim \textnormal{Student}(\textnormal{df} = 8)$ and $n = 5,\!000$. \\  
    {\footnotesize \textbf{Blue line:} $\Prob(S_n \leq x)$ as a function of $x$.} \\
    {\footnotesize \textbf{Continuous green lines:} bounds
    \(\Phi(x) \pm \delta_n^{\textnormal{new}}\)
    where $\delta_n^{\textnormal{new}}$ denotes the right-hand side of Corollary~\ref{cor:improvement_iid_case} with \(\kappa_n = 0.99\), $\Kquatren \leq 9$, and \(\Ktroisn \leq 9^{3/4}\).} \\
    {\footnotesize \textbf{Dashed green lines:} bounds
    \(\Phi(x) \pm  \delta_n^{\textnormal{new, oracle}}\), where $\delta_n^{\textnormal{new, oracle}}$ denotes the right-hand side of Corollary~\ref{cor:improvement_iid_case} with \(\kappa_n = 0.99\) and using the true (oracle) values of $\Kquatren = 4.5$ and $\Ktroisn \approx 1.8$.} \\
    {\footnotesize \textbf{Continuous red lines:} bounds $\Phi(x) \pm \ExistingBoundBEiid \Ktroisn / \sqrt{n}$ using the bound \(\Ktroisn \leq 9^{3/4} \approx 5.2\).} \\
    {\footnotesize \textbf{Dashed red lines:} bounds $\Phi(x) \pm \ExistingBoundBEiid \Ktroisn / \sqrt{n}$ using the true value \(\Ktroisn \approx 1.8\).}
    }
    
    \label{fig:graph_Prob_Sn_student_df8}
\end{figure}

\begin{figure}[hbt]
    \centering
    \includegraphics[width=1\textwidth]{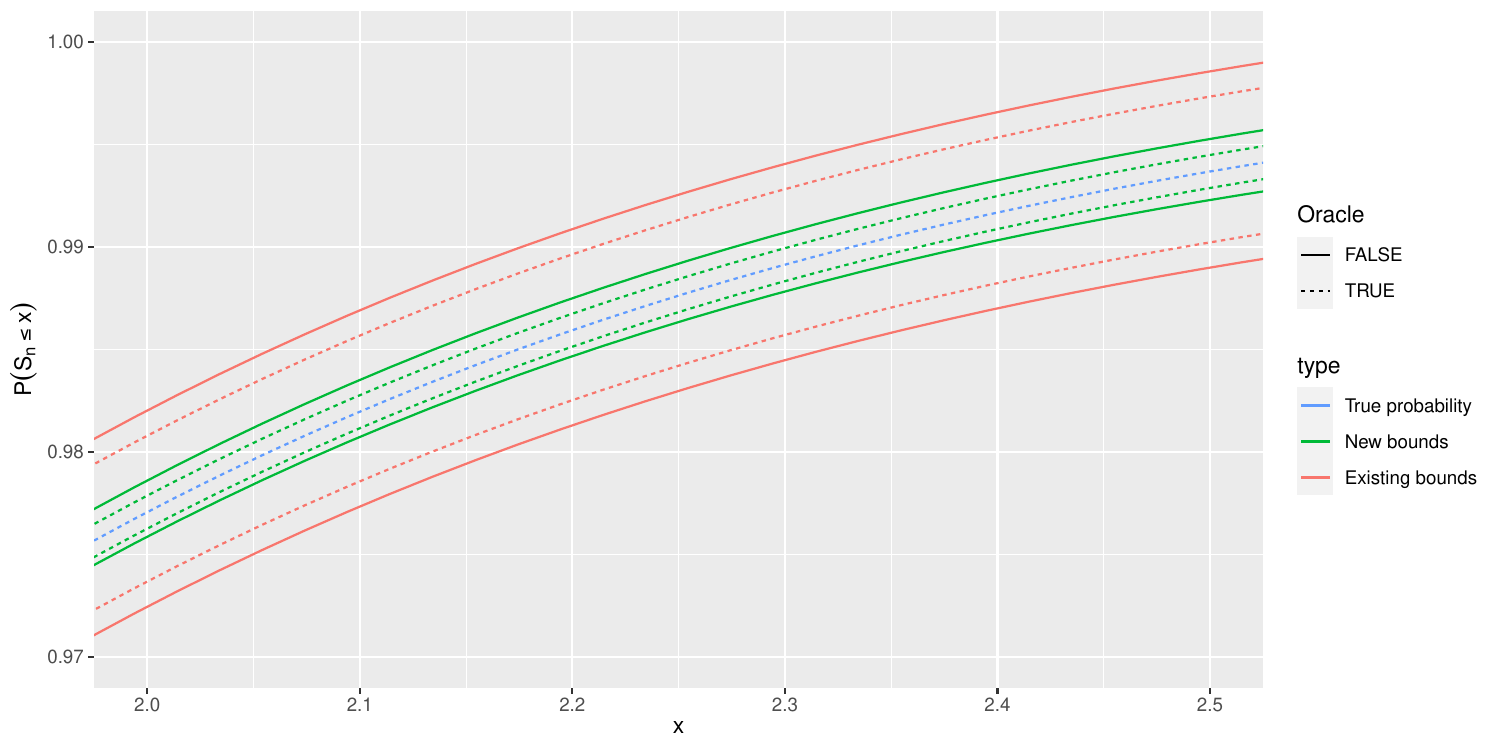}
    \caption{Setting: \iid{} skewed
    ($\lambdatroisn \neq 0$)
    with $X_n \sim \textnormal{Exp}(1) - 1$ and
    $n = 100,\!000$. \\
    {\footnotesize \textbf{Blue line:} $\Prob(S_n \leq x)$ as a function of $x$.} \\
    {\footnotesize \textbf{Continuous green lines:} bounds
    $\Phi(x) \pm 
    \Big(  0.621 \times 9^{3/4} / (6 \sqrt{n}) \times (1 - x^2) \varphi(x)
    + \delta_n^{\textnormal{new}} \Big)$
    where $\delta_n^{\textnormal{new}}$ denotes the right-hand side of Corollary~\ref{cor:improvement_iid_case} with \(\kappa_n = 0.99\) and $\Kquatren \leq 9$
    (as in Example~\ref{ex:be_bounds_comparison_cont}).} \\
    {\footnotesize \textbf{Dashed green lines:} bounds
    $\Phi(x)
    + \lambdatroisn / (6 \sqrt{n}) \times (1 - x^2) \varphi(x)
    \pm  \delta_n^{\textnormal{new, oracle}}$, where $\delta_n^{\textnormal{new, oracle}}$ denotes the right-hand side of Corollary~\ref{cor:improvement_iid_case} with \(\kappa_n = 0.99\) and using the true (oracle) values of $\Kquatren$, $\Ktroisn$ and $\lambdatroisn$.} \\
    {\footnotesize \textbf{Continuous red lines:} bounds $\Phi(x) \pm \ExistingBoundBEiid \Ktroisn / \sqrt{n}$ using the bound \(\Ktroisn \leq 9^{3/4} \approx 5.2\).} \\
    {\footnotesize \textbf{Dashed red lines:} bounds $\Phi(x) \pm \ExistingBoundBEiid \Ktroisn / \sqrt{n}$ using the true value \(\Ktroisn \approx 2.45\).}
    }
    \label{fig:graph_Prob_Sn_gamma}
\end{figure}

{\color{black}

\medskip

The Student distributions illustrate the unskewed case, where our bounds use the information~\(\lambdatroisn = 0\).
Figures~\ref{fig:graph_Prob_Sn_student_df5} and~\ref{fig:graph_Prob_Sn_student_df8} report several bounds contrasting the suggested practical choice \(\Kquatren \leq 9\), to deal with the fact that moments are unknown, with the ``oracle'' bounds where we use the true values of \(\lambdatroisn\), \(\Ktroisn\), and~\(\Kquatren\) (computed or approximated by Monte-Carlo).
As a comparison, we also report two versions of the existing bound: a ``practical'' one using \(\Ktroisn \leq \Kquatren^{3/4} \leq 9^{3/4}\), and an ``oracle'' version using the true value of~\(\Ktroisn\).
The kurtosis of a Student distribution is equal to \( 3 + 6 / (\textnormal{df } - 4)\) with \(\textnormal{df} > 4\) its degree of freedom.
Therefore, for any Student with at least 5 degrees of freedom, the upper bound \(\Kquatren \leq 9\) is valid, but all the more conservative as \(\textnormal{df}\) is large.
We consider two different values of \(\textnormal{df}\) to assess the loss of accuracy of our bounds when the discrepancy between the actual \(\Kquatren\) and our suggested default choice of~\(9\) increases.

In Figure~\ref{fig:graph_Prob_Sn_student_df8}, we choose \(\textnormal{df} = 8\) so that the true value is \(\Kquatren = 4.5\) and the proposed bound \(\Kquatren \leq 9\) is thus conservative.
On the contrary, in Figure~\ref{fig:graph_Prob_Sn_student_df5}, because \(\textnormal{df} = 5\), the true value of \(\Kquatren\) is equal to the suggested choice of~\(9\), which becomes sharp.
In that respect, it is a more favorable situation.
Nonetheless, remark that there remains a difference between the ``practical'' and ``oracle'' versions of our bounds: the latter uses the true value of \(\Ktroisn\) (here, approximately equal to~\(1.8\)) while the former controls \(\Ktroisn\) by \(9^{3/4} \approx 5.2\).


\medskip 

The Exponential distribution displayed in Figure~\ref{fig:graph_Prob_Sn_gamma} illustrates our bounds for a skewed distribution. We choose an Exponential distribution 
with expectation equal to~1. This distribution has a kurtosis \(\Kquatren = 9\) so that the main difference with Figure~\ref{fig:graph_Prob_Sn_student_df5} can be expected to stem from the presence of skewness. In line with the Student case, we report two versions of Shevtsova's bounds and ours, a practical version which uses only the information $\Kquatren \leq 9$ and an ``oracle'' one based on knowledge of $\lambdatroisn$, $\Ktroisn$ and $\Kquatren$. We recall that \(\DeltaB \neq \DeltaE\) when \(\lambdatroisn \neq 0\). What is more, the existing bounds (plotted in red) are bounds on~\(\DeltaB\) whereas ours (in green) originate from a control of \(\DeltaE\).

\medskip  

The ``oracle'' version can be interpreted as a noise-free implementation of the plug-in approach. We remark that oracle versions of existing bounds and ours are twice as accurate as their counterparts which rely on \(\Kquatren \leq 9\). These oracle bounds use by definition the true values of the moments, and therefore correspond to the most favorable case, in the sense of the tightness of the bounds.

}

\section{Non-asymptotic behavior of one-sided tests}
\label{sec:application_to_testing}

We now examine some implications of our theoretical results for the non-asymptotic validity of one-sided statistical tests based on the Gaussian approximation of the distribution of a sample mean using \iid{} data.

\medskip

Let $(Y_i)_{i=1, \dots, n}$ be an \iid{} sequence of random variable with expectation~$\mu$, known variance~$\sigma^2$ and finite fourth moment with ${K_4 : = \E\left[(Y_n-\mu)^4\right]/\sigma^4}$ the kurtosis of the distribution of $Y_n$.
We want to conduct a test of the null hypothesis ${H_0 : \mu \leq \mu_0}$, for some fixed real number $\mu_0$, against the alternative ${H_1 : \mu > \mu_0}$ with a type~I error at most $\alpha \in (0,1)$, and ideally equal to~$\alpha$.
The classical approach to this problem {\color{black}(Gauss test)} amounts to comparing
${S_n = \sum_{i=1}^n X_i / \sqrt{n}}$,
where $X_i := (Y_i - \mu_0) / \sigma$,
with the $1-\alpha$ quantile of the $\mathcal{N}(0,1)$ distribution, $q_{\mathcal{N}(0,1)}(1-\alpha)$, and reject $H_0$ if $S_n$ is larger.
{\color{black}We study this Gauss test in the general non-asymptotic framework without imposing Gaussianity of the data distribution, and we control the difference with respect to normality using the bounds developed in the previous sections.}

\subsection{Computation of sufficient sample sizes}

{\color{black}

In certain fields such as medicine or economics, researchers routinely set up experiments that seek to answer a specific question on an explained variable~\(Y\).
The number of individuals included in the experiment has to be carefully justified as large-scale analyses are very costly. 
This is typically done through the construction of a so-called ``pre-analysis plan'' which presents the sample size needed to detect a given effect with a pre-specified testing power $\beta \in (0,1)$.
In the Gauss test setting considered here,
the researcher determines the effect of interest by fixing a particular alternative hypothesis 
$H_{1, \eta}: \mu = \mu_0 + \sigma \eta$ (with $\mu>\mu_0$).
The quantity \(\eta := (\mu - \mu_0) / \sigma\) is a positive number called the effect size that indicates how far away (in terms of standard deviations) the alternative hypothesis is, compared to the null hypothesis $H_0: \mu \leq \mu_0$. Remark that in our framework, $H_{1, \eta}$ is formally the set of all distributions with mean $\mu$, variance $\sigma^2$, that satisfy our additional moment and regularity conditions.
$H_{1, \eta}$ can be seen as a nonparametric class of distributions at a fixed distance~$\eta$ of the null hypothesis.

\medskip

Researchers usually rely on an asymptotic normal approximation to infer the sample size needed to detect a given effect at power $\beta$.
Our results allow us to bypass this asymptotic approximation 
and to propose a procedure to choose the sample size $n$ of the experiment such that
\begin{equation}
    \Prob \Big( \textnormal{Rejection of } H_0 \Big)
    := \Prob \Big(\sum_{i=1}^n (Y_i - \mu_0) / \sqrt{n \sigma^2}
    > q_{\mathcal{N}(0,1)}(1-\alpha) \Big)
    \geq \beta,
    \label{eq:def_power}
\end{equation}
for any distribution belonging to the alternative hypothesis space. 
Any $n$ that satisfies Equation~\eqref{eq:def_power} for all distributions in the alternative hypothesis is called a (non-asymptotic) sufficient sample size for the effect size $\eta$ at power~$\beta$.

\medskip 

Observe that
\begin{align*}
    \Prob \Big( \textnormal{Rejection } H_0 \Big)
    &= \Prob \Big(\sum_{i=1}^n
    (Y_i - \mu + \mu - \mu_0) / \sqrt{n \sigma^2}
    > q_{\mathcal{N}(0,1)}(1-\alpha) \Big) \\
    &= \Prob \Big(\sum_{i=1}^n X_i / \sqrt{n} > x_n \Big),
\end{align*}
where $X_i := (Y_i - \mu) / \sigma$ are centered with mean $0$ and variance $1$ 
and $x_n := q_{\mathcal{N}(0,1)}(1-\alpha) - \eta \sqrt{n}$.
We remind the reader that the general result from Theorem~\ref{thm:nocont_choiceEps} or Corollary~\ref{cor:improvement_iid_case} implies the following upper and lower bounds for every~\(x \in \Rb\) and~\(n \geq 3\),
\begin{align}
    \frac{\lambda_{3,n}}{6\sqrt{n}}(1-x^2)\varphi(x)
    - \delta_n
    \leq \Prob(S_n \leq x) - \Phi(x)
    \leq \frac{\lambda_{3,n}}{6\sqrt{n}}(1-x^2)\varphi(x)
    + \delta_n,
    \label{eq:bracket_Edgeworth}
\end{align}
where $\delta_n$ is the corresponding bound on $\DeltaE$. 
From Equation~\eqref{eq:bracket_Edgeworth}, we thus obtain
\begin{align*}
    1 - \Prob\Big(\sum_{i=1}^n X_i / \sqrt{n} > x_n \Big)
    - \Phi(x_n)
    \leq \frac{\lambdatroisn}{6\sqrt{n}}(1 - x_n^2) \varphi(x_n)
    + \delta_n.
\end{align*}
Therefore,
\begin{align*}
    \Prob\Big(\sum_{i=1}^n X_i / \sqrt{n} > x_n \Big)
    \geq
    1 - \Phi(x_n)
    - \frac{\lambdatroisn}{6\sqrt{n}}(1 - x_n^2) \varphi(x_n)
    - \delta_n.
\end{align*}
As a consequence, the sample size $n = n_{\eta, \beta}$ defined as the solution of the following equation
\begin{align*}
    1 - \Phi \Big( q_{\mathcal{N}(0,1)}(1-\alpha)
    - \eta \sqrt{n} \Big)
    &- \frac{\lambdatroisn \times \left(1 - \left( q_{\mathcal{N}(0,1)}(1-\alpha)
    - \eta \sqrt{n} \right)^2 \right)}{6\sqrt{n}} \\
    &\times \varphi \Big( q_{\mathcal{N}(0,1)}(1-\alpha)
    - \eta \sqrt{n} \Big)
    - \delta_n = \beta,
\end{align*}
is a non-asymptotic sufficient sample size.
Note that the same reasoning can be also applied if we only impose an upper bound on $\lambdatroisn$.
In particular, if we only know $\Kquatren$, we can use the bound $0.621 \Kquatren^{3/4}$ and then a sufficient sample size $n$ can be found as the solution to
\begin{align}
    1 - \Phi \Big( q_{\mathcal{N}(0,1)}(1-\alpha)
    - \eta \sqrt{n} \Big)
    &- \frac{0.621 \Kquatren^{3/4} \times \left(1 - \left( q_{\mathcal{N}(0,1)}(1-\alpha)
    - \eta \sqrt{n} \right)^2 \right)}{6\sqrt{n}} \nonumber \\
    &\times \varphi \Big( q_{\mathcal{N}(0,1)}(1-\alpha)
    - \eta \sqrt{n} \Big)
    - \delta_n = \beta.
    \label{eq:power_func_n_eta}
\end{align}

\begin{table}[htb]
    \centering
    \begin{tabular}{rrrrrrrr}
    \toprule
    $\beta$ (\%) & $\eta = 0.01$ & $\eta = 0.02$ & $\eta = 0.05$ & $\eta = 0.1$
    & $\eta = 0.2$ & $\eta = 0.5$ \\ 
    \midrule
50 & 27,993 & 7,489 & 1,463 & 501 & 280 & 265 \\ 
80 & 62,597 & 16,237 & 2,988 & 967 & 549 & 548 \\ 
85 & 72,686 & 18,841 & 3,490 & 1,176 & 789 & 789 \\ 
90 & 86,507 & 22,462 & 4,255 & 1,636 & 1,469 & 1,469 \\ 
95 & 109,374 & 28,665 & 5,976 & 4,070 & 4,070 & 4,070 \\ 
99 & 161,151 & 45,735 & 27,946 & 27,946 & 27,946 & 27,946 \\ 
    \bottomrule
    \end{tabular}
    \vspace{0.1cm}
    \caption{Sufficient sample sizes for the experiment to be well-powered for a nominal power $\beta$ for the detection of an effect size $\eta$.
    We use the bound from Corollary~\ref{cor:improvement_iid_case} (\iid{} case with additional regularity assumption).
    As in Examples~\ref{ex:be_bounds_comparison_nocont} and~\ref{ex:be_bounds_comparison_cont}, we use $\Kquatren \leq 9$, \(\lambdatroisn \leq 0.621 \Kquatren^{3/4}\), and $\kappa \leq 0.99$ to compute \(n_{\eta, \beta}\) (see Equation~\eqref{eq:power_func_n_eta}).}
    \label{tab:sufficient_size}
\end{table}

Numerical applications can be found in Table~\ref{tab:sufficient_size}
which displays the computed sample sizes for different choices of effect sizes $\eta$ and of power $\beta$.
In this experiment, we choose $\Kquatren \leq 9$ and $\kappa \leq 0.99$, as before.
We can observe that, as expected, $n_{\eta, \beta}$ increases with $\beta$ and decreases with $\eta$.
For $\eta$ large enough, $n_{\eta, \beta}$ becomes approximately constant in $\eta$ as Equation~\eqref{eq:power_func_n_eta} simplifies to $1 - \delta_n = \beta.$
Conversely, it is also possible to use directly Equation~\eqref{eq:power_func_n_eta} to compute the power for different effects and sample sizes. The results are displayed in Table~\ref{tab:implied_power}.

\begin{table}[htb]
    \centering
    \begin{tabular}{rrrrrrrr}
    \toprule
    $n$ & $\eta = 0.01$ & $\eta = 0.02$ & $\eta = 0.05$ & $\eta = 0.1$
    & $\eta = 0.2$ & $\eta = 0.5$ \\ 
    \midrule
200 & 0.0 & 0.0 & 0.0 & 0.0 & 15.8 & 27.3 \\ 
500 & 0.0 & 0.0 & 7.4 & 49.9 & 78.0 & 78.1 \\ 
800 & 0.0 & 0.0 & 25.3 & 73.5 & 85.1 & 85.1 \\ 
1,000 & 0.0 & 2.8 & 34.0 & 81.0 & 87.2 & 87.2 \\ 
2,000 & 3.5 & 14.3 & 63.9 & 91.7 & 91.8 & 91.8 \\ 
5,000 & 13.1 & 36.3 & 92.9 & 95.7 & 95.7 & 95.7 \\ 
10,000 & 23.3 & 61.2 & 97.4 & 97.5 & 97.5 & 97.5 \\ 
50,000 & 71.7 & 99.2 & 99.5 & 99.5 & 99.5 & 99.5 \\ 
100,000 & 93.3 & 99.8 & 99.8 & 99.8 & 99.8 & 99.8 \\ 
    \bottomrule
    \end{tabular}
    \vspace{0.1cm}
    \caption{Lower bound~\eqref{eq:power_func_n_eta} on the power $\beta$ (\%) as a function of the effect size $\eta$ and sample size $n$, with our bounds from Corollary~\ref{cor:improvement_iid_case},
    $\Kquatren \leq 9$, and $\kappa \leq 0.99$.}
    \label{tab:implied_power}
\end{table}

}

\subsection{Assessing the lack of information}
\label{subsec:uninformativeness_examples}


{\color{black}
As explained below, the non-asymptotic bounds introduced in Sections~\ref{ssec:thms_nocont} and~\ref{ssec:thms_cont} can be used to evaluate the actual (for a finite sample size) level of our one-sided test of interest.

\medskip 

Recall that Berry-Esseen-type inequalities aim to bound \(\DeltaB\), defined in Equation~\eqref{eq:def_Delta_n_B}, the uniform distance between \(\Prob(S_n \leq \cdot)\) and \(\Phi(\cdot)\).
In particular, for a nominal level~\(\alpha\), we thus have
\begin{equation*}
    \Big| \Prob\big( S_n \leq q_{\mathcal{N}(0,1)}(1-\alpha) \big)
    - (1-\alpha) \Big|
    \leq 
    \DeltaB,
\end{equation*}
where the probability operator is to be understood under any data-generating process such that~\(\mu = \mu_0\), to be as close as possible to the alternative hypothesis $H_1$.
Either ``classical'' Berry-Esseen inequalities or ours obtained through an Edgeworth expansion provide bounds on~\(\DeltaB\) (see the different bounds displayed in Examples~\ref{ex:be_bounds_comparison_nocont} and~\ref{ex:be_bounds_comparison_cont} in the \iid{} case).
In this context, a bound on~\(\DeltaB\) is said to \emph{uninformative} when it is larger than~\(\alpha\).
Indeed, in that case, we cannot exclude that $\Prob\!\left( S_n \leq q_{\mathcal{N}(0,1)}(1-\alpha) \right)$ is arbitrarily close to~1, or equivalently, that the probability to reject $H_0$ is arbitrarily close to~$0$, and therefore that the test is arbitrarily conservative (type~I error arbitrarily smaller than the nominal level~\(\alpha\)).
We denote by $n_{\max}(\alpha)$ the largest sample size $n$ for which the bound is uninformative.
Intuitively, $n_{\max}(\alpha)$ indicates the sample size above which the asymptotic normal approximation to the distribution of~\(S_n\) becomes sensible under the assumptions used to bound~\(\DeltaB\).
Indeed, $n_{\max}(\alpha)$ is specific to the bound \(\delta_n\) used, which itself depends on various features of the distribution: number of finite moments, (lack of) skewness, regularity, etc.
Table~\ref{tab:largest_weakly_inf_n} reports the value of $n_{\max}(\alpha)$ for different Berry-Esseen bounds and usual nominal levels~\(\alpha \in \{0.10, 0.05, 0.01\}\).

\begin{table}[ht]
\begin{center}
\begin{minipage}{174pt}
\makegapedcells 
\begin{tabular}{@{}llll@{}}
    \toprule
        Bound on \(\DeltaB\) & $\alpha = 0.10$ & $\alpha = 0.05$ & $\alpha = 0.01$ \\
    \midrule
        Existing & 593 & 2,375 & 59,389 \\
        Thm. \ref{thm:nocont_choiceEps} & 2,339 & 6,705 & 55,894 \\ 
        Thm. \ref{thm:nocont_choiceEps} unskewed & 443 & 1,229 & 17,934 \\ 
        Cor. \ref{cor:improvement_iid_case} & 1,468 & 4,069 & 27,945 \\ 
        Cor. \ref{cor:improvement_iid_case} unskewed & 375 & 474 & 1,062 \\ 
    \bottomrule
    \end{tabular}
\end{minipage}
\caption{$n_{\max}(\alpha)$, for different assumptions and Berry-Esseen bounds: \cite{shevtsova2013}'s bound with finite third moment (Existing), our bound with finite fourth moment (Thm.~\ref{thm:nocont_choiceEps}), our bound with additional regularity condition on $f_{X_n / \sigma_n}$ (Cor.~\ref{cor:improvement_iid_case}).\\
We impose the same restrictions as in Examples~\ref{ex:be_bounds_comparison_nocont} and~\ref{ex:be_bounds_comparison_cont}, namely $\Kquatren \leq 9$ and $\kappa \leq 0.99$.}
\label{tab:largest_weakly_inf_n}
\end{center}
\end{table}

For each bound, \(n_{\max}(\alpha)\) is decreasing in~\(\alpha\).
For \(\alpha = 0.01\) in particular, the situation deteriorates strikingly except in the most favorable case of a regular and unskewed distribution.
With our bounds, the presence or absence of skewness strongly influences \(n_{\max}(\alpha)\).
We also remark that imposing the additional regularity assumption introduced in Section~\ref{ssec:thms_cont} significantly lowers~\(n_{\max}(\alpha)\). 
}

\subsection{Distortions of the level of the test and of the p-values}

We explain now that our non-asymptotic bounds on the Edgeworth expansion can be used to detect whether the test is conservative or liberal.
This goes one step further than merely checking whether it is arbitrarily conservative or not.
Equation~\eqref{eq:bracket_Edgeworth} shows that $\Prob(S_n\leq x)$ belongs to the interval 
\[
\mathcal{I}_{n,x} := \left[ \Phi(x) + \lambdatroisn (1-x^2)\varphi(x) / (6\sqrt{n}) \pm \delta_n \right],
\]
which is not centered at $\Phi(x)$ whenever $\lambdatroisn \neq 0$ and $x \neq \pm \, 1$. The length of the interval does not depend on~$x$ and shrinks at speed $\delta_n$.
On the contrary, its location depends on~$x$.
For given nonzero skewness~\(\lambdatroisn\) and sample size~\(n\), the middle point of $\mathcal{I}_{n,x}$ is all the more shifted away from the asymptotic approximation $\Phi(x)$ as \( (1-x^2)\varphi(x)\) is large in absolute value.
The function \(x \mapsto (1-x^2)\varphi(x)\) has global maximum at $x=0$ and minima at the points $x \approx \pm \, 1.73$.
Consequently, irrespective of $n$, the largest gaps between $\Prob(S_n\leq x)$ and $\Phi(x)$ may be expected around $x=0$ or $x = \pm \, 1.73$.
$\Phi(x)$ could even lie outside $\mathcal{I}_{n,x}$, in which case $\Prob( S_n \leq x )$ has to be either strictly smaller or larger than $\Phi(x)$.
More precisely, \(\Prob( S_n \leq x )\) is all the further from its normal approximation \(\Phi(x)\) as the skewness \(\lambdatroisn\) is large in absolute value; whether \(\Prob( S_n \leq x )\) is strictly smaller or larger than \(\Phi(x)\) depends on the sign of \(1 - x^2\) as developed in Table~\ref{tab:cases_distortion_P_Sn_leq_x_Phi_x}.
\begin{table}[ht]
\begin{center}
\begin{minipage}{\textwidth}
\makegapedcells 
\begin{tabular*}{\textwidth}{@{\extracolsep{\fill}}lcc@{\extracolsep{\fill}}}
    \toprule
        & \(\Prob( S_n \leq x ) < \Phi(x) \) & \(\Prob( S_n \leq x ) > \Phi(x) \) \\
    \midrule
        If \( |x| < 1 \) & \( \lambdatroisn < 6\sqrt{n}\delta_n/\big((x^2-1)\varphi(x)\big) < 0 \) & \( \lambdatroisn > 6\sqrt{n}\delta_n/\big((1-x^2)\varphi(x)\big) > 0 \) \\
        If \( |x| >  1 \) & \( \lambdatroisn > 6\sqrt{n}\delta_n/\big((x^2-1)\varphi(x)\big)  > 0\) & \( \lambdatroisn < 6\sqrt{n}\delta_n/\big((1-x^2)\varphi(x)\big) < 0 \) \\
    \bottomrule
    \end{tabular*}
    \end{minipage}
    \caption{{\small Cases and conditions on the skewness \(\lambdatroisn\) under which \(\Prob( S_n \leq x )\) is either strictly smaller or larger than its normal approximation~\(\Phi(x)\) for any given sample size~\(n \geq 3\).}}
    \label{tab:cases_distortion_P_Sn_leq_x_Phi_x}
\end{center}
\end{table}

These observations allow us to quantify possible non-asymptotic distortions between the nominal level and actual rejection rate of the one-sided test we consider. 
Let us set $x = q_{\mathcal{N}(0,1)}(1-\alpha)$ (henceforth denoted $q_{1-\alpha}$ to lighten notation), which implies that $\Phi(x) = 1- \alpha$.
Here, we focus solely on the case $|q_{1-\alpha}|>1$ to encompass all tests with nominal level $\alpha \leq 0.15$, thus in particular the conventional levels 10\%, 5\%, and~1\%.
When $\lambdatroisn > 6\sqrt{n}\delta_n/\big((q_{1-\alpha}^2-1)\varphi(q_{1-\alpha})\big)$, we conclude that $\Prob\left( S_n \leq q_{1-\alpha} \right) < 1-\alpha$. Since the event \( \{ S_n \leq q_{1-\alpha} \} \) is the complement of the rejection region, the probability of rejecting $H_0$ under the null exceeds~\(\alpha\); in other words, the test cannot guarantee its stated control~\(\alpha\) on the type~I error and is said liberal.
Conversely, when $\lambdatroisn < 6\sqrt{n}\delta_n/\big((1-q_{1-\alpha}^2)\varphi(q_{1-\alpha})\big)$, the probability $\Prob\left( S_n \leq q_{1-\alpha} \right)$ has to be larger than $1-\alpha$; equivalently, the probability to reject under the null is below~\(\alpha\) so that the test is conservative.

The distortion can also be seen in terms of p-values.
In the unilateral test we consider, the p-value is \({pval := 1 - \Prob(S_n \leq s_n)}\) with $s_n$ the observed value of $S_n$ in the sample.
In contrast, the approximated p-value is ${\widetilde{pval} := 1 - \Phi(s_n)}$.
Setting $x = s_n$ in Equation~\eqref{eq:bracket_Edgeworth} yields
\begin{equation*}
    \frac{\lambda_{3,n}}{6\sqrt{n}}(1-s_n^2)\varphi(s_n)
    - \delta_n
    \leq (1 - pval) - (1 - \widetilde{pval})
    \leq \frac{\lambda_{3,n}}{6\sqrt{n}}(1-s_n^2)\varphi(s_n)
    + \delta_n.
\end{equation*}
Therefore,
\begin{equation}
\label{eq:bracket_pval}
    \widetilde{pval}
    - \frac{\lambda_{3,n}}{6\sqrt{n}}(1 - s_n^2)\varphi(s_n)
    - \delta_n
    \leq
    pval 
    \leq 
    \widetilde{pval} - \frac{ \lambda_{3,n}}{6\sqrt{n}}(1 - s_n^2)\varphi(s_n)
    + \delta_n.
\end{equation}
In line with the explanations preceding Table~\ref{tab:cases_distortion_P_Sn_leq_x_Phi_x}, \(\widetilde{pval}\) is strictly smaller or larger than \(pval\) when the skewness is sufficiently large in absolute value relative to~\(\delta_n\).
Indeed, if \(\lambdatroisn \neq 0\), the interval from Equation~\eqref{eq:bracket_pval} that contains the true p-value \(pval\) is not centered at the approximated p-value \(\widetilde{pval}\).
Under additional regularity assumptions (see Corollary~\ref{cor:improvement_iid_case} in the \iid{} case), the remainder term \(\delta_n = O(n^{-1})\) whereas the ``bias'' term involving \(\lambdatroisn\) vanishes at rate \(n^{-1/2}\).
As a result, the interval locates closer to \(\widetilde{pval}\) as $n$ increases and its width shrinks to zero at an even faster rate.

Finally, we stress that such distortions regarding rejection rates and p-values are specific to one-sided tests.
For bilateral or two-sided tests, the skewness of the distribution enters symmetrically in the approximation error and cancels out thanks to the parity of \( x \mapsto (1 - x^2) \phi(x) \).


\bibliography{main_arxiv_v3_v9project}

\newpage 

\appendix

\section{Proof of the main theorems}
\label{appendix:sec:proof_main_theorems}

\subsection{Outline of the proofs of Theorems~\ref{thm:nocont_choiceEps} and~\ref{thm:cont_choiceEps}}
\label{ssec:proof:outline}

We start by presenting a lemma derived in \cite{prawitz1975}, which is central to prove our theorems.
This result helps control the distance between the cumulative distribution function $F$ of a random variable with skewness $v$ and its first order Edgeworth expansion  $G_v(x):=\Phi(x)+\frac{v}{6}(1-x^2) \varphi(x)$ in terms of their respective Fourier transforms.

\begin{lemma}
    Let $F$ be an arbitrary cumulative distribution function with characteristic function $f$ and skewness $v$.
    Let $\tau, T > 0$. Then we have
    \begin{align}
        \sup_{x\in\Rb} \big| F(x) - G_v(x) \big| 
        & \leq \, \Omega_1(T, v, \tau)
        + \Omega_2(T) + \Omega_3(T, v, \tau)
        + \Omega_4(\tau \wedge T/\pi, T/\pi, T),
        \label{eq:smoothing}
    \end{align}
    where
    \vspace{-0.5em}
    \begin{align*}\arraycolsep=0.15em \def\arraystretch{2}
    \begin{array}{lll}
        \Omega_1(T, v, \tau)
        &:=& 2 \displaystyle\int_0^{T/\pi}
        \left|\frac{1}{T}\Psi(u/T)-\dfrac{i}{2\pi u}\right|e^{-u^2/2}
        \left(1+\dfrac{|v|u^3}{6}\right) du \\
        && \hspace{2cm} + \dfrac{1}{\pi} \displaystyle \int_{T/\pi}^{+\infty}
        \dfrac{e^{-u^2/2}}{u} \left(1+\dfrac{|v|u^3}{6}\right) du \\
        && \hspace{2cm} + 2 \displaystyle \int_{\tau \wedge T/\pi}^{T/\pi}
        \left|\frac{1}{T}\Psi(u/T)\right|\, e^{-u^2/2}
        \dfrac{|v|u^3}{6} du, \\
        \Omega_2(T) 
        &:=& 2\displaystyle\int_{T/\pi}^T
        \left|\frac{1}{T}\Psi(u/T)\right|\,|f(u)| du, \\
        \Omega_3(T, v, \tau)
        &:=& \displaystyle 2 \int_0^{\tau \wedge T/\pi}      \left|\frac{1}{T}\Psi(u/T)\right|\, \left|f(u)-e^{-u^2/2}
        \left(1-\dfrac{viu^3}{6}\right)\right| du, \\
        \Omega_4(a, b, T)
        &:=& \displaystyle 2 \int_{a}^{b}
        \left|\frac{1}{T}\Psi(u/T)\right| \,
        \left|f(u)-e^{-u^2/2} \right| du, \\
    \end{array}
    \end{align*}
    and
    $\Psi(t):= \frac{1}{2}
    \left(1-|t|+i\left[(1-|t|)\cot(\pi t)+\frac{\sign(t)}{\pi}\right]\right)
    \Indicator\left\{ |t| \leq 1 \right\}$.
    \label{lemma:smoothing}
\end{lemma}

For the sake of completeness, we give a proof of this lemma in Section~\ref{proof:lemma:smoothing}.
We also use the following properties on the function $\Psi$ \cite[Equations (I.29) and (I.30)]{prawitz1975}
\begin{equation}
    \label{eq:properties_Psi}
    \big| \Psi(t) \big| \leq \frac{1.0253}{2 \pi |t|}
    \; \text{ and } \;
    \left| \Psi(t) - \frac{i}{2\pi t} \right|
    \leq \frac{1}{2}
    \left( 1 - |t| + \frac{\pi^2}{18} t^2\right).
\end{equation}

Lemma~\ref{lemma:smoothing} is valid for any positive values \(T\) and \(\tau\).
The latter are free parameters whose values determine which terms are the dominant ones among \(\Omega_1\) to \(\Omega_4\).

\medskip

Theorem~\ref{thm:nocont_choiceEps} written in the body of the article synthesizes Theorems~\ref{thm:nocont_inid} and \ref{thm:nocont_iid} stated and proven below respectively in the \inid{} and the \iid{} cases.
Likewise, Theorem~\ref{thm:cont_choiceEps} corresponds to Theorems~\ref{thm:cont_inid} (\inid{} case) and~\ref{thm:cont_iid} (\iid{} case).
The four proofs start by applying Lemma~\ref{lemma:smoothing} with $F$ the cdf of $S_n$ and thus $v = \lambdatroisn / \sqrt{n}$.
Then, for specific values of \(T\) and \(\tau\), we derive upper bounds on each of the four terms of Equation~\eqref{eq:smoothing}.

In all our theorems, we set 
\begin{equation}
\label{eq:tau_choice}
    \tau = \sqrt{2\eps} (n/K_{4,n})^{1/4},
\end{equation}
where \(\eps\) is a dimensionless free parameter.
It is not obvious to optimize our bounds over that parameter.
Consequently, Theorems~\ref{thm:nocont_inid}
to~\ref{thm:cont_iid} are proven for any \(\eps \in (0, 1/3)\) and, in the body of the article, we present the results with \(\eps = 0.1\), a sensible value according to our numerical comparisons.

Unlike~\(\tau\), we vary the rate of \(T\) across theorems.
In Theorems~\ref{thm:nocont_inid} and~\ref{thm:nocont_iid}, we choose
\[
T = \frac{2 \pi \sqrt{n}}{\Ktroisntilde}.
\]
The resulting bound is interesting under moment conditions only (Assumption~\ref{hyp:basic_as_inid} for \inid{} cases and~\ref{hyp:basic_as_iid} for \iid{} cases).

In Theorems~\ref{thm:cont_inid} and~\ref{thm:cont_iid}, we make a different choice, namely
\begin{equation*}
    T = \frac{16 \pi^4 n^2}{\Ktroisntilde^4}.
\end{equation*}
These last two theorems present alternative bounds, also valid under moment conditions only.
They improve on Theorems~\ref{thm:nocont_inid} and~\ref{thm:nocont_iid} under regularity conditions on the tail behavior of the characteristic function~\(\caracfsum\) of~\(S_n\).
Examples of such conditions are to be found in Corollaries~\ref{cor:improvement_inid_case} (\inid{} case) and~\ref{cor:improvement_iid_case} (\iid{} case).

\subsection{Proof of Theorem~\ref{thm:nocont_choiceEps} under Assumption~\ref{hyp:basic_as_inid}}
\label{ssec:proof:inid_nocont}

In this section, we state and prove a more general theorem (Theorem~\ref{thm:nocont_inid} below). We recover Theorem~\ref{thm:nocont_choiceEps} when we set $\eps = 0.1$.

\begin{theorem}[One-term Edgeworth expansion under Assumption~\ref{hyp:basic_as_inid}]
\label{thm:nocont_inid}
    
    (i) Under Assumption~\ref{hyp:basic_as_inid}, for every $\eps \in (0,1/3)$ and every $n \geq 1$, we have the bound
    \begin{align}
    \label{eq:edg_exp_inid_nocont}
        \DeltaE \leq  \frac{0.1995 \, \Ktroisntilde}{\sqrt{n}}
        + \frac{1}{n} \Bigg\{ & 0.031 \, \Ktroisntilde^2
        + 0.327 \, K_{4,n}\left(\frac{1}{12}+\frac{1}{4(1-3\eps)^2}\right) \nonumber \\
        &+ 0.054 \, |\lambda_{3,n}|\Ktroisntilde
        + 0.037 \, e_{1,n}(\eps)\lambda_{3,n}^2 \Bigg\} + \rninidskew{1}(\eps),
    \end{align}
    where $e_{1,n}(\eps)$ is given in Equation~\eqref{eq:def_e_1n} and $\rninidskew{1}(\eps)$ is given in Equation~\eqref{eq:def_r_1n_inid_eps}.

    \medskip
    
    (ii) If we further impose $\E[X_i^3]=0$ for every $i = 1,\dots,n$, the upper bound reduces to
    \begin{align}
    \label{eq:edg_exp_inid_nocont_sym}
        & \frac{0.1995\Ktroisntilde}{\sqrt{n}} + \frac{1}{n}\left\{ 0.031\Ktroisntilde^2 + 0.327K_{4,n}\left(\frac{1}{12}+\frac{1}{4(1-3\eps)^2}\right) \right\} + \rninidnoskew{1}(\eps),
    \end{align}
    where $\rninidnoskew{1}(\eps)$ is given in Equation~\eqref{eq:def_r_1n_inid_noskew_eps}.
    
    \medskip
    
    (iii) Finally, when $K_{4,n} = O(1)$ as \(n \to \infty\), we obtain $\rninidskew{1}(\eps) = O(n^{-5/4})$ and $ \rninidnoskew{1}(\eps) = O(n^{-3/2})$.
\end{theorem}

\medskip

Using Theorem~\ref{thm:nocont_inid}, we can finish the proof of Theorem~\ref{thm:nocont_choiceEps} by plugging-in our choice \(\eps = 0.1\) and computing the numerical constants.
In particular, the computation of $e_{1,n}(0.1)$ gives the upper bound \({e_{1,n}(0.1) \leq 1.0157}\).

In the general case with skewness, using the computations for $\RinidInt(0.1)$ carried out in Section~\ref{sec_bound_R1n}, the rest \( \rninidskew{1}(0.1) \) is bounded by the explicit expression
{\color{black}given in Equation~\eqref{eq:def_r_1n_inid_eps_eq_01}.}

In the no-skewness case,
the rest \( \rninidnoskew{1}(0.1) \)
is bounded by the explicit expression {\color{black}given in Equation~\eqref{eq:def_r_1n_inid_eps_eq_01_noskew},}
where we use the expression of $\RinidInt(\eps)$ in Equation~\eqref{eq:simplified_R1n_int_noskewed} and the computations when \( \eps = 0.1 \) that follow Equation~\eqref{eq:simplified_R1n_int_noskewed}.

\begin{proof}[Proof of Theorem~\ref{thm:nocont_inid}.]
\textbf{We first prove (i).}
We apply Lemma~\ref{lemma:smoothing} with $F$ denoting the cdf of $S_n$ and obtain
\begin{align*}
    \DeltaE
    &\leq \, \Omega_1(T, v, \tau)
    + \Omega_2(T) + \Omega_3(T, v, \tau)
    + \Omega_4(\tau \wedge T/\pi, T/\pi, T).
\end{align*}
Let $T := 2 \pi \sqrt{n} / \Ktroisntilde$,
$v := \lambda_{3,n} / \sqrt{n}$ and
$\tau := \sqrt{2\eps} (n/K_{4,n})^{1/4}$.
We combine now Lemma~\ref{lemma:bound_omega_1} (control of $\Omega_1$), 
Equation~\eqref{eq:bound_omega2_nocont} (control of $\Omega_2$), 
Lemma~\ref{lemma:bound_omega3} (control of $\Omega_3$), and
Lemma~\ref{lemma:bound_I32}(i) (control of $\Omega_4$) so that we get
\begin{align*}
    \DeltaE
    &\leq \frac{1.2533}{T}
    + \frac{0.3334 \, |\lambda_{3,n}|}{T\sqrt{n}}
    + \frac{14.1961}{T^4}
    + \frac{4.3394 \, |\lambdatroisn|}{T^3 \sqrt{n}} \nonumber \\
    &+ \frac{|\lambdatroisn|
    \big( \Gamma( 3/2 , \tau \wedge T/\pi)
    - \Gamma( 3/2 , T/\pi) \big)}{\sqrt{n}} \\
    &+ \frac{67.0415}{T^4} + \frac{1.2187}{T^2}
    + \frac{0.327 \, K_{4,n}}{n}
    \left(\frac{1}{12}+\frac{1}{4(1-3\eps)^2}\right)
    + \frac{1.306 \, e_{1,n}(\eps) |\lambda_{3,n}|^2}{36n} \nonumber \\
    &+ \frac{1.0253}{\pi}
    \int_0^{\tau \wedge T/\pi} u e^{-u^2/2} \Rinid(u,\eps) du
    + \frac{K_{3,n}}{3\sqrt{n}} J_{2} \big(3, \tau \wedge T/\pi, T/\pi, T/\pi , T \big).
\end{align*}
Bounding \((1.0253/\pi)\times\int_0^{\tau \wedge T/\pi} u e^{-u^2/2} \Rinid(u,\eps)\) by \(\RinidInt(\eps) := (1.0253/\pi)\times\int_0^{+ \infty} u e^{-u^2/2} \Rinid(u,\eps)\), bounding $J_2$ by Lemma~\ref{lemma:bound_J2p}, and replacing $T$ and $\tau$ by their values, we obtain
\begin{align*}
    \DeltaE
    &\leq \frac{1.2533 \, \Ktroisntilde}{2\pi \sqrt{n}}
    + \frac{0.3334 \, |\lambda_{3,n}|\Ktroisntilde}{2\pi n}
    + \frac{1.2187 \, \Ktroisntilde^2}{4\pi n} + \frac{0.327 K_{4,n}}{n}\left(\frac{1}{12}
    + \frac{1}{4(1-3\eps)^2}\right) \\
    & \hspace{2cm}
    + \frac{1.306 \, e_{1,n}(\eps) \lambda_{3,n}^2}{36n} + \rninidskew{1}(\eps)
\end{align*}
where
\begin{align}
    \rninidskew{1}(\eps)
    &:= \frac{(14.1961 + 67.0415) \, \Ktroisntilde^4}{16\pi^4 n^2}
    + \frac{4.3394 \, |\lambdatroisn| \Ktroisntilde^3}{8 \pi^3 n^2}
    + \RinidInt(\eps)
    \nonumber \\
    &+ \frac{ |\lambdatroisn| \big( \Gamma( 3/2 , \sqrt{2\eps} (n/K_{4,n})^{1/4} \wedge 2 \sqrt{n} / \Ktroisntilde)
    - \Gamma( 3/2 , 2 \sqrt{n} / \Ktroisntilde) \big) }{\sqrt{n}} \nonumber \\
    &+ \frac{1.0253\Ktroisn}{6\pi\sqrt{n}} \Bigg\{ 0.5|\Delta|^{-3/2}\Indicator_{\{\Delta \neq 0\}} \times \bigg|\gamma(3/2, 4 \Delta n / \Ktroisntilde^2) \nonumber \\
    & \qquad\qquad\qquad\;\; - \gamma\big(3/2, 2\Delta ( \eps (n/K_{4,n})^{1/2} \wedge 2 n / \Ktroisntilde^2 ) \big) \bigg| \nonumber \\
    & \qquad\qquad\qquad\;\; + \Indicator_{\{\Delta = 0\}}
    \frac{ ( 2 \sqrt{n} / \Ktroisntilde )^3
    - (\sqrt{2\eps} (n/K_{4,n})^{1/4} \wedge 2 \sqrt{n} / \Ktroisntilde )^3 }{3}
    \Bigg\},
    \label{eq:def_r_1n_inid_eps}
\end{align}
and $\Delta := (1 - 4 \chi_1 - \sqrt{K_{4,n}/n}) / 2$.

\medskip

We obtain the result of Equation~\eqref{eq:edg_exp_inid_nocont} by computing all numerical constants; for instance, \( 1.0253 / (2\pi) \approx 0.19942 < 0.1995 \).

\bigskip

\textbf{We now prove (ii).} In the no-skewness case, namely when \(\E[X_i^3] = 0\) for every~\(i = 1, \ldots, n\), the start of the proof is identical except that Lemma~\ref{lemma:bound_I32}(ii) is used in lieu of Lemma~\ref{lemma:bound_I32}(i) to control $\Omega_4$.
This yields 
\begin{align*}
    \DeltaE
    &\leq \frac{1.2533}{T}
    + \frac{14.1961}{T^4} + \frac{67.0415}{T^4} + \frac{1.2187}{T^2} + \frac{0.327 \, K_{4,n}}{n}
    \left(\frac{1}{12}+\frac{1}{4(1-3\eps)^2}\right) \\
    &+ \RinidInt(\eps) + \frac{K_{4,n}}{3n} J_{2} \big(4, \tau \wedge T/\pi, T/\pi , T/\pi , T \big).
\end{align*}
Bounding $J_2$ by Lemma~\ref{lemma:bound_J2p} and
replacing $T$ and $\tau$ by their values, we obtain
\begin{align*}
    \DeltaE
    &\leq \frac{1.2533 \, \Ktroisntilde}{2\pi \sqrt{n}} + \frac{1.2187 \, \Ktroisntilde^2}{4\pi n} + \frac{0.327 K_{4,n}}{n}\left(\frac{1}{12}
    + \frac{1}{4(1-3\eps)^2}\right) + \rninidnoskew{1}(\eps)
\end{align*}
where
\begin{align}
    \rninidnoskew{1}(\eps)
    &:= \frac{(14.1961 + 67.0415) \, \Ktroisntilde^4}{16\pi^4 n^2} + \RinidInt(\eps) \nonumber \\
    &+ \frac{1.0253\Kquatren}{6\pi n} \Bigg\{ 0.5|\Delta|^{-2}\Indicator_{\{\Delta \neq 0\}} \times \bigg|\gamma(2, 4 \Delta n / \Ktroisntilde^2) \nonumber \\
    & \qquad\qquad\qquad\;\; - \gamma\big(2, 2\Delta ( \eps (n/K_{4,n})^{1/2} \wedge 2 n / \Ktroisntilde^2 ) \big) \bigg| \nonumber \\
    & \qquad\qquad\qquad\;\; + \Indicator_{\{\Delta = 0\}}
    \frac{(2 \sqrt{n} / \Ktroisntilde)^4
    - (\sqrt{2\eps} (n/K_{4,n})^{1/4} \wedge 2 \sqrt{n} / \Ktroisntilde)^4 }{4}
    \Bigg\}
    \label{eq:def_r_1n_inid_noskew_eps}
\end{align}

We obtain the result of Equation~\eqref{eq:edg_exp_inid_nocont_sym} by computing all the numerical constants.

\bigskip

\textbf{We finally prove (iii).} 
When $\Kquatren = O(1)$, we remark that \(\lambdatroisn\), \(\Ktroisn\), and \(\Ktroisntilde\) are bounded as well. Given the detailed analysis of $\RinidInt(\eps)$ carried out in Section~\ref{sec_bound_R1n} (in particular Equations~\eqref{eq:simplified_R1n_int_skewed} and~\eqref{eq:simplified_R1n_int_noskewed}), boundedness of the former moments ensures that $\RinidInt(\eps) = O(n^{-5/4})$ in general and $\RinidInt(\eps) = O(n^{-3/2})$ in the no-skewness case. 

We can also see (remember that \(\chi_1 \approx 0.099 \)) that $\Delta > 0$ for $n$ large enough when $\Kquatren = O(1)$. Consequently, for $n$ large enough, we can write in the general case
\begin{align*}
    & \frac{1.0253\Ktroisn}{6\pi\sqrt{n}} \Bigg\{ 0.5|\Delta|^{-3/2} \Indicator_{\{\Delta \neq 0\}} \times
    \bigg| \gamma(3/2, 4 \Delta n / \Ktroisntilde^2)
    - \gamma\big(3/2, 2\Delta ( \eps (n/K_{4,n})^{1/2} \wedge 2 n / \Ktroisntilde^2\big) \bigg| \\
    & \qquad\qquad\qquad + \Indicator_{\{\Delta = 0\}}
    \frac{ ( 2 \sqrt{n} / \Ktroisntilde )^3
    - (\sqrt{2\eps} (n/K_{4,n})^{1/4} \wedge 2 \sqrt{n} / \Ktroisntilde )^3 }{3} \Bigg\} \\
    &= \frac{1.0253\Ktroisn}{6\pi\sqrt{n}} \big\{ \Gamma\big(3/2, 2\Delta ( \eps (n/K_{4,n})^{1/2} \wedge 2 n / \Ktroisntilde^2) \big) - \Gamma(3/2, 4 \Delta n / \Ktroisntilde^2) \big\},
\end{align*}
and, in the no-skewness case,
\begin{align*}
    & \frac{1.0253\Kquatren}{6\pi n} \Bigg\{ 0.5|\Delta|^{-2} 
    \Indicator_{\{\Delta \neq 0\}} \times
    \bigg| \gamma(2, 4 \Delta n / \Ktroisntilde^2)
    - \gamma\big(2, 2\Delta ( \eps (n/K_{4,n})^{1/2} \wedge 2 n / \Ktroisntilde^2\big) \bigg| \\
    & \qquad\qquad\qquad + \Indicator_{\{\Delta = 0\}}
    \frac{ ( 2 \sqrt{n} / \Ktroisntilde )^4
    - (\sqrt{2\eps} (n/K_{4,n})^{1/4} \wedge 2 \sqrt{n} / \Ktroisntilde )^4 }{4} \Bigg\} \\
    &= \frac{1.0253\Kquatren}{6\pi n} \big\{ \Gamma\big(2, 2\Delta ( \eps (n/K_{4,n})^{1/2} \wedge 2 n / \Ktroisntilde^2) \big) - \Gamma(2, 4 \Delta n / \Ktroisntilde^2) \big\}.
\end{align*}

This reasoning enables us to obtain a difference of Gamma functions and therefore apply the asymptotic expansion $\Gamma(a,x) = x^{a-1} e^{-x} (1 + O((a-1) / x))$ which is valid for every fixed $a$ in the regime $x\to \infty$, see Equation~(6.5.32) in \cite{abramowitz85handbook}.
We also use this asymptotic expansion for the term
\begin{equation*}
    \frac{ |\lambdatroisn| \big( \Gamma( 3/2 , \sqrt{2\eps} (n/K_{4,n})^{1/4} \wedge 2 \sqrt{n} / \Ktroisntilde)
    - \Gamma( 3/2 , 2 \sqrt{n} / \Ktroisntilde) \big) }{\sqrt{n}}.
\end{equation*}
Consequently, we get the stated rate \(\rninidskew{1}(\eps) = O(n^{-5/4})\) in the general case and \(\rninidskew{1}(\eps) = O(n^{-3/2})\) in the no-skewness case. 

\end{proof}

\subsection{Proof of Theorem~\ref{thm:nocont_choiceEps} under Assumption~\ref{hyp:basic_as_iid}}
\label{ssec:proof:iid_nocont}

We present and prove a more general result, Theorem~\ref{thm:nocont_iid}, and choose $\eps = 0.1$ to recover Theorem~\ref{thm:nocont_choiceEps} under Assumption~\ref{hyp:basic_as_iid} 

\begin{theorem}[One-term Edgeworth expansion under Assumption~\ref{hyp:basic_as_iid}]
\label{thm:nocont_iid}

(i) Under Assumption~\ref{hyp:basic_as_iid}, for every $\eps \in (0,1/3)$ and every $n \geq 3$, we have the bound 
    \begin{align}
    \label{eq:edg_exp_iid_nocont}
        \DeltaE \leq & \frac{0.1995\Ktroisntilde}{\sqrt{n}} + \frac{1}{n}\bigg\{ 0.031\Ktroisntilde^2 + 0.327K_{4,n}\left(\frac{1}{12}+\frac{1}{4(1-3\eps)^2}\right) \nonumber \\
        & \quad\quad\quad\quad\quad\quad\quad\quad  + 0.054|\lambda_{3,n}|\Ktroisntilde + 0.037e_{3}(\eps)\lambda_{3,n}^2 \bigg\} + \rniidskew{1}(\eps), 
    \end{align}
    where $\rniidskew{1}(\eps)$ is given in Equation~\eqref{eq:def_r_1n_iid_eps} and $e_{3}(\eps) = e^{\eps^2/6 + \eps^2 / (2(1-3 \eps) )^2}$.
    
    \medskip
    
    (ii) If we further impose $\E[X_n^3]=0$, the upper bound reduces to
    \begin{align}
    \label{eq:edg_exp_iid_nocont_sym}
        \frac{0.1995\Ktroisntilde}{\sqrt{n}} + \frac{1}{n}\bigg\{ 0.031\Ktroisntilde^2 + 0.327K_{4,n}\left(\frac{1}{12}+\frac{1}{4(1-3\eps)^2}\right) \bigg\} + \rniidnoskew{1}(\eps), 
    \end{align}    
    where $\rniidnoskew{1}(\eps)$ is given in Equation~\eqref{eq:def_r_1n_iid_noskew_eps}. 
    
    \medskip
    
    (iii) Finally, when $K_{4,n}=O(1)$ as \(n \to \infty\), we obtain $\rniidskew{1}(\eps) = O(n^{-5/4})$ and $\rniidnoskew{1}(\eps) = O(n^{-2})$.
\end{theorem}

We use this result to finish the proof of Theorem~\ref{thm:nocont_choiceEps}, which corresponds to the case \( \eps = 0.1 \), by computing the numerical constants.
In particular, the computation of $e_{3}(0.1)$ gives the upper bound $e_{3}(0.1) \leq 1.0068 $.
Note that in the statement of Theorem~\ref{thm:nocont_choiceEps}, to obtain a more concise presentation, we control $e_{3}(0.1)$ from above by the slightly larger bound $1.0157$ used in the \inid{} case to upper bound $e_{1,n}(0.1)$. 

In this case, we obtain the bound \( \rniidskew{1}\) on \( \rniidskew{1}(0.1) \)
{\color{black}which is given in Equation~\eqref{eq:def_r_1n_iid_eps_eq_01},}
where \( \RiidIntSkew \) is explicitly defined in Equation~\eqref{eq:definition_R2n_int_skew_bound_when_eps_01}.

In the no-skewness case, the rest \( \rniidnoskew{1}(0.1) \) is bounded by the explicit expression {\color{black}given in Equation~\eqref{eq:def_r_1n_iid_eps_eq_01_noskew},}
where \(\RiidIntNoskew\) is defined in Equation~\eqref{eq:definition_R2n_int_noskew_bound_when_eps_01}.

\begin{proof}[Proof of Theorem~\ref{thm:nocont_iid}.]

The overall scheme of the proof is similar to that of Theorem~\ref{thm:nocont_inid} except for some improvements obtained in the \iid{} set-up.

\medskip

\textbf{We first prove (i).} We apply Lemma~\ref{lemma:smoothing} with
$F$ the cdf of $S_n$ and obtain
\begin{align*}
    \DeltaE
    &\leq \, \Omega_1(T, v, \tau)
    + \Omega_2(T) + \Omega_3(T, v, \tau)
    + \Omega_4(\tau \wedge T/\pi, T/\pi, T).
\end{align*}
Let $T = 2 \pi \sqrt{n} / \Ktroisntilde$, $v = \lambda_{3,n} / \sqrt{n}$ and $\tau = \sqrt{2\eps} (n/K_{4,n})^{1/4}$. We combine Lemma~\ref{lemma:bound_omega_1} (control of $\Omega_1$), 
Equation~\eqref{eq:bound_omega2_nocont} (control of $\Omega_2$), 
Lemma~\ref{lemma:bound_omega3} (control of $\Omega_3$), 
Lemma~\ref{lemma:bound_I32}(iii) (control of $\Omega_4$) to get
\begin{align*}
    \DeltaE
    &\leq \frac{1.2533}{T}
    + \frac{0.3334 \, |\lambda_{3,n}|}{T\sqrt{n}}
    + \frac{14.1961}{T^4}
    + \frac{4.3394 \, |\lambdatroisn|}{T^3 \sqrt{n}} \nonumber \\
    &+ \frac{|\lambdatroisn|
    \big( \Gamma( 3/2 , \tau \wedge T/\pi)
    - \Gamma( 3/2 , T/\pi) \big)}{\sqrt{n}}
    + \frac{67.0415}{T^4} + \frac{1.2187}{T^2} \\
    &+ \frac{0.327 \, K_{4,n}}{n}
    \left(\frac{1}{12}+\frac{1}{4(1-3\eps)^2}\right)
    + \frac{1.306 \, e_{2,n}(\eps) \lambda_{3,n}^2}{36n} \nonumber \\
    &+ \frac{1.0253}{\pi}
    \int_0^{\tau \wedge T/\pi} u e^{-u^2/2} \Riid(u,\eps) du 
    + \frac{K_{3,n}}{3\sqrt{n}} J_{3} \big(3, \tau \wedge T/\pi, T/\pi , T/\pi , T \big).
\end{align*}
Bounding \((1.0253/\pi)\times\int_0^{\tau \wedge T/\pi} u e^{-u^2/2} \Riid(u,\eps)\) by \(\RiidInt(\eps) := (1.0253/\pi)\times\int_0^{+ \infty} u e^{-u^2/2} \Riid(u,\eps)\), bounding $J_3$ by Lemma~\ref{lemma:bound_J3p}, and replacing $T$ and $\tau$ by their values, we obtain
\begin{align*}
    \DeltaE
    &\leq
    \frac{1.2533 \Ktroisntilde}{2 \pi \sqrt{n}}
    + \frac{0.3334 |\lambdatroisn| \Ktroisntilde}{2 \pi n}
    + \frac{1.2187 \Ktroisntilde^2}{4 \pi n}
    + \frac{0.327 \Kquatren}{n} \left(\frac{1}{12}+\frac{1}{4(1-3\eps)^2}\right) \\
    & +
    \frac{1.306 e_3(\eps) \lambdatroisn^2}{36 n} 
    + \rniidskew{1}(\eps),
\end{align*}
where
\begin{align}
    \rniidskew{1}(\eps) & :=
    \frac{(14.1961 + 67.0415) \, \Ktroisntilde^4}{16\pi^4 n^2}
    + \frac{4.3394 \, |\lambdatroisn| \Ktroisntilde^3}{8 \pi^3 n^2}
    + \RiidInt(\eps) \nonumber \\
    & + \frac{1.306 \big( e_{2,n}(\eps) - e_3(\eps) \big) \lambdatroisn^2}{36 n} \nonumber \\
    &+ \frac{ |\lambdatroisn| \big( \Gamma( 3/2 , \sqrt{2\eps} (n/K_{4,n})^{1/4} \wedge 2 \sqrt{n} / \Ktroisntilde)
    - \Gamma( 3/2 , 2 \sqrt{n} / \Ktroisntilde) \big) }{\sqrt{n}} \nonumber \\
    &+ 
    \frac{1.0253 \times 2^{5/2} \, \Ktroisn}{3 \pi \sqrt{n}}
    \bigg( 
    \Gamma \Big( 3/2, \big\{ \sqrt{2\eps} (n/\Kquatren)^{1/4} \wedge 2\sqrt{n}/\Ktroisntilde \big\}^2/8 \Big) \nonumber 
    \\
    & \qquad \qquad \qquad \qquad \qquad - \Gamma \big( 3/2,
    4 n / (8 \Ktroisntilde^2)
    \big) \bigg).
    \label{eq:def_r_1n_iid_eps}
\end{align}

We obtain the result of Equation~\eqref{eq:edg_exp_iid_nocont} by computing the numerical constants.

\bigskip

\textbf{We now prove (ii).} In the no-skewness case, namely when \(\E[X_n^3] = 0\), the start of the proof is identical except that Lemma~\ref{lemma:bound_I32}(iv) is used in lieu of Lemma~\ref{lemma:bound_I32}(iii) to control $\Omega_4$.
This yields
\begin{align*}
    \DeltaE
    &\leq \frac{1.2533}{T}
    + \frac{14.1961}{T^4} + \frac{67.0415}{T^4} + \frac{1.2187}{T^2} + \frac{0.327 \, K_{4,n}}{n}
    \left(\frac{1}{12}+\frac{1}{4(1-3\eps)^2}\right) \\
    &+ \RiidInt(\eps) + \frac{K_{4,n}}{3n} J_3 \big(4, \tau \wedge T/\pi, T/\pi , T/\pi , T \big).
\end{align*}
Bounding $J_3$ by Lemma~\ref{lemma:bound_J3p} and
replacing $T$ and $\tau$ by their values, we obtain
\begin{align*}
    \DeltaE
    &\leq \frac{1.2533 \, \Ktroisntilde}{2\pi \sqrt{n}} + \frac{1.2187 \, \Ktroisntilde^2}{4\pi n} + \frac{0.327 K_{4,n}}{n}\left(\frac{1}{12}
    + \frac{1}{4(1-3\eps)^2}\right) + \rniidnoskew{1}(\eps)
\end{align*}
where
\begin{align}
    \rniidnoskew{1}(\eps)
    &:= \frac{(14.1961 + 67.0415) \, \Ktroisntilde^4}{16\pi^4 n^2} + \RiidInt(\eps) \nonumber \\
    &+ \frac{16 \times 1.0253\Kquatren}{3\pi n} 
    \bigg(
    \Gamma \big(2, \big\{ \sqrt{2\eps} (n/\Kquatren)^{1/4} \wedge 2\sqrt{n}/\Ktroisntilde \big\}^2/8 \big) \nonumber \\
    &\qquad\qquad\qquad\qquad \quad - \Gamma \big( 2,
    4 n / (8 \Ktroisntilde^2)
    \big) \bigg).
    \label{eq:def_r_1n_iid_noskew_eps}
\end{align}

We obtain the result of Equation~\eqref{eq:edg_exp_iid_nocont_sym} by computing all the numerical constants.

\bigskip

\textbf{We finally prove (iii).} Following the line of proof as in Section~\ref{ssec:proof:inid_nocont}, we can prove that $\Kquatren = O(1)$ ensures the standardized moments \(\lambdatroisn\), \(\Ktroisn\), and \(\Ktroisntilde\) are bounded as well.
Given the detailed analysis of $\RiidInt(\eps)$ carried out in Section~\ref{sec_bound_R2n} (in particular Equation~\eqref{eq:simplified_R2n_int}), boundedness of the former moments ensures that $\RiidInt(\eps) = O(n^{-3/2})$ in general and $\RiidInt(\eps) = O(n^{-2})$ in the no-skewness case.

\medskip

From the definitions of \(e_{2,n}\) and \(e_{3}\) in Equations~\eqref{eq:definition_e2n} and~\eqref{eq:definition_e3}, we note that the term 
\begin{equation*}
    \frac{1.306 \big( e_{2,n}(\eps) - e_3(\eps) \big) \lambdatroisn^2}{36 n}
    = O(n^{-5/4}).
\end{equation*}

Applying the asymptotic expansion $\Gamma(a,x) = x^{a-1} e^{-x} (1 + O((a-1) / x))$, we can claim
\begin{align*}
    & \frac{ |\lambdatroisn| \big( \Gamma( 3/2 , \sqrt{2\eps} (n/K_{4,n})^{1/4} \wedge 2 \sqrt{n} / \Ktroisntilde)
    - \Gamma( 3/2 , 2 \sqrt{n} / \Ktroisntilde) \big) }{\sqrt{n}} \\
    & + \frac{1.0253 \times 2^{5/2} \, \Ktroisn}{3 \pi \sqrt{n}}
    \bigg( \Gamma \big( 3/2, \big\{ \sqrt{2\eps} (n/\Kquatren)^{1/4} \wedge 2\sqrt{n}/\Ktroisntilde \big\}^2/8 \big) - \Gamma \big( 3/2,
    4 n / (8 \Ktroisntilde^2)
    \big) \bigg) \\
    &= o\big( n^{-5/4} \big),
\end{align*}
and
\begin{align*}
    \frac{16 \times 1.0253\Kquatren}{3\pi n} 
    \bigg( \Gamma \big(2, \big\{ \sqrt{2\eps} (n/\Kquatren)^{1/4} \wedge 2\sqrt{n}/\Ktroisntilde \big\}^2/8 \big) - \Gamma \big( 2,
    4 n / (8 \Ktroisntilde^2)
    \big) \bigg) = o\big( n^{2} \big).
\end{align*}
As a result, we obtain \(\rniidskew{1}(\eps) = O(n^{-5/4})\) in general and \(\rniidnoskew{1}(\eps) = O(n^{-2})\) in the no-skewness case, as claimed.
\end{proof}

\subsection{Proof of Theorem~\ref{thm:cont_choiceEps} under Assumption~\ref{hyp:basic_as_inid}}
\label{ssec:proof:inid_cont}

We use Theorem~\ref{thm:cont_inid}, proved below, with the choice $\eps = 0.1$. Recall that $t_1^* = \theta_1^*/(2\pi) \approx 0.64$ where $\theta_1^*$ is the unique root in $(0,2\pi)$ of the equation $\theta^2+2\theta\sin(\theta)+6(\cos(\theta)-1)=0.$ Recall also that $a_n := 2t_1^*\pi\sqrt{n}/\Ktroisntilde \wedge 16\pi^3n^2/\Ktroisntilde^4,$ and $b_n := 16\pi^4n^2/\widetilde{K}_{3,n}^4$.

\begin{theorem}[Alternative one-term Edgeworth expansion under Assumption~\ref{hyp:basic_as_inid}]
\label{thm:cont_inid}

    (i) Under Assumption~\ref{hyp:basic_as_inid}, for every $\eps \in (0,1/3)$ and every $n \geq 1$, we have the bound 
    \begin{align}
        \DeltaE
        &\leq \frac{1}{n} \left\{ 0.327 \, K_{4,n} \left(\frac{1}{12}
        + \frac{1}{4(1-3\eps)^2} \right)
        + 0.037 \, e_{1,n}(\eps)\lambda_{3,n}^2 \right\} \nonumber \\
        & + \frac{1.0253}{\pi} \int_{a_n}^{b_n} \frac{|\caracfsum(t)|}{t}dt 
        + \rninidskew{2}(\eps), 
    \label{eq:edg_exp_inid_thm_3_1}
    \end{align}
    where $\rninidskew{2}(\eps)$ is given in Equation~\eqref{eq:def_r_2n_inid_eps}.
    
    \medskip
    
    (ii) If we further impose $\E[X_n^3]=0$, the upper bound reduces to
    \begin{align}
    \label{eq:edg_exp_inid_thm_3_1_sym}
        \frac{0.327K_{4,n}}{n}
        \left( \frac{1}{12} + \frac{1}{4(1-3\eps)^2} \right)
        + \frac{1.0253}{\pi} \int_{a_n}^{b_n} \frac{|\caracfsum(t)|}{t}dt
        + \rninidnoskew{2}(\eps),
    \end{align}    
    where $\rninidnoskew{2}(\eps)$ is given in Equation~\eqref{eq:def_r_2n_inid_eps_noskew}. 
    
    \medskip
    
    (iii) Finally, when $K_{4,n}=O(1)$ as \(n \to \infty\), we obtain $\rninidskew{2}(\eps) = O(n^{-5/4})$ and $\rninidnoskew{2}(\eps) = O(n^{-3/2})$.
\end{theorem}

Using Theorem~\ref{thm:cont_inid}, we can finish the proof of Theorem~\ref{thm:nocont_choiceEps} by setting \(\eps = 0.1\), computing the numerical constants and using the upper bounds on $\RinidInt(0.1)$ computed in Section~\ref{sec_bound_R1n}. In particular, \( \rninidskew{2}(0.1) \) is bounded by the explicit expression {\color{black} given in Equation~\eqref{eq:def_r_2n_inid_eps_eq_01}}.
while \( \rninidnoskew{2}(0.1) \) is {\color{black}bounded by the quantity $\rninidnoskew{2}$ defined in Equation~\eqref{eq:def_r_2n_inid_eps_eq_01_noskew}.}

\begin{proof}[Proof of Theorem~\ref{thm:cont_inid}.]
\textbf{We first prove (i).}
We apply Lemma~\ref{lemma:smoothing} with
$F$ the cdf of $S_n$ and obtain
\begin{align*}
    \DeltaE
    &\leq \, \Omega_1(T, v, \tau)
    + \Omega_2(T) + \Omega_3(T, v, \tau)
    + \Omega_4(\tau \wedge T/\pi, T/\pi, T).
\end{align*}
Let $T = 16 \pi^4 n^2 / \Ktroisntilde^4$, $v = \lambda_{3,n} / \sqrt{n}$ and $\tau = \sqrt{2\eps} (n/K_{4,n})^{1/4}$.
We combine Lemma~\ref{lemma:bound_omega_1} (control of $\Omega_1$),
Lemma~\ref{lemma:bound_omega3} (control of $\Omega_3$), 
Lemma~\ref{lemma:omega4_smooth} and then~\ref{lemma:bound_I32}(i) (control of $\Omega_4$) to get
\begin{align*}
    \DeltaE
    &\leq \frac{1.2533}{T}
    + \frac{0.3334 \, |\lambda_{3,n}|}{T\sqrt{n}}
    + \frac{14.1961}{T^4}
    + \frac{4.3394 \, |\lambdatroisn|}{T^3 \sqrt{n}} \nonumber \\
    &+ \frac{|\lambdatroisn|
    \big( \Gamma( 3/2 , \tau \wedge T/\pi)
    - \Gamma( 3/2 , T/\pi) \big)}{\sqrt{n}} + \frac{1.0253}{\pi} \int_{T / \pi}^T \frac{| \caracfsum(u) |}{u} du \nonumber \\ 
    &+ \frac{0.327 \, K_{4,n}}{n}
    \left(\frac{1}{12}+\frac{1}{4(1-3\eps)^2}\right)
    + \frac{1.306 \, e_{1,n}(\eps) \lambda_{3,n}^2}{36n} \nonumber \\
    &+ \frac{1.0253}{\pi}
    \int_0^{\tau \wedge T/\pi} u e^{-u^2/2} \Rinid(u,\eps) du \nonumber \\
    & + \Big| \Omega_4(\sqrt{2\eps}(n/K_{4,n})^{1/4} \wedge T/\pi, T^{1/4}/\pi \wedge T/\pi, T) \Big| \\
    & + \frac{1.0253}{2\pi} \left( \Gamma\left( 0 , T^{1/2}(1-4\pi\chi_1t_1^*) / (2\pi^2) \right) - \Gamma\left( 0 , t_1^{*2}T^{1/2}(1-4\pi\chi_1t_1^*) / 2 \right) \right) \nonumber \\
    & + \frac{1.0253}{\pi} \int_{t_1^*T^{1/4} \wedge T/\pi}^{T/\pi}
        \frac{|\caracfsum(u)|}{u}du
        + \frac{1.0253}{4\pi}
        \big| \Gamma(0, T^2 / 2\pi) - \Gamma(0, T^{1/2} / 2\pi^2) \big| \nonumber \displaybreak[0] \\
    & \leq \frac{1.2533}{T}
    + \frac{0.3334 \, |\lambda_{3,n}|}{T\sqrt{n}}
    + \frac{14.1961}{T^4}
    + \frac{4.3394 \, |\lambdatroisn|}{T^3 \sqrt{n}} \nonumber \\
    &+ \frac{|\lambdatroisn|
    \big( \Gamma( 3/2 , \tau \wedge T/\pi)
    - \Gamma( 3/2 , T/\pi) \big)}{\sqrt{n}} + \frac{1.0253}{\pi} \int_{t_1^*T^{1/4} \wedge T / \pi}^T \frac{| \caracfsum(u) |}{u} du \nonumber \\ 
    &+ \frac{0.327 \, K_{4,n}}{n}
    \left(\frac{1}{12}+\frac{1}{4(1-3\eps)^2}\right)
    + \frac{1.306 \, e_{1,n}(\eps) \lambda_{3,n}^2}{36n} \nonumber \\
    &+ \frac{1.0253}{\pi}
    \int_0^{\tau \wedge T/\pi} u e^{-u^2/2} \Rinid(u,\eps) du \nonumber \\
    & + \frac{K_{3,n}}{3\sqrt{n}} \Big| J_2 \big(3, \tau \wedge T/\pi, T^{1/4}/\pi \wedge T/\pi, T^{1/4}/\pi , T \big) \Big| \\
    & + \frac{1.0253}{\pi} \left( \Gamma\left( 0 , (T^{1/2}\wedge T^2)(1-4\pi\chi_1t_1^*) / (2\pi^2) \right) \right. \nonumber \\
    & \left. \qquad\qquad\quad - \Gamma\left( 0 , (t_1^{*2}T^{1/2}\wedge T^2/\pi^2)(1-4\pi\chi_1t_1^*) / 2 \right) \right) \nonumber \\
    & + \frac{1.0253}{\pi} \left( \Gamma\left( 0 , (T^{1/2} \wedge T^2) / (2\pi^2) \right) - \Gamma\left( 0 , T^2/(2\pi^2) \right) \right).
\end{align*}

Bounding \((1.0253/\pi)\times\int_0^{\tau \wedge T/\pi} u e^{-u^2/2} \Rinid(u,\eps)\) by \(\RinidInt(\eps) := (1.0253/\pi)\times\int_0^{+ \infty} u e^{-u^2/2} \Rinid(u,\eps)\), bounding $J_2$ by Lemma~\ref{lemma:bound_J2p}, and replacing $T$ and $\tau$ by their values, we obtain
\begin{align*}
    \DeltaE
    &\leq \frac{0.327 K_{4,n}}{n}\left(\frac{1}{12}
    + \frac{1}{4(1-3\eps)^2}\right) + \frac{1.306 \, e_{1,n}(\eps) \lambda_{3,n}^2}{36n} + \frac{1.0253}{\pi} \int_{a_n}^{b_n} \frac{| \caracfsum(u) |}{u} du \\& + \rninidskew{2}(\eps),
\end{align*}
where $a_n := 2t_1^*\pi\sqrt{n}/\Ktroisntilde \wedge 16\pi^3n^2/\Ktroisntilde^4$ and $b_n := 16\pi^4n^2/\Ktroisntilde^4$,
\begin{align}
    \rninidskew{2}(\eps) & := \frac{1.2533 \, \Ktroisntilde^4}{16\pi^4n^2}
    + \frac{0.3334 \, \Ktroisntilde^4 \, |\lambda_{3,n}|}{16\pi^4n^{5/2}}
    + \frac{14.1961 \, \Ktroisntilde^{16}}{(2\pi)^{16}n^8}
    + \frac{4.3394 \, |\lambdatroisn| \, \Ktroisntilde^{12}}{(2\pi)^{12} n^{13/2}} \nonumber \\
    & + \frac{|\lambdatroisn|
    \big( \Gamma( 3/2 , \sqrt{2\eps} (n/K_{4,n})^{1/4} \wedge 16\pi^3n^2/\Ktroisntilde^4)
    - \Gamma( 3/2 , 16\pi^3n^2/\Ktroisntilde^4) \big)}{\sqrt{n}} \nonumber \\
    & + \RinidInt(\eps) \nonumber \\
    & + \frac{1.0253\Ktroisn}{6\pi\sqrt{n}} \Bigg\{ 0.5|\Delta|^{-3/2}\Indicator_{\{\Delta \neq 0\}} \times \bigg|\gamma(3/2, 2^8\pi^6 \Delta n^4 / \Ktroisntilde^8) \nonumber \\
    & \qquad\qquad\qquad - \gamma\big(3/2, \Delta ( 2\eps (n/K_{4,n})^{1/2} \wedge 2^8\pi^6 n^4 / \Ktroisntilde^8 ) \big) \bigg| \nonumber \\
    & \qquad\qquad\qquad\;\; + \Indicator_{\{\Delta = 0\}}
    \frac{(16 \pi^3 n^2 / \Ktroisntilde^4)^3
    - (\sqrt{2\eps} (n/K_{4,n})^{1/4} \wedge 16 \pi^3 n^2 / \Ktroisntilde^4)^3 }{3} \Bigg\} \nonumber \\  
    & + \frac{1.0253}{\pi} \left( \Gamma\left( 0 , (4\pi^2n/\Ktroisntilde^2 \wedge 144\pi^8n^4/\Ktroisntilde^8)(1-4\pi\chi_1t_1^*) / (2\pi^2) \right) \right. \nonumber \\
    & \left. \qquad\qquad\quad - \Gamma\left( 0 , (4t_1^{*2}\pi^2n/\Ktroisntilde^2 \wedge 144\pi^6n^4/\Ktroisntilde^8)(1-4\pi\chi_1t_1^*) / 2 \right) \right) \nonumber \\
    & + \frac{1.0253}{\pi} \left( \Gamma\left( 0 , (4\pi^2n/\Ktroisntilde^2 \wedge 144\pi^8n^4/\Ktroisntilde^8) / (2\pi^2) \right) - \Gamma\left( 0 , 144\pi^6n^4/\Ktroisntilde^8 \right) \right),
    \label{eq:def_r_2n_inid_eps}    
\end{align}
and $\Delta := (1 - 4 \chi_1 - \sqrt{K_{4,n}/n}) / 2$.

\bigskip

\textbf{We now prove (ii).} The proof is exactly the same as the one we have used in (i) just above, except that Lemma~\ref{lemma:bound_I32}(i) is replaced with Lemma~\ref{lemma:bound_I32}(ii). Consequently,
\begin{align*}
    \DeltaE
    &\leq \frac{0.327 K_{4,n}}{n}\left(\frac{1}{12}
    + \frac{1}{4(1-3\eps)^2}\right) + \frac{1.0253}{\pi} \int_{a_n}^{b_n} \frac{| \caracfsum(u) |}{u} du + \rninidnoskew{2}(\eps),
\end{align*}
where
\begin{align}
    \rninidnoskew{2}(\eps) & := \frac{1.2533 \, \Ktroisntilde^4}{16\pi^4n^2}
    + \frac{14.1961 \, \Ktroisntilde^{16}}{(2\pi)^{16}n^8} + \RinidInt(\eps) \nonumber \\
    & + \frac{1.0253\Kquatren}{6\pi n} \Bigg\{ 0.5|\Delta|^{-2}\Indicator_{\{\Delta \neq 0\}} \times \bigg|\gamma(2, 2^8\pi^6 \Delta n^4 / \Ktroisntilde^8) \nonumber \\
    & \qquad\qquad\qquad\;\; - \gamma\big(2, \Delta ( 2\eps (n/K_{4,n})^{1/2} \wedge 2^8\pi^6 n^4 / \Ktroisntilde^8 ) \big) \bigg| \nonumber \\
    & \qquad\qquad\qquad + \Indicator_{\{\Delta = 0\}}
    \frac{(16 \pi^3 n^2 / \Ktroisntilde^4)^4
    - (\sqrt{2\eps} (n/K_{4,n})^{1/4} \wedge 16 \pi^3 n^2 / \Ktroisntilde^4)^4}{4}
     \Bigg\} \nonumber \\  
    & + \frac{1.0253}{\pi} \left( \Gamma\left( 0 , (4\pi^2n/\Ktroisntilde^2 \wedge 144\pi^8n^4/\Ktroisntilde^8)(1-4\pi\chi_1t_1^*) / (2\pi^2) \right) \right. \nonumber \\
    & \left. \qquad\qquad\quad - \Gamma\left( 0 , (4t_1^{*2}\pi^2n/\Ktroisntilde^2 \wedge 144\pi^6n^4/\Ktroisntilde^8)(1-4\pi\chi_1t_1^*) / 2 \right) \right) \nonumber \\
    & + \frac{1.0253}{\pi} \left( \Gamma\left( 0 , (4\pi^2n/\Ktroisntilde^2 \wedge 144\pi^8n^4/\Ktroisntilde^8) / (2\pi^2) \right) - \Gamma\left( 0 , 144\pi^6n^4/\Ktroisntilde^8 \right) \right).
    \label{eq:def_r_2n_inid_eps_noskew}    
\end{align}

\bigskip

\textbf{We finally prove (iii).} The reasoning is completely analogous to the proof of Theorem~\ref{thm:nocont_inid}.(iii). Leading terms in $\rninidskew{2}(\eps)$ (\textit{resp.} $\rninidnoskew{2}(\eps)$) stem from $\RinidInt(\eps)$. This term appeared in $\rninidskew{1}(\eps)$ and $\rninidnoskew{1}(\eps)$ and we showed $\RinidInt(\eps) = O(n^{-5/4})$ in the general case and $\RinidInt(\eps) = O(n^{-3/2})$ in the no-skewness case.
\end{proof}

\subsection{Proof of Theorem~\ref{thm:cont_choiceEps} under Assumption~\ref{hyp:basic_as_iid}}
\label{ssec:proof:iid_cont}

We use Theorem~\ref{thm:cont_iid}, proved below, with the choice $\eps = 0.1$. Recall that $t_1^* = \theta_1^*/(2\pi) \approx 0.64$ where $\theta_1^*$ is the unique root in $(0,2\pi)$ of the equation $\theta^2+2\theta\sin(\theta)+6(\cos(\theta)-1)=0.$ Recall also that $a_n := 2t_1^*\pi\sqrt{n}/\Ktroisntilde \wedge 16\pi^3n^2/\Ktroisntilde^4,$ and $b_n := 16\pi^4n^2/\widetilde{K}_{3,n}^4$.

\begin{theorem}[Alternative one-term Edgeworth expansion under Assumption~\ref{hyp:basic_as_iid}]
\label{thm:cont_iid}

(i) Under Assumption~\ref{hyp:basic_as_iid}, for every $\eps \in (0,1/3)$ and every $n \geq 3$, we have the bound 
    \begin{align}
        \DeltaE
        &\leq \frac{1}{n}\left\{ 0.327K_{4,n}\left(\frac{1}{12}
        +  \frac{1}{4(1-3\eps)^2}\right)
        + 0.037 e_3(\eps)\lambda_{3,n}^2 \right\} \nonumber \\
        &+ \frac{1.0253}{\pi}\int_{a_n}^{b_n}\frac{|\caracfsum(t)|}{t}dt
        + \rniidskew{2}(\eps),
    \end{align}
    where $\rniidskew{2}(\eps)$ is given in Equation~\eqref{eq:def_r_2n_iid_eps} and $e_{3}(\eps) = e^{\eps^2/6 + \eps^2 / (2(1-3 \eps) )^2}$.
    
    \medskip
    
    (ii) If we further impose $\E[X_n^3]=0$, the upper bound reduces to
    \begin{align}
        \frac{0.327K_{4,n}}{n} \left(\frac{1}{12}
        + \frac{1}{4(1-3\eps)^2}\right)
        + \frac{1.0253}{\pi} \int_{a_n}^{b_n} \frac{|\caracfsum(t)|}{t}dt
        + \rniidnoskew{2}(\eps),
    \end{align}    
    where $\rniidnoskew{2}(\eps)$ is given in Equation~\eqref{eq:def_r_2n_iid_eps_noskew}. 
    
    \medskip
    
    (iii) Finally, when $K_{4,n} = O(1)$ as \(n \to \infty\), we obtain $\rniidskew{2}(\eps) = O(n^{-5/4})$ and $\rniidnoskew{2}(\eps) = O(n^{-2})$.
\end{theorem}

We can use this result to wrap up the proof of Theorem~\ref{thm:cont_choiceEps}. We set $\eps = 0.1$, use the upper bound $\RiidInt(0.1) \leq \RiidIntSkew$ in the general case (\textit{resp.} $\RiidInt(0.1) \leq \RiidIntNoskew$ in the no-skewness case) and compute all the numerical constants depending on $\eps$. {\color{black} This gives us the explicit expression written in Equation~\eqref{eq:def_r_2n_iid_eps_eq_01} as an upper bound on $\rniidskew{2}(0.1)$.}
{\color{black}In the same way,  $\rniidnoskew{2}(0.1)$ is bounded by the value $\rniidnoskew{2}$ given in Equation~\eqref{eq:def_r_2n_iid_eps_eq_01_noskew}.}

\begin{proof}[Proof of Theorem~\ref{thm:cont_iid}.]
\textbf{We first prove (i).} The proof is similar to that of Theorem~\ref{thm:cont_inid} except that we use Lemma~\ref{lemma:bound_I32}(iii) instead of Lemma~\ref{lemma:bound_I32}(i) (and the second part of Lemma~\ref{lemma:bound_omega3}). This leads to
\begin{align*}
    \DeltaE
    & \leq \frac{1.2533}{T}
    + \frac{0.3334 \, |\lambda_{3,n}|}{T\sqrt{n}}
    + \frac{14.1961}{T^4}
    + \frac{4.3394 \, |\lambdatroisn|}{T^3 \sqrt{n}} \nonumber \\
    &+ \frac{|\lambdatroisn|
    \big( \Gamma( 3/2 , \tau \wedge T/\pi)
    - \Gamma( 3/2 , T/\pi) \big)}{\sqrt{n}} + \frac{1.0253}{\pi} \int_{t_1^*T^{1/4} \wedge T / \pi}^T \frac{| \caracfsum(u) |}{u} du \nonumber \\ 
    &+ \frac{0.327 \, K_{4,n}}{n}
    \left(\frac{1}{12}+\frac{1}{4(1-3\eps)^2}\right)
    + \frac{1.306 \, e_{2,n}(\eps) \lambda_{3,n}^2}{36n} \nonumber \\
    &+ \frac{1.0253}{\pi}
    \int_0^{\tau \wedge T/\pi} u e^{-u^2/2} \Riid(u,\eps) du \nonumber \\
    & + \frac{K_{3,n}}{3\sqrt{n}} \Big| J_3 \big(3, \tau \wedge T/\pi, T^{1/4}/\pi \wedge T/\pi, T^{1/4}/\pi , T \big) \Big| \\
    & + \frac{1.0253}{\pi} \left( \Gamma\left( 0 , (T^{1/2}\wedge T^2)(1-4\pi\chi_1t_1^*) / (2\pi^2) \right) \right. \\
    & \left. \qquad\qquad\quad - \Gamma\left( 0 , (t_1^{*2}T^{1/2}\wedge T^2/\pi^2)(1-4\pi\chi_1t_1^*) / 2 \right) \right) \nonumber \\
    & + \frac{1.0253}{\pi} \left( \Gamma\left( 0 , (T^{1/2} \wedge T^2) / (2\pi^2) \right) - \Gamma\left( 0 , T^2/(2\pi^2) \right) \right).
\end{align*}

Using Lemma~\ref{lemma:bound_J3p} instead of Lemma~\ref{lemma:bound_J2p}, we arrive at
\begin{align*}
    \DeltaE
    &\leq \frac{0.327 K_{4,n}}{n}\left(\frac{1}{12}
    + \frac{1}{4(1-3\eps)^2}\right) + \frac{1.306 \, e_3(\eps) \lambda_{3,n}^2}{36n} + \frac{1.0253}{\pi} \int_{a_n}^{b_n} \frac{| \caracfsum(u) |}{u} du \\& + \rniidskew{2}(\eps),
\end{align*}
where
\begin{align}
    \rniidskew{2}(&\eps) := \frac{1.2533 \, \Ktroisntilde^4}{16\pi^4n^2}
    + \frac{0.3334 \, \Ktroisntilde^4 \, |\lambda_{3,n}|}{16\pi^4n^{5/2}}
    + \frac{14.1961 \, \Ktroisntilde^{16}}{(2\pi)^{16}n^8}
    + \frac{4.3394 \, |\lambdatroisn| \, \Ktroisntilde^{12}}{(2\pi)^{12} n^{13/2}} \nonumber \\
    & + \frac{|\lambdatroisn|
    \big( \Gamma( 3/2 , \sqrt{2\eps} (n/K_{4,n})^{1/4} \wedge 16\pi^3n^2/\Ktroisntilde^4)
    - \Gamma( 3/2 , 16\pi^3n^2/\Ktroisntilde^4) \big)}{\sqrt{n}} + \RiidInt(\eps) \nonumber \\
    & + \frac{1.0253 \times 2^{5/2} \, \Ktroisn
        \big| \Gamma(3/2, 2^5\pi^6n^4/\Ktroisntilde^8) - \Gamma(3/2, \eps\sqrt{n/(16\Kquatren)} \wedge 2^5\pi^6n^4/\Ktroisntilde^8) \big|}{3\pi\sqrt{n}} \nonumber \\  
    & + \frac{1.306 \big( e_{2,n}(\eps) - e_3(\eps) \big) \lambdatroisn^2}{36 n} \nonumber \\
    & + \frac{1.0253}{\pi} \left( \Gamma\left( 0 , (4\pi^2n/\Ktroisntilde^2 \wedge 144\pi^8n^4/\Ktroisntilde^8)(1-4\pi\chi_1t_1^*) / (2\pi^2) \right) \right. \nonumber \\
    & \left. \qquad\qquad\quad - \Gamma\left( 0 , (4t_1^{*2}\pi^2n/\Ktroisntilde^2 \wedge 144\pi^6n^4/\Ktroisntilde^8)(1-4\pi\chi_1t_1^*) / 2 \right) \right) \nonumber \\
    & + \frac{1.0253}{\pi} \left( \Gamma\left( 0 , (4\pi^2n/\Ktroisntilde^2 \wedge 144\pi^8n^4/\Ktroisntilde^8) / (2\pi^2) \right) - \Gamma\left( 0 , 144\pi^6n^4/\Ktroisntilde^8 \right) \right).
    \label{eq:def_r_2n_iid_eps}    
\end{align}

\bigskip

\textbf{We now prove (ii).} The proof is the same as that of Result (i), except that we use Lemma~\ref{lemma:bound_I32}(iv) instead of Lemma~\ref{lemma:bound_I32}(iii). We conclude
\begin{align*}
    \DeltaE
    &\leq \frac{0.327 K_{4,n}}{n}\left(\frac{1}{12}
    + \frac{1}{4(1-3\eps)^2}\right) + \frac{1.0253}{\pi} \int_{a_n}^{b_n} \frac{| \caracfsum(u) |}{u} du + \rniidnoskew{2}(\eps),
\end{align*}
where
\begin{align}
    \rniidnoskew{2}(\eps)
    & := \frac{1.2533 \, \Ktroisntilde^4}{16\pi^4n^2}
    + \frac{14.1961 \, \Ktroisntilde^{16}}{(2\pi)^{16}n^8} + \RiidInt(\eps) \nonumber \\
    & + \frac{16 \times 1.0253 \, \Ktroisn
        \big| \Gamma(2, 2^5\pi^6n^4/\Ktroisntilde^8) - \Gamma(2, \eps\sqrt{n/(16\Kquatren)} \wedge 2^5\pi^6n^4/\Ktroisntilde^8) \big|}{3\pi n} \nonumber \\  
    & + \frac{1.0253}{\pi} \left( \Gamma\left( 0 , (4\pi^2n/\Ktroisntilde^2 \wedge 144\pi^8n^4/\Ktroisntilde^8)(1-4\pi\chi_1t_1^*) / (2\pi^2) \right) \right. \nonumber \\
    & \left. \qquad\qquad\quad - \Gamma\left( 0 , (4t_1^{*2}\pi^2n/\Ktroisntilde^2 \wedge 144\pi^6n^4/\Ktroisntilde^8)(1-4\pi\chi_1t_1^*) / 2 \right) \right) \nonumber \\
    & + \frac{1.0253}{\pi} \left( \Gamma\left( 0 , (4\pi^2n/\Ktroisntilde^2 \wedge 144\pi^8n^4/\Ktroisntilde^8) / (2\pi^2) \right) - \Gamma\left( 0 , 144\pi^6n^4/\Ktroisntilde^8 \right) \right).
    \label{eq:def_r_2n_iid_eps_noskew}  
\end{align}

\bigskip

\textbf{We finally prove (iii).} $\RiidInt(\eps)$ is the leading term in both $\rniidskew{2}(\eps)$ and $\rniidnoskew{2}(\eps)$. In the proof of Theorem~\ref{thm:nocont_iid}, $\RiidInt(\eps)$ was shown to be of order $n^{-5/4}$ in general and $n^{-2}$ in the no-skewness case.
\end{proof}

\subsection{Proof of Lemma~\ref{lemma:smoothing}}
\label{proof:lemma:smoothing}

Let us denote by ``$\vp \int \,$'' Cauchy's principal value, defined by
$$\vp \int_{-a}^a f(u)du := \lim_{x \to 0, \, x > 0} \int_{-a}^{-x} f(u) du + \int_{x}^a f(u) du,$$
where $f$ is a measurable function on $[-a,a] \backslash \{0\}$ for a given $a > 0$.
In the following, we use the following inequalities, which are due to \cite{prawitz1972}
\begin{align*}
    \lim_{y \to x, \, y > x} F(y)
    &\leq \frac{1}{2} +
    \vp \int_{-T}^{T} e^{-ixu} \frac{1}{T} \Psi \left( \frac{u}{T} \right) f(u) du, \\
    \lim_{y \to x, \, y < x} F(y)
    &\geq \frac{1}{2} +
    \vp \int_{-T}^{T} e^{-ixu} \frac{1}{T} \Psi \left( \frac{-u}{T} \right) f(u) du.
\end{align*}
Note that these inequalities hold for every distribution $F$ with characteristic function $f$, without any assumption. However, they only involve values of the characteristic function $f$ on the interval $[-T, T]$ (independently of the fact that $f$ may be non-zero elsewhere).

Therefore,
\begin{align}
    F(x)-G_v(x) &\leq \frac{1}{2} +
    \vp \int_{-T}^{T} e^{-ixu} \frac{1}{T} \Psi \left( \frac{u}{T} \right) f(u) du - G_v(x) 
    \label{eq:bound_diff_cdf_sup} \\
    F(x)-G_v(x) &\geq \frac{1}{2} +
    \vp \int_{-T}^{T} e^{-ixu} \frac{1}{T} \Psi \left( \frac{-u}{T} \right) f(u) du - G_v(x).
    \label{eq:bound_diff_cdf_inf} 
\end{align}
Note that the Gil-Pelaez inversion formula~(see \cite{gil1951note}) is valid for any bounded-variation function. Formally, for every bounded-variation function $G(x)=\int_{-\infty}^x g(t) dt$, denoting the Fourier transform of a given function $g$ by
$\check g := \int_{-\infty}^{+\infty} e^{ixu} g(u) du$, we have
\begin{align}
    G(x) &= \frac{1}{2} + \frac{i}{2\pi} \, \vp \int_{-\infty}^{+\infty} e^{-ixu} \check g(u) du.
    \label{eq:fourier_cdf}
\end{align}
Therefore, applying Equation~(\ref{eq:fourier_cdf}) to the function $G_v(x)
:= \Phi(x) + v(1-x^2) \varphi(x) / 6$ whose (generalized) density has the Fourier transform $(1 - vix^3 / 6) e^{- x^2 / 2}$, we get
\begin{align*}
    G_v(x)
    &= \frac{1}{2} + \frac{i}{2\pi} \,
    \vp \int_{-\infty}^{+\infty} e^{-ixu} \left(
    1 - \frac{v}{6} iu^3
    \right) e^{- u^2 / 2} \frac{du}{u}.
\end{align*}
Combining this expression of $G_v(x)$ with the bounds (\ref{eq:bound_diff_cdf_sup}) and (\ref{eq:bound_diff_cdf_inf}), we get
\begin{align*}
    & \left| F(x) - G_v(x) \right| \\
    \leq & \left| \vp \int_{-\infty}^{+\infty}e^{-ixu}\left\{ \frac{1}{T}\Indicator_{\{ |u|\leq T \}}
    \Psi\left(\frac{u}{T}\right)f(u) - \frac{i}{2\pi}\left(
    1 - \frac{v}{6} iu^3
    \right) \frac{e^{- u^2 / 2}}{u} \right\}du \right| \\
    \leq & \vp \int_{-\infty}^{+\infty} \left| \frac{1}{T}\Indicator_{\{ |u|\leq T \}}
    \Psi\left(\frac{u}{T}\right)f(u) - \frac{i}{2\pi}\left(
    1 - \frac{v}{6} iu^3
    \right) \frac{e^{- u^2 / 2}}{u} \right| du \\
    = & \int_{-\infty}^{+\infty} \left| \frac{1}{T}\Indicator_{\{ |u|\leq T \}}
    \Psi\left(\frac{u}{T}\right)f(u) - \frac{i}{2\pi}\left(
    1 - \frac{v}{6} iu^3
    \right) \frac{e^{- u^2 / 2}}{u} \right|du,
\end{align*}
where we resort to the triangle inequality and to the fact that the principal value of the integral of a positive function is the (usual) integral of that function.
Combining $\Psi(-u) = \overline\Psi(u)$ and $f(-u) = \overline{f(u)}$ with basic properties of conjugate and modulus, so that
\begin{align*}
    &\left| \frac{1}{T}\Indicator_{\{ |-u|\leq T \}}
    \Psi\left(\frac{-u}{T}\right)f(-u) - \frac{i}{2\pi}\left(
    1 - \frac{v}{6} i(-u)^3
    \right) \frac{e^{- (-u)^2 / 2}}{-u} \right| \\
    &= \left| \frac{1}{T}\Indicator_{\{ |u|\leq T \}}
    \overline \Psi\left(\frac{u}{T}\right) \overline f(u)
    + \frac{i}{2\pi}\left(
    1 + \frac{v}{6} iu^3
    \right) \frac{e^{- u^2 / 2}}{u} \right| \\
    &= \left| \overline{ \frac{1}{T}\Indicator_{\{ |u|\leq T \}}
    \Psi\left(\frac{u}{T}\right) f(u)}
    - \overline{\frac{i}{2\pi}\left(
    1 - \frac{v}{6} iu^3
    \right) \frac{e^{- u^2 / 2}}{u}} \right| \\
    &= \left| \frac{1}{T}\Indicator_{\{ |u|\leq T \}}
    \Psi\left(\frac{u}{T}\right) f(u)
    - \frac{i}{2\pi}\left(
    1 - \frac{v}{6} iu^3
    \right) \frac{e^{- u^2 / 2}}{u} \right|.
\end{align*}
Using this symmetry with respect to $u$, we obtain
\begin{align*}
    \left| F(x) - G_v(x) \right|
    &= 2 \int_{0}^{+\infty} \left| \frac{1}{T}\Indicator_{\{ u \leq T \}}
    \Psi\left(\frac{u}{T}\right)f(u) - \frac{i}{2\pi}\left(
    1 - \frac{v}{6} iu^3
    \right) \frac{e^{- u^2 / 2}}{u} \right| du.
\end{align*}
By distinguishing the cases $u \leq T$ and $u \geq T$, we obtain
\begin{align*}
    &\left| F(x) - G_v(x) \right| \\
    &\leq 2 \int_{0}^{T} \left| \frac{1}{T}
    \Psi\left(\frac{u}{T}\right)f(u) - \frac{i}{2\pi}\left(
    1 - \frac{v}{6} iu^3
    \right) \frac{e^{- u^2 / 2}}{u} \right| du \\
    & \;\;\; + 2 \int_{T}^{+\infty} \left| \frac{i}{2\pi} \left(
    1 - \frac{v}{6} iu^3
    \right) \frac{e^{- u^2 / 2}}{u} \right| du \\
    &\leq 2 \int_{0}^{T/\pi} \left| \frac{1}{T}
    \Psi\left(\frac{u}{T}\right)f(u) - \frac{i}{2\pi}\left(
    1 - \frac{v}{6} iu^3
    \right) \frac{e^{- u^2 / 2}}{u} \right| du \\
    & \;\;\; + 2 \int_{T/\pi}^{T} \left| \frac{1}{T}
    \Psi\left(\frac{u}{T}\right)f(u) - \frac{i}{2\pi}\left(
    1 - \frac{v}{6} iu^3
    \right) \frac{e^{- u^2 / 2}}{u} \right| du \\
    & \;\;\; + \int_{T}^{+\infty} \frac{1}{\pi} \left(
    1 + \frac{|v|}{6} u^3
    \right) \frac{e^{- u^2 / 2}}{u} du \\
    &\leq 2 \int_{0}^{T/\pi} \left| \frac{1}{T}
    \Psi\left(\frac{u}{T}\right)f(u) - \frac{i}{2\pi}\left(
    1 - \frac{v}{6} iu^3
    \right) \frac{e^{- u^2 / 2}}{u} \right| du
    + 2 \int_{T/\pi}^{T} \left| \frac{1}{T}
    \Psi\left(\frac{u}{T}\right)f(u) \right| du \\
    & \hspace{2cm} + 2 \int_{T/\pi}^{T} \left| \frac{1}{T}
    \frac{i}{2\pi}\left(
    1 - \frac{v}{6} \right) \frac{e^{- u^2 / 2}}{u} \right| du
    + \int_{T}^{+\infty} \frac{1}{\pi} \left(
    1 + \frac{|v|}{6} u^3
    \right) \frac{e^{- u^2 / 2}}{u} du.
\end{align*}
We merge the last two terms together as they correspond to the same integrand, integrated from $T/\pi$ to $+ \infty$.
\begin{align*}
    \left| F(x) - G_v(x) \right|
    &\leq 2 \int_{0}^{T/\pi} \left| \frac{1}{T}
    \Psi\left(\frac{u}{T}\right)f(u) - \frac{i}{2\pi}\left(
    1 - \frac{v}{6} iu^3
    \right) \frac{e^{- u^2 / 2}}{u} \right| du \\
    &\hspace{2cm} + 2 \int_{T/\pi}^{T} \left| \frac{1}{T}
    \Psi\left(\frac{u}{T}\right)f(u) \right| du \\
    &\hspace{2cm} + \int_{T/\pi}^{+\infty} \frac{1}{\pi} \left(
    1 + \frac{|v|}{6} u^3
    \right) \frac{e^{- u^2 / 2}}{u} du.
\end{align*}
We use the triangle inequality to break the first integral into two parts
\begin{align*}
    & \left| F(x) - G_v(x) \right| \\
    \leq & 2 \int_{0}^{T/\pi} \left| \frac{1}{T}
    \Psi\left(\frac{u}{T}\right)f(u) 
    - \frac{1}{T} \Psi\left(\frac{u}{T}\right)\left(
    1 + \frac{v}{6} iu^3
    \right)e^{- u^2 / 2} \right| du \\
    & + 2 \int_{0}^{T/\pi} \left| \frac{1}{T} \Psi\left(\frac{u}{T}\right)\left(
    1 + \frac{v}{6} iu^3
    \right)e^{- u^2 / 2}
    - \frac{i}{2\pi}\left(
    1 + \frac{v}{6} iu^3
    \right) \frac{e^{- u^2 / 2}}{u} \right| du \\
    & + 2 \int_{T/\pi}^{T} \left| \frac{1}{T}
    \Psi\left(\frac{u}{T}\right)f(u) \right| du
    + \int_{T/\pi}^{+\infty} \frac{1}{\pi} \left(
    1 + \frac{|v|}{6} u^3
    \right) \frac{e^{- u^2 / 2}}{u} du.
\end{align*}
We successively split the first term into two integrals, and apply the triangle inequality to break the first integral into two parts
\begin{align*}
    \left| F(x) - G_v(x) \right|
    &\leq 2\int_0^{\tau \wedge T/\pi} \left| \frac{1}{T}
    \Psi\left(\frac{u}{T}\right)\left( f(u) -  \left(
    1 + \frac{|v|}{6} u^3
    \right) e^{- u^2 / 2} \right) \right| du \\
    &\hspace{2cm} + 2\int_{\tau \wedge T/\pi}^{T/\pi} \left| \frac{1}{T}
    \Psi\left(\frac{u}{T}\right)\left(
    1 + \frac{|v|}{6} u^3
    \right) e^{- u^2 / 2} \right| du \\
    &\hspace{2cm} + 2\int_{\tau \wedge T/\pi}^{T/\pi} \left| \frac{1}{T}
    \Psi\left(\frac{u}{T}\right)\left( f(u) - e^{- u^2 / 2} \right) \right| du \\
    &\hspace{2cm} + 2 \int_{0}^{T/\pi} \left| \frac{1}{T} \Psi\left(\frac{u}{T}\right)
    - \frac{i}{2\pi} \right| \left(
    1 + \frac{|v|}{6} u^3
    \right) \frac{e^{- u^2 / 2}}{u} du \\
    &\hspace{2cm} + 2 \int_{T/\pi}^{T} \left| \frac{1}{T}
    \Psi\left(\frac{u}{T}\right)f(u) \right| du \\
    &\hspace{2cm} + \int_{T/\pi}^{+\infty} \frac{1}{\pi} \left(
    1 + \frac{|v|}{6} u^3
    \right) \frac{e^{- u^2 / 2}}{u} du.
    \qquad\qquad \Box
\end{align*}

\section{Control of \texorpdfstring{$(\Omega_\ell)_{\ell=1}^4$}{Omega}}
\label{appendix:sec:omega}

\subsection{Control of the term \texorpdfstring{$\boldsymbol{\Omega_1}$}{omega1}}
\label{subsection:bound_Omega1}

The following lemma enables to control the term \(\Omega_1\).
The same control is used in all cases (\iid{} and \inid{} cases, Theorems~\ref{thm:nocont_choiceEps} and~\ref{thm:cont_choiceEps}).

\begin{lemma}
    For every $T > 0$, we have
    \begin{align}
    \label{eq:up_bound_easy_term}
        \Omega_1(T, |\lambda_{3,n}|/\sqrt{n}, \tau)
        &\leq \frac{1.2533}{T}
        + \frac{0.3334|\lambda_{3,n}|}{T\sqrt{n}}
        + \frac{14.1961}{T^4}
        + \frac{4.3394 |\lambdatroisn|}{T^3 \sqrt{n}} \nonumber \\
        &+ \frac{|\lambdatroisn|
        \big( \Gamma( 3/2 , \tau \wedge T/\pi)
        - \Gamma( 3/2 , T/\pi) \big)}{\sqrt{n}}
      . 
    \end{align}
    \label{lemma:bound_omega_1}
\end{lemma}

\begin{proof}
We can decompose $\Omega_1(T, v, \tau)$ as
\begin{align*}
    & \Omega_1(1/\pi,T,v)
    := \frac{I_{1,1}(T)}{T} + v \times \frac{I_{1,2}(T)}{T} 
    + \frac{I_{1,3}(T)}{T^4}
    + v \times \frac{I_{1,4}(T)}{T^3}
    + v \times I_{1,5}(T)
\end{align*}
where 
\begingroup \allowdisplaybreaks
\begin{align*}
    I_{1,1}(T) &:= T \int_0^{1/\pi}
    \left|2\Psi(t)-\frac{i}{\pi t}\right| e^{-(Tt)^2/2} dt, \\
    I_{1,2}(T) &:= T^4 \int_0^{1/\pi}
    \left|2\Psi(t)-\frac{i}{\pi t}\right| e^{-(Tt)^2/2} \frac{t^3}{6} dt,
    \\
    I_{1,3}(T) &:= T^4 \frac{1}{\pi} \int_{1/\pi}^{+\infty}
    \frac{e^{-(Tt)^2/2}}{t} dt
    = \frac{T^4}{2\pi} \Gamma \left( 0 \, , \, \frac{T^2}{2 \pi^2} \right) , \\
    I_{1,4}(T) &:= T^6 \frac{1}{\pi} \int_{1/\pi}^{+\infty}
    e^{-(Tt)^2/2} \frac{t^2}{6} dt
    = \frac{T^3}{3 \sqrt{2} \pi}
    \int_{T^2 / (2\pi^2)}^{+\infty} e^{-u} \sqrt{u} du
    = \frac{T^3}{3 \sqrt{2} \pi}
    \,\Gamma\!\left(\frac{3}{2}, \frac{T^2}{2 \pi^2}\right) \\
    I_{1,5}(T,\tau) &:= 2 \int_{\tau \wedge T/\pi}^{T/\pi}
    \left|\frac{1}{T}\Psi(u/T)\right|\, e^{-u^2/2}
     \frac{u^3}{6} du
\end{align*}
\endgroup
To compute $I_{1,3}$ and $I_{1,4}$, we used the change of variable $u = (tT)^2/2$ and the incomplete Gamma function $\Gamma(a,x) := \int_x^{+\infty} u^{a-1} e^{-u} du$ which can be computed numerically using the package \texttt{expint}~\citep{goulet2016expint} in R.
We estimate numerically the first two integrals using the R package \texttt{cubature}~\citep{balasubramanian2020cubature} and optimize using the \texttt{optimize} function with the \texttt{L-BFGS-B} method, we find the following upper bounds:
\begin{align*}
    & \sup_{T \geq 0} I_{1,1}(T) \leq 1.2533, \hspace{0.5cm}
    & \sup_{T \geq 0} I_{1,2}(T) \leq 0.3334, \\ 
    & \sup_{T \geq 0} I_{1,3}(T) \leq 14.1961,
    \hspace{0.5cm}
    & \sup_{T \geq 0} I_{1,4}(T) \leq 4.3394,
\end{align*}
which can be used to bound the first four terms.

By Lemma~\ref{lemma:bound_J1p}, we obtain
\begin{align*}
    I_{1,5}(T,\tau)
    &= \frac{1}{3}
    J_{1} \big(3, \tau \wedge T/\pi, T/\pi, T \big) \\
    &\leq \frac{1.0253}{3 \pi \sqrt{2}}
    \big( \Gamma( 3/2 , \tau \wedge T/\pi)
    - \Gamma( 3/2 , T/\pi) \big),
\end{align*}
as claimed.
\end{proof}

Note that
\begin{align*}
    I_{1,5}(T,\tau)
    &= \frac{1}{3}
    J_{1} \big(3, \tau \wedge T/\pi, T/\pi, T \big) \\
    &\leq \frac{1.0253}{3 \pi \sqrt{2}}
    \Big( \Gamma \big( 3/2 , \tau \wedge T/\pi \big)
    - \Gamma \big( 3/2 , T/\pi \big) \Big) \\
    &= O \Big( n^{1/4} e^{- \eps \sqrt{n}/ \sqrt{K_{4,n}} } \Big),
\end{align*}
where we apply the asymptotic expansion $\Gamma(a,x) = x^{a-1} e^{-x} (1 + O((a-1) / x))$ which is valid for every fixed $a$ in the regime $x\to \infty$, see Equation~(6.5.32) in \cite{abramowitz85handbook}.

\medskip

Note that the first term on the right-hand side of \eqref{eq:up_bound_easy_term} is of leading order as soon as $|\lambda_{3,n}|/\sqrt{n} = o(1)$ and $T = T(n) = o(1).$
Our approach is related to the one used in \cite{shevtsova2012}, except that we do not upper bound $\Omega_1$ analytically, which allows us to get a sharper control on this term.
To further highlight the gains from using numerical approximations instead of direct analytical upper bounds, we remark that from $\left|\Psi(t)-\frac{i}{2\pi t}\right| \leq \frac{1}{2}\left(1-|t|+\frac{\pi^2t^2}{18}\right)$ and some integration steps, we get
\begin{align*}
    I_{1,1}(T)
    &\leq T \int_0^{1/\pi} \left(1-|t|+\frac{\pi^2t^2}{18}\right)
    e^{-(Tt)^2/2}dt \nonumber \\
    &= \sqrt{2\pi}\left( \Phi(T/\pi)-\frac{1}{2} \right) + \frac{1}{T}\left( e^{-(T/\pi)^2/2}-1 \right) \\
    & \;\;\; + \frac{\pi^{5/2}}{9\sqrt{2}T^2}\mathbb{E}_{U\sim\mathcal{N}(0,1)}[U^2\Indicator\left\{ 0\leq U\leq T/\pi\right\}] \nonumber \\
    &\leq \sqrt{2\pi} + \frac{1}{T}\left( e^{-T^2/(2\pi^2)}-1 \right) + \frac{\pi^{5/2}}{9\sqrt{2} \, T^2},
\end{align*}
whose main term is approximately twice as large as the numerical bound $1.2533$ that we obtained before.

\subsection{Control of the term \texorpdfstring{$\boldsymbol{\Omega_2}$}{omega2}}
\label{sssec:bound_nocount_inid_Omega2}

In this section, we control
$\Omega_2(T)
= 2\int_{1/\pi}^1|\Psi(t)|\,|\caracfsum(Tt)|dt.$
The control used in Theorem~\ref{thm:cont_choiceEps} comes directly from the upper bound on the absolute value of \(\Psi\) (Equation~\eqref{eq:properties_Psi}):
\begin{equation*}
    \Omega_2(T) 
    \leq 
    \frac{1.0253}{\pi T}
    \int_{1 / \pi}^1 \frac{|\caracfsum(Tt)|}{t}  dt.
\end{equation*}

In Theorem~\ref{thm:nocont_choiceEps}, we derive a bound based on the following lemma.
\begin{lemma}
\label{lemma:bound_Psi_Caracfsum}
    Let $t_1^* = \theta_1^* / (2\pi)$ where $\theta_1^*$ is the unique root in $(0,2\pi)$ of the equation $\theta^2+2\theta\sin(\theta)+6(\cos(\theta)-1)=0$ and $\xi_n := \Ktroisntilde/\sqrt{n}$.
    We obtain
    \begin{align*}
        (i) \quad  & \int_{1/\pi}^{t_1^*} |\Psi(t)|\,|\caracfsum(2\pi t/\xi_n)|dt
        \leq 
        \int_{1/\pi}^{t_1^*} |\Psi(t)|
        e^{-(2\pi t/\xi_n)^2 (1 - 4 \pi \chi_1 t) / 2}dt \\
        (ii) \quad 
        & \int_{t_1^*} |\Psi(t)|\,|\caracfsum(2\pi t/\xi_n)|dt
        \leq 
        \int_{t_1^*}^1 |\Psi(t)| e^{-(1-\cos(2\pi t))/\xi_n^2}dt.
    \end{align*}
\end{lemma}

\noindent
{\it Proof of
Lemma~\ref{lemma:bound_Psi_Caracfsum}: }
Applying Theorem 2.2 in \cite{shevtsova2012} with $\delta =1$, we get for all $u \in \Rb$
\begin{align*}
    |\caracfsum(u)|
    &\leq \exp \left( - \psi(u , \epsilon_n) \right),
\end{align*}
where $\epsilon_n := n^{-1/2} \Ktroisntilde,$
and, for any real $u, \epsilon > 0$
\begin{align*}
    \psi(u , \epsilon) &:=
    \begin{cases}
        u^2/2 - \chi_1 \epsilon |u|^3,
        & \text{ for } |u| < \theta_1^* \epsilon^{-1} , \\
        \dfrac{1 - \cos(\epsilon u)}{\epsilon^2},
        & \text{ for } \theta_1^* \epsilon^{-1}
        \leq |u| \leq 2 \pi \epsilon^{-1} , \\
        0,
        & \text{ for } |u| > 2 \pi \epsilon^{-1}.
    \end{cases}
\end{align*}
Therefore,
\begin{align}
    |\caracfsum(u)| &\leq \begin{cases}
        \exp \big( - u^2/2 + \chi_1 \xi_n |u|^3 \big),
        & \text{ for } |u| < \theta_1^* / \xi_n , \\
        \exp \bigg( \dfrac{\cos(\xi_n u)-1}{\xi_n^2} \bigg),
        & \text{ for } \theta_1^* / \xi_n
        \leq |u| \leq 2 \pi / \xi_n , \\
        1,
        & \text{ for } |u| > 2 \pi / \xi_n.
    \end{cases}
    \label{eq:bound_caracfsum_u_exp}
\end{align}
Choosing $u=2\pi t/\xi_n$, multiplying by $|\Psi|$, integrating from $1/\pi$ to $1$ and separating the two cases yields the claimed inequalities.
$\Box$

Recall that under moment conditions only, we choose
\( T = \frac{2\pi}{\xi_n}
= \frac{2 \pi \sqrt{n}}{\Ktroisntilde}
\).
Combining this with the two inequalities~(i) and~(ii) of Lemma~\ref{lemma:bound_Psi_Caracfsum} yields
\begin{align*}
    \int_{1/\pi}^1|\Psi(t)|\,|\caracfsum(Tt)|dt
    &= \frac{I_{2,1}(T)}{2 T^4} + \frac{I_{2,2}(T)}{2 T^2},
\end{align*}
where
\begin{align*}
    I_{2,1}(T) &:= T^4 \int_{1/\pi}^{t_1^*}
    2 |\Psi(t)| e^{-\frac{(Tt)^2}{2} \left(1-4\pi\chi_1|t|\right)}dt, \\
    I_{2,2}(T) &:= T^2 \int_{t_1^*}^1
    2 |\Psi(t)| e^{-T^2(1-\cos(2\pi t))/(4\pi^2)} dt.
\end{align*}
Note that the difference in the two exponents of $T$ in the above definitions may seem surprising as these two integrals look similar. However they have very different behaviors since the first one decays much faster than the second one.
In line with Section~\ref{subsection:bound_Omega1}, we compute numerically these integrals using the R package \texttt{cubature}~\citep{balasubramanian2020cubature} and optimize them using the \texttt{optimize} function with the \texttt{L-BFGS-B} method.
This gives
\begin{align*}
    \sup_{T \geq 0} I_{2,1}(T) \leq 67.0415, \hspace{0.25cm}
    \text{ and } \hspace{0.25cm}
    \sup_{T \geq 0} I_{2,2}(T) \leq 1.2187.
\end{align*}
Finally, we arrive at
\begin{align}
\label{eq:bound_omega2_nocont}
    &\Omega_2(T)
    =  2\int_{1/\pi}^1|\Psi(t)| \, |\caracfsum(Tt)|dt
    \leq \frac{67.0415}{T^4} + \frac{1.2187}{T^2}.
\end{align}

\subsection{Control of the term \texorpdfstring{$\boldsymbol{\Omega_3}$}{omega3}}

We recall that $\tau$ is defined as \( \tau = \sqrt{2\eps} (n/K_{4,n})^{1/4} \) (see Equation~\eqref{eq:tau_choice}).

\begin{lemma}
    \label{lemma:bound_omega3}
    Under Assumption~\ref{hyp:basic_as_inid}, we have for any \(\eps \in (0, 1/3)\) and any \(T > 0\),
    \begin{align}
    \label{eq:bound_omega_3_inid}
        \Omega_3(T, \lambda_{3,n} / \sqrt{n}, \tau)
        & \leq \frac{0.327 \, K_{4,n}}{n}
        \left(\frac{1}{12}+\frac{1}{4(1-3\eps)^2}\right)
        + \frac{1.306 \,  e_{1,n}(\eps) |\lambda_{3,n}|^2}{36n} \nonumber \\
        &+ \frac{1.0253}{\pi}
        \int_0^{\tau \wedge T/\pi} u e^{-u^2/2} \Rinid(u,\eps) du,
    \end{align}
    where the functions $\Rinid$ and $e_{1,n}$ are defined in Equations~\eqref{eq:def_R_1n} and~\eqref{eq:def_e_1n} respectively.
    
    Under Assumption~\ref{hyp:basic_as_iid}, we have
    \begin{align}
    \label{eq:bound_omega_3_iid}
        \Omega_3(T, \lambda_{3,n} / \sqrt{n}, \tau)
        & \leq \frac{0.327 \, K_{4,n}}{n}
        \left(\frac{1}{12}+\frac{1}{4(1-3\eps)^2}\right)
        + \frac{1.306 \, e_{2,n}(\eps) |\lambda_{3,n}|^2}{36n} \nonumber \\
        & + \frac{1.0253}{\pi}
        \int_0^{\tau \wedge T/\pi}
        u e^{-u^2/2} \Riid(u,\eps) du,
    \end{align}
    where the functions $\Riid$ and $e_{2,n}$ are defined in Equations~\eqref{eq:def_R_2n} and~\eqref{eq:definition_e2n} respectively.
\end{lemma}

\noindent
{\it Proof of
Lemma~\ref{lemma:bound_omega3}: }

First, assume that Assumption~\ref{hyp:basic_as_inid} holds.
Lemma~\ref{lem:taylor_exp_cf} enables us to write
\begin{align*}
    \Omega_3(T, \lambda_{3,n} / \sqrt{n}, \tau)
    & = \int_{0}^{\tau \wedge T/\pi}
    |\Psi(t)|\, \left|\caracfsum(Tt) - e^{-(Tt)^2/2}
    \left(1-\dfrac{vi(Tt)^3}{6}\right)\right|dt \\
    & \leq  \frac{K_{4,n}}{n}
    \left(\frac{1}{12}+\frac{1}{4(1-3\eps)^2}\right)
    J_{1} \left(4 ,  0 , \tau \wedge T/\pi , T \right)
    \\
    & \;\;\; + \frac{e_{1,n}(\eps)}{36}
    \frac{|\lambda_{3,n}|^2}{n}
    J_{1} \left(6 , 0 , \tau \wedge T/\pi , T \right)  \\
    & \;\;\; + \frac{2}{T}\int_0^{\tau \wedge T/\pi}
    |\Psi(u/T)|e^{-u^2/2}\Rinid(u,\eps)du,
\end{align*}
where the function $J_1$ is defined in Equation~\eqref{eq:def_J1_p}. 
Using Equation~\eqref{lemma:bound_J1p}, we obtain the bounds $J_{1}(4, 0,+\infty,T)\leq 0.327$
and $J_{1}(6, 0,+\infty,T) \leq 1.306.$
Besides, by the first inequality in~(\ref{eq:properties_Psi}), we get
\begin{align*}
    \Omega_3(T, \lambda_{3,n} / \sqrt{n}, \tau)
    &\leq \frac{0.327 K_{4,n}}{n}
    \left(\frac{1}{12} + \frac{1}{4(1-3\eps)^2}\right)
    + \frac{1.306e_{1,n}(\eps)}{36}\frac{|\lambda_{3,n}|^2}{n} \nonumber \\
    & \;\;\; + \frac{1.0253}{\pi}
    \int_0^{\tau \wedge T/\pi} u e^{-u^2/2} \Rinid(u,\eps) du.
\end{align*}
showing Equation~\eqref{eq:bound_omega_3_inid} as claimed.

\medskip

Assume now that Assumption~\ref{hyp:basic_as_iid} holds.
The integrand of $I_{4,1}(T)$ can be upper bounded thanks to Lemma~\ref{lem:taylor_exp_cf_iid}. We obtain
\begin{align*}
    \Omega_3(T, \lambda_{3,n} / \sqrt{n}, \tau)
    &\leq \frac{K_{4,n}}{n}
    \left(\frac{1}{12}+\frac{1}{4(1-3\eps)^2}\right)
    J_{1}\left(4, 0, \tau \wedge T/\pi, T \right)
    \nonumber \\
    &\;\;\; + \frac{e_{2,n}(\eps)|\lambda_{3,n}|^2}{36n}
    J_{1}\left(6, 0, \tau \wedge T/\pi, T \right) \\
    &\;\;\; + \frac{2}{T} \int_0^{\tau \wedge T/\pi}
    |\Psi(u/T)| e^{-u^2/2} \Riid(u,\eps) du.
\end{align*}
This completes the proof of Equation~\eqref{eq:bound_omega_3_iid}.
$\Box$

\subsection{Control of the term \texorpdfstring{$\boldsymbol{\Omega_4}$}{omega4}}
\label{ssec:up_bound_omega4_basic}

\noindent
In this section, we bound the fourth term of Equation~\eqref{eq:smoothing}, which is
\begin{align*}
    \Omega_4(a, b, T)
    &:= 2 \int_{a}^{b}
    \left|\frac{1}{T}\Psi(u/T)\right| \,
    \left|f(u)-e^{-u^2/2} \right| du,
\end{align*}
for $f = \caracfsum$.

We prove a bound on 
$\Omega_4(\sqrt{2\eps}(n/K_{4,n})^{1/4} \wedge T/\pi, T/\pi, T)$
under four different sets of assumptions.

\begin{lemma}
Let $-\infty < a \neq b < +\infty$ and $T > 0$. Then
\begin{enumerate}
    \item Under Assumption~\ref{hyp:basic_as_inid}, we have
    \begin{align*}
        \big| \Omega_4(a,b, T) \big|
        \leq \frac{K_{3,n}}{3\sqrt{n}}
        \Big| J_{2} \big(3, a, b , 2\sqrt{n}/\Ktroisntilde , T \big) \Big|,
    \end{align*}
    where $J_2$ is defined in Equation~\eqref{eq:def_J2_p}.
    
    \item Under Assumption~\ref{hyp:basic_as_inid} and assuming $\E[X_i^3]=0$ for all $i= 1,\dots,n$, we get the improved bound
    \begin{align*}
        \big| \Omega_4(a, b, T) \big|
        \leq \frac{K_{4,n}}{3n}
        \Big| J_{2} \big(4, a, b , 2\sqrt{n}/\Ktroisntilde , T) \Big|,
    \end{align*}
    
    \item Under Assumption~\ref{hyp:basic_as_iid}, we have
    \begin{align*}
        \big| \Omega_4(a, b, T) \big|
        \leq \frac{K_{3,n}}{3 \sqrt{n}}
        \Big| J_{3}(3, a, b, 2\sqrt{n}/\Ktroisntilde , T) \Big|,
    \end{align*}
    where $J_3$ is defined in Equation~\eqref{eq:def_J3_p}.
    
    \item Under Assumption~\ref{hyp:basic_as_iid} and assuming $\E[X_i^3]=0$ for all $i= 1,\dots,n$, we get the improved bound
    \begin{align*}
        \big| \Omega_4(a, b, T) \big|
        &\leq  \frac{K_{4,n}}{3 n}
        \Big| J_{3}(4, a, b, 2\sqrt{n}/\Ktroisntilde, T) \Big|.
    \end{align*}
    
\end{enumerate}
\label{lemma:bound_I32}
\end{lemma}

Remark that if $a < b$, the four inequalities hold without absolute values since $\Omega_4$ and $J_2$ are then non-negative.

\noindent {\it Proof of Lemma~\ref{lemma:bound_I32}(i)}.
Let $t \in \Rb$.
As in the proof of Lemma 2.7 in \cite{shevtsova2012} with $\delta = 1$, using the fact that for every $i= 1,\dots,n$, we have
\begin{align*}
    \max\left\{|f_{P_{X_i}}(t)| , \,
    \exp \left( - \frac{t^2\sigma_i^2}{2} \right) \right\}
    \leq \exp \left(
    - \frac{t^2\sigma_i^2}{2}
    + \frac{\chi_1 t^3(\mathbb{E}[|X_i|^3] + \mathbb{E}[|X_i|]\sigma_i^2)}{B_n^3} \right),
\end{align*}
so that
\begin{align*}
    & \left|\caracfsum(t)-e^{-t^2/2}\right| \\
    \leq & \sum_{i=1}^n \left|
    f_{P_{X_i}} \Big( \frac{t}{B_n} \Big) - e^{-\dfrac{t^2\sigma_i^2}{2B_n^2}} \right|
    e^{\dfrac{t^2\sigma_i^2}{2B_n^2}}
    e^{-\dfrac{t^2}{2}
    + \dfrac{\chi_1|t|^3\sum_{l=1}^n\big( \mathbb{E}[|X_l|^3] + \mathbb{E}[|X_i|]\sigma_i^2 \big)}{B_n^3}} \\
    = & \sum_{i=1}^n \left|
    f_{P_{X_i}} \Big( \frac{t}{B_n} \Big)
    - e^{-\dfrac{t^2\sigma_i^2}{2B_n^2}}\right|
    e^{-\dfrac{t^2}{2} + \dfrac{\chi_1|t|^3\Ktroisntilde}{\sqrt{n}} + \dfrac{t^2\sigma_i^2}{2B_n^2}}.
\end{align*}
By Equation~\eqref{eq:edg_exp_1}, we have $\max_{1\leq i\leq n}\sigma_i^2 \leq B_n^2 \times (K_{4,n}/n)^{1/2}$ so that we obtain
\begin{align*}
    & \left|\caracfsum(t)-e^{-t^2/2}\right|
    \leq \sum_{i=1}^n
    \left|f_{P_{X_i}} \Big( \frac{t}{B_n} \Big)-e^{-\dfrac{t^2\sigma_i^2}{2B_n^2}}\right|
    e^{-\dfrac{t^2}{2}
    + \dfrac{\chi_1|t|^3\Ktroisntilde}{\sqrt{n}}
    + \dfrac{t^2}{2}\sqrt{\dfrac{K_{4,n}}{n}}}.
\end{align*}
Applying Lemma 2.8 in~\cite{shevtsova2012}, we get that for every variable $X$ such that $\E[|X|^3]$ is finite, $|f(t) - e^{-\sigma^2 t^2}| \leq \E[|X|^3] \times |t|^3 / 6$. Therefore,
\begin{align}
\label{eq:bound_caracfun_caracGauss}
    \left|\caracfsum(t)-e^{-t^2/2}\right|
    & \leq \sum_{i=1}^n
    \frac{\mathbb{E}[|X_i|^3]}{6B_n^3}\left|t\right|^3
    \exp \left(-\frac{t^2}{2}
    + \frac{\chi_1|t|^3\Ktroisntilde}{\sqrt{n}}
    + \dfrac{t^2}{2}\sqrt{\dfrac{K_{4,n}}{n}} \right) 
    \nonumber \\
    & = \frac{K_{3,n}}{6\sqrt{n}}|t|^3
    \exp \left( -\frac{t^2}{2}
    + \frac{\chi_1|t|^3\Ktroisntilde}{\sqrt{n}} 
    + \dfrac{t^2}{2}\sqrt{\dfrac{K_{4,n}}{n}} \right).
\end{align}
Integrating the latter equation,
we have
\begin{align}
\label{eq:gen_case_nocont_3}
    \Big| \Omega_4(a, b, T) \Big|
    &= \frac{2}{T} \, \Bigg| \int_{a}^{b}
    |\Psi(u/T)| \, \left|\caracfsum(u)-e^{-u^2/2}\right|du \Bigg| \nonumber \\
    &\leq \frac{K_{3,n}}{3\sqrt{n}T} \, \Bigg| \int_{a}^{b}
    |\Psi(u/T)| \, u^3
    \exp \left( -\frac{u^2}{2} + \frac{u^3 \chi_1 \Ktroisntilde}{\sqrt{n}} 
    + \dfrac{u^2}{2}\sqrt{\dfrac{K_{4,n}}{n}} \right) du \Bigg| \nonumber \\
    &= \frac{K_{3,n}}{3\sqrt{n}T} \, \Bigg| \int_{a}^{b}
    |\Psi(u/T)| \, u^3
    \exp \left( -\frac{u^2}{2}\Big( 1 + \frac{ 4 u \chi_1 \Ktroisntilde}{2\sqrt{n}} 
    + \sqrt{\dfrac{K_{4,n}}{n}} \Big) \right) du \Bigg| \nonumber \\    
    &= \frac{K_{3,n}}{3\sqrt{n}}
    \Big| J_{2} \big(3, a, b , T \big) \Big|,
\end{align}
as claimed.

\bigskip

\noindent {\it Proof of Lemma~\ref{lemma:bound_I32}(ii)}. This second part of the proof mostly follows the reasoning of the first one, with suitable modifications.

First, using a Taylor expansion of order 3 of $f_{P_{X_i}}$ around 0 (with explicit Lagrange remainder) and the inequality $\left| e^{-x} - 1 + x \right| \leq x^2/2,$ we can claim for every real $t$
\begin{align*}
    \left| f_{P_{X_i}}(t) - e^{-t^2\sigma_i^2/2} \right|
    \leq \frac{t^4\gamma_i}{24} + \frac{\sigma_i^4t^4}{8} \leq \frac{t^4\gamma_i}{6}.
\end{align*}
Reasoning as in the proof of Lemma~2.7 in \cite{shevtsova2012} with $\delta = 1$, we obtain
\begin{align*}
    \left|\caracfsum(t)-e^{-t^2/2}\right|
    & \leq \sum_{i=1}^n \frac{t^4\gamma_i}{6B_n^4}
    \exp \left( -\dfrac{t^2}{2} + \dfrac{\chi_1|t|^3\Ktroisntilde}{\sqrt{n}} + \dfrac{t^2}{2}\sqrt{\dfrac{K_{4,n}}{n}} \right) \\
    &\leq \frac{K_{4,n}}{6n}t^4
    \exp \left( -\dfrac{t^2}{2} + \dfrac{\chi_1|t|^3\Ktroisntilde}{\sqrt{n}}
    + \dfrac{t^2}{2}\sqrt{\dfrac{K_{4,n}}{n}} \right).
\end{align*}
Plugging this into the definition of $I_{3,2}(T)$, we can write
\begin{align}
\label{eq:gen_case_nocont_4}
    \Omega_4(a, b, T)
    &= \frac{2}{T} \int_{a}^{b}
    |\Psi(u/T)| \, \left|\caracfsum(u)-e^{-u^2/2}\right|du, \nonumber \\
    &\leq \frac{K_{4,n}}{3nT} \int_{a}^{b}
    |\Psi(u/T)| \, u^4
    \exp \left( -\dfrac{u^2}{2}
    + \dfrac{\chi_1|u|^3\Ktroisntilde}{\sqrt{n}}
    + \dfrac{\colRev{u^2}}{2}\sqrt{\dfrac{K_{4,n}}{n}} \right) du \nonumber \\
    &\leq \frac{K_{4,n}}{3n}
    J_{2} \big(4, a, b , T),
\end{align}
as claimed.

\bigskip

\noindent {\it Proof of Lemma~\ref{lemma:bound_I32}(iii)}.
Under the \iid{} assumption, we can prove that, for every real~$t$,
\begin{align*}
    & \left|\caracfsum(t)-e^{-t^2/2}\right|
    \leq \frac{K_{3,n}}{6\sqrt{n}} |t|^3
    \exp \left( -\dfrac{t^2}{2} + \dfrac{\chi_1|t|^3\Ktroisntilde}{\sqrt{n}} + \dfrac{t^2}{2n} \right),
\end{align*}
following the method of Lemma~\ref{lemma:bound_I32}(i).
Multiplying by $|\Psi(t)|$ and integrating this, we get the claimed inequality.

\bigskip

\noindent {\it Proof of Lemma~\ref{lemma:bound_I32}(iv)}.
This can be recovered using the same techniques as in the proof of Lemma~\ref{lemma:bound_I32}(ii).
$\Box$

\medskip

In Section~\ref{ssec:thms_cont}, we want to give improved bounds that uses the tail behavior of $\caracfsum$ via the integral $\int |\caracfsum(u)| u^{-1} du$. Therefore, the following lemma is used to control \(\Omega_4\) in Theorem~\ref{thm:cont_choiceEps}.

\begin{lemma}

    Let \(T = 16 \pi^4 n^2 / \Ktroisntilde^4 \).
    Then,
    \begin{align*}
        &\Omega_4(\sqrt{2\eps}(n/K_{4,n})^{1/4} \wedge T/\pi, T/\pi, T)
        \leq \Big| \Omega_4(\sqrt{2\eps}(n/K_{4,n})^{1/4} \wedge T/\pi, T^{1/4}/\pi \wedge T/\pi, T) \Big| \\
        & + \frac{1.0253}{\pi} \left( \Gamma\left( 0 , (T^{1/2}\wedge T^2)(1-4\pi\chi_1t_1^*) / (2\pi^2) \right) - \Gamma\left( 0 , (t_1^{*2}T^{1/2}\wedge T^2/\pi^2)(1-4\pi\chi_1t_1^*) / 2 \right) \right) \\
        & + \frac{1.0253}{\pi} \left( \Gamma\left( 0 , (T^{1/2} \wedge T^2) / (2\pi^2) \right) - \Gamma\left( 0 , T^2/(2\pi^2) \right) \right) + \frac{1.0253}{\pi} \int_{t_1^*T^{1/4} \wedge T/\pi}^{T/\pi}
        \frac{|\caracfsum(u)|}{u}du.
    \end{align*}
    \label{lemma:omega4_smooth}
\end{lemma}
Note that the first term of this inequality will be bounded by Lemma~\ref{lemma:bound_I32}. 
The second and fourth terms decrease to zero faster than polynomially with $n$ (see~\cite{abramowitz85handbook} and the discussion at the end of Subsection~\ref{subsection:bound_Omega1}).
Finally, the term containing the integral of $u^{-1}|\caracfsum(u)|$ is the dominant one and allows us to use the assumption on the tail behavior of $\caracfsum$ to obtain Corollaries~\ref{cor:improvement_inid_case} (\inid{} case) and~\ref{cor:improvement_iid_case} (\iid{} case).

\medskip

\noindent
{\it Proof of
Lemma~\ref{lemma:omega4_smooth}: }
We decompose $\Omega_4$ in two parts
\begin{align*}
    &\Omega_4(\sqrt{2\eps}(n/K_{4,n})^{1/4} \wedge T/\pi, T/\pi, T) \\
    = & 2 \int_{\sqrt{2\eps}(n/K_{4,n})^{1/4} \wedge T/\pi}^{T/\pi}
    \left|\frac{1}{T}\Psi(u/T)\right| \,
    \left| f_{S_n}(u) - e^{-u^2/2} \right| du \\
    \leq & \left| \Omega_4(\sqrt{2\eps}(n/K_{4,n})^{1/4} \wedge T/\pi, T^{1/4}/\pi \wedge T/\pi, T) \right|
    + \Omega_4(T^{1/4}/\pi \wedge T/\pi, T/\pi, T).
\end{align*}
Note that the second term of this inequality can be bounded as
\begin{align*}
    \Omega_4(T^{1/4}/\pi \wedge T/\pi, T/\pi, T)
    \leq J_4(T) + J_5(T)
    + J_1(0, T^{1/4} / \pi, T / \pi, T),
\end{align*}
where
\begin{align*}
    & J_4(T) := \frac{2}{T} \int_{T^{1/4}/\pi \wedge T/\pi}^{t_1^*T^{1/4} \wedge T/\pi}
    |\Psi(u/T)|
    \left|\caracfsum(u)\right| du \\
    & \qquad\quad = \frac{2}{T^{3/4}} \int_{1 / \pi \wedge T^{3/4}/\pi}^{t_1^* \wedge T^{3/4}/\pi}
    |\Psi(v/T^{3/4})|
    \left|\caracfsum(T^{1/4}v)\right| dv , \\
    &J_5(T)
    := \frac{2}{T} \int_{t_1^*T^{1/4} \wedge T/\pi}^{T / \pi}
    |\Psi(u/T)|
    \left|\caracfsum(u)\right| du , \\
    &J_1(0, T^{1/4} / \pi, T / \pi, T)
    := \frac{2}{T} \int_{T^{1/4}/\pi \wedge T/\pi}^{T / \pi}
    |\Psi(u/T)|e^{-u^2/2} du .
\end{align*}

By the first inequality of Equation~\eqref{eq:bound_caracfsum_u_exp} and our choice of $T$, we know $|\caracfsum(T^{1/4}v)|$ can be upper bounded by $\exp(-T^{1/2}v^2(1-4\pi\chi_1|v|)/2)$ when $v \in [1/\pi,t_1^*].$ 
Using the properties of $u \mapsto \Psi(u)$ in Equation~\eqref{eq:properties_Psi}, the fact that $1-4\pi\chi_1t_1^* > 0$ and a change of variable, we get
\begin{align*}
    J_4(T) &\leq \frac{2}{T^{3/4}} \int_{1 / \pi \wedge T^{3/4}/\pi}^{t_1^* \wedge T^{3/4}/\pi}
    |\Psi(v/T^{3/4})|
    e^{-\frac{T^{1/2}v^2}{2}\left(1-4\pi\chi_1|v|\right)} dv \\
    &\leq \frac{1.0253}{\pi} \int_{1 / \pi \wedge T^{3/4}/\pi}^{t_1^* \wedge T^{3/4}/\pi} v^{-1} e^{-\frac{T^{1/2}v^2}{2}\left(1-4\pi\chi_1t_1^*\right)} dv \\
    & = \frac{1.0253}{\pi} \int_{\pi^{-1}\sqrt{1-4\pi\chi_1t_1^*}\left(T^{1/4} \wedge T\right)}^{\sqrt{1-4\pi\chi_1t_1^*}\left(t_1^*T^{1/4} \wedge T/\pi\right)} v^{-1} e^{-v^2/2} dv \\
    &= \frac{1.0253}{\pi} \int_{(T^{1/2}\wedge T^2)(1-4\pi\chi_1t_1^*) / (2\pi^2)}^{(t_1^{*2}T^{1/2}\wedge T^2/\pi^2)(1-4\pi\chi_1t_1^*) / 2} u^{-1} e^{-u} du \\
    &= \frac{1.0253}{\pi} \left( \Gamma\left( 0 , (T^{1/2}\wedge T^2)(1-4\pi\chi_1t_1^*) / (2\pi^2) \right) \right. \\
    & \left. \quad - \Gamma\left( 0 , (t_1^{*2}T^{1/2}\wedge T^2/\pi^2)(1-4\pi\chi_1t_1^*) / 2 \right) \right).
\end{align*}

To control $J_5(T)$, we use~Equation~\eqref{eq:properties_Psi} to write
\begin{align*}
    & J_5(T)
    \leq \frac{1.0253}{\pi} \int_{t_1^*T^{1/4} \wedge T/\pi}^{T/\pi} u^{-1}|\caracfsum(u)|du.
\end{align*}

To control $J_1(0, T^{1/4} / \pi, T / \pi, T)$, we use~Equation~\eqref{eq:properties_Psi} and a change of variable
\begin{align*}
    J_1(0, T^{1/4} / \pi, T / \pi, T)
    & \leq \frac{1.0253}{\pi} \int_{T^{1/4}/\pi \wedge T/\pi}^{T/\pi} u^{-1}e^{-u^2/2}du \\
    & = \frac{1.0253}{\pi} \int_{(T^{1/2} \wedge T^2) / (2\pi^2)}^{T^2/(2\pi^2)} u^{-1}e^{-u}du \\
    & = \frac{1.0253}{\pi} \left( \Gamma\left( 0 , (T^{1/2} \wedge T^2) / (2\pi^2) \right) - \Gamma\left( 0 , T^2/(2\pi^2) \right) \right). \qquad \Box
\end{align*}

\section{Technical lemmas}
\label{sec:lemmas}

\subsection{Control of the residual term in an Edgeworth expansion under Assumption~\ref{hyp:basic_as_inid}}
\label{ssec:proof_edg_exp_inid}

For $\eps \in (0, 1/3)$ and $t \geq 0$, let us define the following quantities:
\begingroup \allowdisplaybreaks
\begin{align}
    &\Rinid(t,\eps)
    := \frac{U_{1,1,n}(t)+U_{1,2,n}(t)}{2(1-3\eps)^2} \nonumber \\ %
    & \qquad\qquad\qquad + e_{1,n}(\eps)\left(\frac{t^8K_{4,n}^2}{2n^2}\left(\frac{1}{24} + \frac{P_{1,n}(\eps)}{2(1-3\eps)^2}\right)^2 \right.\nonumber \\
    & \left. \qquad\qquad\qquad + \frac{|t|^7|\lambda_{3,n}|K_{4,n}}{6n^{3/2}}\left(\frac{1}{24} + \frac{P_{1,n}(\eps)}{2(1-3\eps)^2}\right)\right) , \label{eq:def_R_1n} \\
    & P_{1,n}(\eps) \nonumber \\
    := & \frac{144 + 48 \eps + 4 \eps^2
    + \left\{ 96 \sqrt{2\eps} + 32\eps
    + 16\sqrt{2}\eps^{3/2} \right\}
    \Indicator\!\left\{\exists i \in \{1,...,n\}:\E[X_i^3] \neq 0 \right\}}{576} , \nonumber \\
    & e_{1,n}(\eps) := \exp\left( \eps^2\left( \frac{1}{6}+\frac{2P_{1,n}(\eps)}{(1-3\eps)^2} \right) \right), \label{eq:def_e_1n} \\
    & U_{1,1,n}(t)
    := \frac{t^6}{24}\left(\frac{K_{4,n}}{n}\right)^{3/2} + \frac{t^8}{24^2}\left(\frac{K_{4,n}}{n}\right)^2 , \nonumber \\
    & U_{1,2,n}(t) \nonumber \\
    := & \left(\frac{|t|^5}{6}\left(\frac{K_{4,n}}{n}\right)^{5/4} + \frac{t^6}{36}\left(\frac{K_{4,n}}{n}\right)^{3/2} + \frac{|t|^7}{72}\left(\frac{K_{4,n}}{n}\right)^{7/4} \right)\Indicator\left\{\exists i \in \{1,...,n\}:\E[X_i^3] \neq 0 \right\}.
    \label{eq:def_U_112n}
\end{align}
\endgroup

We want to show the following lemma:
\begin{lemma}
    \label{lem:taylor_exp_cf}
    Under Assumption~\ref{hyp:basic_as_inid}, for every $\eps\in(0,1/3)$ and $t$ such that $|t|\leq\sqrt{2\eps}(n/K_{4,n})^{1/4},$
    we have
    \begin{align*}
        & \left| \caracfsum(t) - e^{-\frac{t^2}{2}}     \left(1-\frac{it^3\lambda_{3,n}}{6\sqrt{n}}\right) \right| \\
        \leq & e^{-t^2/2}\left\{ \frac{t^4K_{4,n}}{8n}\left(\frac{1}{3}
        + \frac{1}{(1-3\eps)^2}\right)
        + \frac{e_1(\eps)|t|^6|\lambda_{3,n}|^2}{72n}
        + \Rinid(t,\eps) \right\}.
    \end{align*}
\end{lemma}

\noindent
{\it Proof of
Lemma~\ref{lem:taylor_exp_cf}: }
Remember that $\gamma_j := \E[X_j^4]$,
$\sigma_j := \sqrt{\E[X_j^2]}$,
$B_n := \sqrt{\sum_{i=1}^n \E[X_i^2]}$ and
$K_{4,n}:= n^{-1} \sum_{i=1}^n\mathbb{E}[X_i^4]
\, / \left(n^{-1} B_n^2 \right)^{2}$. Applying Cauchy-Schwartz inequality, we get
\begin{align}
\label{eq:edg_exp_1}
    \max_{1 \leq j \leq n}\sigma_j^2
    \leq \max_{1 \leq j \leq n}\gamma_j^{1/2}
    \leq \bigg( \sum_{j=1}^n \gamma_j \bigg)^{1/2}
    = B_n^2(K_{4,n}/n)^{1/2},
\end{align}
\begin{align}
\label{eq:edg_exp_2}
    \max_{1 \leq j \leq n} \mathbb{E}[|X_j|^3]
    \leq \max_{1 \leq j \leq n} \gamma_j^{3/4}
    \leq \bigg( \sum_{j=1}^n \gamma_j \bigg)^{3/4}
    = B_n^3(K_{4,n}/n)^{3/4},
\end{align}
and
\begin{align}
\label{eq:edg_exp_3}
    \max_{1 \leq j \leq n}\gamma_j
    \leq \sum_{j=1}^n \gamma_j = B_n^4K_{4,n}/n.
\end{align}

Combining~\eqref{eq:edg_exp_1}, \eqref{eq:edg_exp_2} and \eqref{eq:edg_exp_3}, we observe that for every $\eps\in(0,1)$ and $t$ such that $|t| \leq \sqrt{2\eps}(n/K_{4,n})^{1/4},$
\begin{align}
    \label{eq:bound_max_Ujn}
    \max_{1\leq j\leq n}
    \left\{ \frac{\sigma_j^2 t^2}{2 B_n^2 }
    + \frac{\mathbb{E}[|X_j|^3] \times |t|^3}{6 B_n^3}
    + \frac{\gamma_jt^4}{24B_n^4} \right\}
    \leq 3 \eps.
\end{align}

As we assume that $X_j$ has a moment of order four for every $j=1,\dots,n$, the characteristic functions $(f_{P_{X_j}})_{j=1,\dots,n}$ are four times differentiable on $\Rb$. Applying a Taylor-Lagrange expansion, we get the existence of a complex number $\theta_{1,j,n}(t)$ such that $|\theta_{1,j,n}(t)|\leq 1$ and
\begin{align*}
    U_{j,n}(t) := f_{P_{X_j}}(t/B_n)-1
    = -\frac{\sigma_j^2t^2}{2B_n^2} - \frac{i\mathbb{E}[X_j^3]\,t^3}{6B_n^3} + \frac{\theta_{1,j,n}(t)\gamma_jt^4}{24B_n^4},
\end{align*}
for every $t \in \Rb$ and $j = 1,\dots,n$.
Let $\log$ stand for the principal branch of the complex logarithm function. For every $\eps\in(0,1/3)$ and $t$ such that
$|t| \leq \sqrt{2 \eps}(n / K_{4,n})^{1/4},$ Equation~(\ref{eq:bound_max_Ujn}) shows that $|U_{j,n}(t)| \leq 3\eps < 1$, so that we can use another Taylor-Lagrange expansion. This ensures existence of a complex number $\theta_{2,j,n}(t)$ such that $|\theta_{2,j,n}(t)| \leq 1$ and
\begin{align*}
    \log(f_{P_{X_j}}(t/B_n))
    = \log(1+U_{j,n}(t))
    = U_{j,n}(t)
    - \frac{U_{j,n}(t)^2}{2(1+\theta_{2,j,n}(t)U_{j,n}(t))^2}.
\end{align*}
Summing over $j= 1,\dots,n$ and exponentiating, we can claim that under the same conditions on $t$ and $\eps,$
\begin{align*}
    & \caracfsum(t) = \exp\left(-\frac{t^2}{2}-\frac{it^3\lambda_{3,n}}{6\sqrt{n}}+t^4\sum_{j=1}^n\frac{\theta_{1,j,n}(t)\gamma_j}{24B_n^4}-\sum_{j=1}^n\frac{U_{j,n}(t)^2}{2(1+\theta_{2,j,n}(t)U_{j,n}(t))^2} \right).
\end{align*}
A third Taylor-Lagrange expansion guarantees existence of a complex number $\theta_{3,n}(t)$ with modulus at most $\exp\bigg( \frac{t^4K_{4,n}}{24n} + \sum_{j=1}^n\frac{|U_{j,n}(t)|^2}{2|1+\theta_{2,j,n}(t)U_{j,n}(t)|^2} \bigg)$ such that
\begin{align*}
    \caracfsum(t) = \, & e^{-t^2 / 2} \left(1-\frac{it^3\lambda_{3,n}}{6\sqrt{n}}+t^4\sum_{j=1}^n\frac{\theta_{1,j,n}(t)\gamma_j}{24B_n^4}-\sum_{j=1}^n\frac{U_{j,n}(t)^2}{2(1+\theta_{2,j,n}(t)U_{j,n}(t))^2} \right. \nonumber \\
    & \left. \quad +\frac{\theta_{3,n}(t)}{2}\left(-\frac{it^3\lambda_{3,n}}{6\sqrt{n}}+t^4\sum_{j=1}^n\frac{\theta_{1,j,n}(t)\gamma_j}{24B_n^4}-\sum_{j=1}^n\frac{U_{j,n}(t)^2}{2(1+\theta_{2,j,n}(t)U_{j,n}(t))^2}\right)^2\right).
\end{align*}
Using the triangle inequality and its reverse version, as well as the restriction on $|t| \leq \sqrt{2\eps}(n/K_{4,n})^{1/4},$ we can write
\begin{align}
\label{eq:edg_exp_4}
    &\Bigg| \caracfsum(t) - e^{- t^2 / 2} \left(1-\frac{it^3\lambda_{3,n}}{6\sqrt{n}}\right) \Bigg|
    \leq e^{-t^2/2} \times \Bigg( \frac{t^4K_{4,n}}{24n} + \frac{1}{2(1-3\eps)^2}\sum_{j=1}^n|U_{j,n}(t)|^2 
    \nonumber \\
    & \qquad\qquad\qquad\qquad\qquad\qquad\qquad\quad + \frac{1}{2} \exp \bigg( \frac{\eps^2}{6} + \frac{1}{2(1-3\eps)^2}\sum_{j=1}^n|U_{j,n}(t)|^2 \bigg) \nonumber \\
    & \qquad\qquad\qquad\qquad\qquad \times \bigg( \frac{|t|^3|\lambda_{3,n}|}{6\sqrt{n}} + \frac{t^4K_{4,n}}{24n} + \frac{1}{2(1-3\eps)^2}\sum_{j=1}^n|U_{j,n}(t)|^2 \bigg)^2 \Bigg).
\end{align}
We now control $\sum_{j=1}^n |U_{j,n}(t)|^2$. We first expand the squares, giving the decomposition
\begin{align}
\label{eq:edg_exp_5}
    \sum_{j=1}^n |U_{j,n}(t)|^2
    = & \, \frac{t^4\sum_{j=1}^n\sigma_j^4}{4B_n^4}
    + \frac{t^6\sum_{j=1}^n|\mathbb{E}[X_j^3]|^2}{36B_n^6}
    + \frac{t^8\sum_{j=1}^n\gamma_j^2}{24^2B_n^8} \nonumber \\
    &
    + \frac{|t|^5\sum_{j=1}^n\sigma_j^2|\mathbb{E}[X_j^3]|}{6B_n^5}
    + \frac{t^6\sum_{j=1}^n\sigma_j^2\gamma_j}{24B_n^6}
    + \frac{|t|^7\sum_{j=1}^n|\mathbb{E}[X_j^3]|\gamma_j}{72B_n^7}.
\end{align}
Using Equations~(\ref{eq:edg_exp_1})-(\ref{eq:edg_exp_3}), we can bound the right-hand side of Equation~(\ref{eq:edg_exp_5}) in the following manner
\begin{align*}
    & \frac{t^4\sum_{j=1}^n\sigma_j^4}{4B_n^4}
    \leq \frac{t^4K_{4,n}}{4n},
\end{align*}
\begin{align*}
    & \frac{t^6\sum_{j=1}^n\sigma_j^2\gamma_j}{24B_n^6} + \frac{t^8\sum_{j=1}^n\gamma_j^2}{24^2B_n^8}
    \leq \frac{t^6}{24}\left(\frac{K_{4,n}}{n}\right)^{3/2} + \frac{t^8}{24^2}\left(\frac{K_{4,n}}{n}\right)^2 =: U_{1,1,n}(t),
\end{align*}
and
\vspace{0.2cm}
\begin{align}
\label{eq:nocount_U_112n}
    &\frac{|t|^5\sum_{j=1}^n\sigma_j^2|\mathbb{E}[X_j^3]|}{6B_n^5}
    + \frac{t^6\sum_{j=1}^n|\mathbb{E}[X_j^3]|^2}{36 B_n^6}
    + \frac{|t|^7\sum_{j=1}^n|\mathbb{E}[X_j^3]|\gamma_j}{72B_n^7} \nonumber \\
    \leq & \left(\frac{|t|^5}{6}\left(\frac{K_{4,n}}{n}\right)^{5/4} + \frac{t^6}{36}\left(\frac{K_{4,n}}{n}\right)^{3/2} + \frac{|t|^7}{72}\left(\frac{K_{4,n}}{n}\right)^{7/4} \right)\Indicator\left\{\exists i \in \{1,...,n\}:\E[X_i^3] \neq 0 \right\} \nonumber \\
    =: & U_{1,2,n}(t).
\end{align}
Moreover, we have $\sum_{j=1}^n U_{j,n}(t)^2
\leq \frac{t^4K_{4,n}}{n} P_{1,n}(\eps)$ under our conditions on $\eps$ and $t.$
Combining Equation~(\ref{eq:edg_exp_4}), the decomposition~(\ref{eq:edg_exp_5}) and the previous three bounds, and grouping similar terms together, we conclude that for every $\eps\in(0,1/3)$ and $t$ such that ${|t|\leq\sqrt{2\eps}(n/K_{4,n})^{1/4},}$
\begin{align*}
    \bigg| &\caracfsum(t) - e^{-\frac{t^2}{2}} \left(1-\frac{it^3\lambda_{3,n}}{6\sqrt{n}}\right) \bigg| \\
    &\leq  e^{-t^2/2} \Bigg\{ \frac{t^4K_{4,n}}{8n}\left(\frac{1}{3} + \frac{1}{(1-3\eps)^2}\right) + \frac{e_{1,n}(\eps)|t|^6|\lambda_{3,n}|^2}{72n} + \frac{U_{1,1,n}(t)+U_{1,2,n}(t)}{2(1-3\eps)^2} \\
    & \quad \quad \quad + e_{1,n}(\eps)\left(\frac{t^8K_{4,n}^2}{2n^2}\left(\frac{1}{24} + \frac{P_{1,n}(\eps)}{2(1-3\eps)^2}\right)^2
    + \frac{|t|^7|\lambda_{3,n}|K_{4,n}}{6n^{3/2}}\left(\frac{1}{24} + \frac{P_{1,n}(\eps)}{2(1-3\eps)^2}\right)\right) \Bigg\},
    \quad
\end{align*}
where $e_{1,n}(\eps) := \exp\left( \eps^2\left( \frac{1}{6}+\frac{2P_{1,n}(\eps)}{(1-3\eps)^2} \right) \right).$ Combining this with the definition of $\Rinid(t,\eps)$ finishes the proof. $\Box$

\subsection{Control of the residual term in an Edgeworth expansion under Assumption~\ref{hyp:basic_as_iid}}
\label{ssec:proof_edg_exp_iid}

Lemma~\ref{lem:taylor_exp_cf} can be improved in the \iid{} framework. To do so, we introduce analogues of $\Rinid(t,\eps),$ $P_{1,n}(\eps),$ $e_{1,n}(\eps)$ and $U_{1,2,n}(t)$ defined by
\begingroup \allowdisplaybreaks
\begin{align}
    \Riid(t,\eps)
    &:= \frac{U_{2,2,n}(t)}{2(1-3\eps)^2}
    + e_{2,n}(\eps)\Bigg( \frac{t^8}{8n^2}\left(\frac{K_{4,n}}{12} +
    \frac{1}{4 (1 - 3\eps)^2} +
    \frac{P_{2,n}(\eps)}{576 (1 - 3\eps)^2} \right)^2 \nonumber \\
    & \hspace{3.5cm} + \frac{|t|^7|\lambda_{3,n}|}{12n^{3/2}}\left(\frac{K_{4,n}}{12} +
    \frac{1}{4 (1 - 3\eps)^2} +
    \frac{P_{2,n}(\eps)}{576 (1 - 3\eps)^2} \right) \Bigg) ,  \label{eq:def_R_2n}  \\
    P_{2,n}(\eps)
    &:= \frac{96\sqrt{2\eps}|\lambda_{3,n}|}{(K_{4,n}^{1/4}n^{1/4})}
    + 48\eps \left(\frac{K_{4,n}}{n}\right)^{1/2}
    + \frac{32\eps\lambda_{3,n}^2}{(K_{4,n}n)^{1/2}}
    + \frac{16\sqrt{2}K_{4,n}^{1/4}|\lambda_{3,n}|\eps^{3/2}}{n^{3/4}} \nonumber \\
    & \quad\; + \frac{4\eps^2K_{4,n}}{n} , \label{eq:definition_P2n} \\
    e_{2,n}(\eps) & := \exp\left( \eps^2\left( \frac{1}{6}+\frac{1}{2(1-3\eps)^2}
    +\frac{2P_{2,n}(\eps)}{576(1-3\eps)^2}
    \right) \right), \label{eq:definition_e2n} \\
    U_{2,2,n}(t) 
    & := \frac{|t|^5|\lambda_{3,n}|}{6n^{3/2}}
    + \frac{t^6K_{4,n}}{24n^2}
    + \frac{t^6\lambda_{3,n}^2}{36n^2}
    + \frac{|t|^7K_{4,n}|\lambda_{3,n}|}{72n^{5/2}}
    + \frac{t^8K_{4,n}^2}{576n^3}. \nonumber
\end{align}
\endgroup
Note that
\begin{equation*}
    e_{2,n}(\eps)
    = e_{3}(\eps)
    \exp\left( \frac{2 \eps^2 P_{2,n}(\eps)}{576(1-3\eps)^2} \right),
\end{equation*}
where 
\begin{equation}
\label{eq:definition_e3}
    e_{3}(\eps) := 
    e^{\eps^2/6 + \eps^2 / (2(1-3 \eps) )^2}.    
\end{equation}

\begin{lemma}
    \label{lem:taylor_exp_cf_iid}
    Under Assumption~\ref{hyp:basic_as_iid}, for every $\eps\in(0,1/3)$ and $t$ such that $|t|\leq\sqrt{2\eps} (n/K_{4,n})^{1/4},$
    \begin{align*}
        \left| \caracfsum(t) -  e^{-\frac{t^2}{2}}     \left(1-\frac{it^3\lambda_{3,n}}{6\sqrt{n}}\right) \right|
        &\leq \, e^{-t^2/2} \Bigg\{ \frac{t^4K_{4,n}}{8n}\left(\frac{1}{3} + \frac{1}{(1-3\eps)^2}\right)
        \\ & \hspace{2cm} 
        + \frac{e_{2,n}(\eps)|t|^6|\lambda_{3,n}|^2}{72n} + \Riid(t,\eps) \Bigg\}.
    \end{align*}
\end{lemma}

\noindent
{\it Proof of
Lemma~\ref{lem:taylor_exp_cf_iid}: }
This proof is very similar to that of Lemma~\ref{lem:taylor_exp_cf}. We note that $B_n = \sigma \sqrt{n}.$
As before, using two Taylor-Lagrange expansions successively, we can write that for every $\eps\in(0,1/3)$ and $t$ such that $|t| \leq \sqrt{2\eps n}/K_{4,n}^{1/4}$
\begin{align*}
    & \log(f_{P_{X_n}}(t/B_n)) = U_{1,n}(t) - \frac{U_{1,n}(t)^2}{2(1+\theta_{2,n}(t)U_{1,n}(t))^2},
\end{align*}
where $$U_{1,n}(t) := -\frac{t^2}{2n}-\frac{i\lambda_{3,n}t^3}{6n^{3/2}}+\frac{\theta_{1,n}(t)K_{4,n}t^4}{24n^2},$$ and $\theta_{1,n}(t)$ and $\theta_{2,n}(t)$ are two complex numbers with modulus bounded by 1.
Using a third Taylor-Lagrange expansion, we can write that for some complex $\theta_{3,n}(t)$ with modulus bounded by $\exp\left( \frac{K_{4,n}t^4}{24n} + \frac{n|U_{1,n}(t)|^2}{2(1-3\eps)^2} \right),$ the following holds
\begin{align*}
    \caracfsum(t) &= e^{-\frac{t^2}{2}} \Bigg( 1-\frac{it^3\lambda_{3,n}}{6\sqrt{n}}+\frac{t^4K_{4,n}\theta_{1,n}(t)}{24n}-\frac{nU_{1,n}(t)^2}{2(1+\theta_{2,n}(t)U_{1,n}(t))^2}  \\
    & \quad \quad \quad + \frac{\theta_{3,n}(t)}{2}\left(-\frac{it^3\lambda_{3,n}}{6\sqrt{n}}+\frac{t^4K_{4,n}\theta_{1,n}(t)}{24n}-\frac{nU_{1,n}(t)^2}{2(1+\theta_{2,n}(t)U_{1,n}(t))^2}\right)^2 \Bigg).    
\end{align*}
Using the triangle inequality and its reverse version in addition to the condition $|t| \leq \sqrt{2\eps} (n/K_{4,n})^{1/4},$ we obtain
\begin{align}
    \label{eq:edg_exp_iid_1}
    \bigg| \caracfsum(t) &- e^{-t^2/2}\left(1-\frac{it^3\lambda_{3,n}}{6\sqrt{n}}\right) \bigg| 
    \leq e^{-t^2/2}
    \bigg\{ \frac{t^4K_{4,n}}{24n} + \frac{nU_{1,n}(t)^2}{2(1-3\eps)^2} \nonumber \\
    &+ \frac{1}{2} \exp \bigg( \frac{\eps^2}{6} + \frac{n|U_{1,n}(t)|^2}{2(1-3\eps)^2} \bigg)
    \times \left( \frac{|t|^3|\lambda_{3,n}|}{6\sqrt{n}} + \frac{t^4K_{4,n}}{24n} + \frac{nU_{1,n}(t)^2}{2(1-3\eps)^2} \right)^2 \bigg\}.
\end{align}
We can decompose $n U_{1,n}(t)^2$ as
\begin{align}
    \label{eq:edg_exp_iid_2}
    n U_{1,n}(t)^2  & = \frac{t^4}{4n}
    + \underbrace{ \frac{|t|^5|\lambda_{3,n}|}{6n^{3/2}}
    + \frac{t^6K_{4,n}}{24n^2} + \frac{t^6\lambda_{3,n}^2}{36n^2}
    + \frac{|t|^7K_{4,n}|\lambda_{3,n}|}{72n^{5/2}}
    + \frac{t^8K_{4,n}^2}{576n^3} }_{=U_{2,2,n}(t)} \nonumber \\
    & \leq \frac{t^4}{n} \left( \frac{1}{4} + \frac{P_{2,n}(\eps)}{576} \right).
\end{align}
Combining Equations~(\ref{eq:edg_exp_iid_1}) and~(\ref{eq:edg_exp_iid_2}) and grouping terms, we conclude that for every ${\eps\in(0,1/3)}$ and $t$ such that $|t| \leq \sqrt{2\eps}(n/K)^{1/4},$
\begin{align*}
    \Bigg| &\caracfsum(t) - e^{-\frac{t^2}{2}} \left(1-\frac{it^3\lambda_{3,n}}{6\sqrt{n}}\right) \Bigg| \\
    &\leq \, e^{-t^2/2}\Bigg\{\frac{t^4K_{4,n}}{8n}\left(\frac{1}{3} + \frac{1}{(1-3\eps)^2}\right) + \frac{e_{2,n}(\eps)t^6\lambda_{3,n}^2}{72n} + \frac{U_{2,2,n}(t)}{2(1-3\eps)^2} \\
    & \hspace{2cm} + e_{2,n}(\eps)\Bigg[ \frac{t^8}{8n^2}\left(\frac{K_{4,n}}{12} +
    \frac{1}{4 (1-3\eps)^2} +
    \frac{P_{2,n}(\eps)}{576 (1-3\eps)^2} \right)^2
    \\
    & \hspace{4cm} +
    \frac{|t|^7|\lambda_{3,n}|}{12n^{3/2}}\left(\frac{K_{4,n}}{12} +
    \frac{1}{4 (1-3\eps)^2} +
    \frac{P_{2,n}(\eps)}{576 (1-3\eps)^2} \right) \Bigg] \Bigg\}.
\end{align*}
\hfill $\Box$

\subsection{Bound on integrated \texorpdfstring{$\Rinid$}{Rinid} and \texorpdfstring{$\Riid$}{Riid}}

\subsubsection{Bound on integrated \texorpdfstring{$\Rinid$}{Rinid}}
\label{sec_bound_R1n}

Our goal in this section is to compute a bound on
\begin{align*}
    &\RinidInt(\eps) := \frac{1.0253}{\pi} \int_0^\infty u e^{-u^2/2} \Rinid(u,\eps) du
    = A_1(n, \eps) + \cdots + A_7(n, \eps),
\end{align*}
where
\begin{align*}
    A_1(n, \eps) &:= \frac{1.0253}{2(1-3\eps)^2 \pi} \int_0^\infty
    u e^{-u^2/2} \frac{u^6}{24}\left(\frac{K_{4,n}}{n}\right)^{3/2} du, \\
    A_2(n, \eps) &:= \frac{1.0253}{2(1-3\eps)^2 \pi} \int_0^\infty
    u e^{-u^2/2} \frac{u^8}{24^2}\left(\frac{K_{4,n}}{n}\right)^2 du, \\
    A_3(n, \eps) &:= \frac{1.0253}{ 2(1-3\eps)^2 \pi} \int_0^\infty
    u e^{-u^2/2} \frac{u^5}{6}\left(\frac{K_{4,n}}{n}\right)^{5/4} du
    \times \Indicator\left\{\exists i \in \{1,...,n\}:\E[X_i^3] \neq 0 \right\}
    , \\
    A_4(n, \eps) &:= \frac{1.0253}{ 2(1-3\eps)^2 \pi} \int_0^\infty
    u e^{-u^2/2} \frac{u^6}{36}\left(\frac{K_{4,n}}{n}\right)^{3/2} du
    \times \Indicator\left\{\exists i \in \{1,...,n\}:\E[X_i^3] \neq 0 \right\}
    , \\
    A_5(n, \eps) &:= \frac{1.0253}{ 2(1-3\eps)^2 \pi} \int_0^\infty
    u e^{-u^2/2} \frac{u^7}{72}\left(\frac{K_{4,n}}{n}\right)^{7/4} du
    \times \Indicator\left\{\exists i \in \{1,...,n\}:\E[X_i^3] \neq 0 \right\}
    , \\
    A_6(n, \eps) &:= \frac{1.0253 e_{1,n}(\eps)}{\pi} \int_0^\infty
    u e^{-u^2/2}\frac{u^8 K_{4,n}^2}{2n^2}\bigg(\frac{1}{24} + \frac{P_{1,n}(\eps)}{2(1-3\eps)^2}\bigg)^2 du, \\
    A_7(n, \eps) &:= \frac{1.0253 e_{1,n}(\eps)}{\pi} \int_0^\infty
    u e^{-u^2/2} \frac{u^7|\lambda_{3,n}|K_{4,n}}{6n^{3/2}}
    \bigg(\frac{1}{24}
    + \frac{P_{1,n}(\eps)}{2(1-3\eps)^2}\bigg) du,
\end{align*}
where
\begin{align*}
    P_{1,n}(\eps)
    & := \frac{144 + 48 \eps + 4 \eps^2
    + \left\{ 96 \sqrt{2\eps} + 32\eps
    + 16\sqrt{2}\eps^{3/2} \right\}
    \Indicator \left\{\exists i \in \{1,...,n\}:\E[X_i^3] \neq 0 \right\}}{576} , \nonumber \\
    e_{1,n}(\eps) & := \exp\left( \eps^2\left( \frac{1}{6}+\frac{2P_{1,n}(\eps)}{(1-3\eps)^2} \right) \right).
\end{align*}

\begin{lemma}
    For any $p > 0$,
    $\int_0^{+\infty} u^p e^{-u^2/2} du
    = 2^{(p-1)/2} \Gamma \big((p+1)/2 \big)$.
    \label{lemma:int_up_expmu2}
\end{lemma}
\begin{proof}
    We use the change of variable $v = u^2/2$, $u = \sqrt{2v}$,
    $dv = u du$, $du = dv/\sqrt{2v}$, so that
    \begin{align*}
        \int_0^{+\infty} u^p e^{-u^2/2} du
        &= \int_0^{+\infty} 2^{(p-1)/2}v^{(p-1)/2} e^{-v} dv
        = 2^{(p-1)/2} \int_0^{+\infty} v^{(p-1)/2} e^{-v} dv,
    \end{align*}
    and, by definition of \(\Gamma(\cdot)\), this is equal to \(2^{(p-1)/2} \Gamma \big((p+1)/2 \big)\)
    as claimed.
\end{proof}

By Lemma~\ref{lemma:int_up_expmu2}, we get the following equalities
\begingroup \allowdisplaybreaks
\begin{align*}
    A_1(n, \eps) &= \frac{1.0253}{48(1-3\eps)^2 \pi} \left(\frac{K_{4,n}}{n}\right)^{3/2}
    \times 2^{(7-1)/2} \Gamma(8/2), \\
    A_2(n, \eps) &= \frac{1.0253}{24^2 \times 2(1-3\eps)^2 \pi} 
    \left(\frac{K_{4,n}}{n}\right)^2
    2^{(9-1)/2} \Gamma(10/2), \\
    A_3(n, \eps) &= \frac{1.0253}{12(1-3\eps)^2 \pi}
    \left(\frac{K_{4,n}}{n}\right)^{5/4}
    2^{(6-1)/2} \Gamma(7/2)
    \times \Indicator\left\{\exists i \in \{1,...,n\}:\E[X_i^3] \neq 0 \right\}
    ,\\
    A_4(n, \eps) &= \frac{1.0253}{ 72(1-3\eps)^2 \pi}
    \left(\frac{K_{4,n}}{n}\right)^{3/2}
    2^{(7-1)/2} \Gamma(8/2)
    \times \Indicator\left\{\exists i \in \{1,...,n\}:\E[X_i^3] \neq 0 \right\}
    , \\
    A_5(n, \eps) &= \frac{1.0253}{144(1-3\eps)^2 \pi}
    \left(\frac{K_{4,n}}{n}\right)^{7/4}
    2^{(8-1)/2} \Gamma(9/2)
    \times \Indicator\left\{\exists i \in \{1,...,n\}:\E[X_i^3] \neq 0 \right\}
    , \\
    A_6(n, \eps) &= \frac{1.0253 e_{1,n}(\eps)}{2\pi}
    \left(\frac{K_{4,n}}{n}\right)^{2}
    \bigg(\frac{1}{24} + \frac{P_{1,n}(\eps)}{2(1-3\eps)^2}\bigg)^2
    2^{(9-1)/2} \Gamma(10/2), \\
    A_7(n, \eps) &= \frac{1.0253 e_{1,n}(\eps)}{6\pi}
    \frac{|\lambda_{3,n}|K_{4,n}}{n^{3/2}}
    \bigg(\frac{1}{24}
    + \frac{P_{1,n}(\eps)}{2(1-3\eps)^2}\bigg)
    2^{(9-1)/2} \Gamma(10/2).
\end{align*}
\endgroup

When skewness is not ruled out, \(\RinidInt(\eps)\) can be written as a polynomial in~$n$ with coefficients \(a_{k,n}\) that still depend on~$n$ but only through the moments \(\lambdatroisn\) and \(\Kquatren\) (and are therefore constant when the distribution of the observations is fixed with the sample size)
\begin{equation}
\label{eq:simplified_R1n_int_skewed}
     \RinidInt(\eps) = \frac{a_{1,n}(\eps)}{n^{5/4}} + \frac{a_{2,n}(\eps)}{n^{3/2}} + \frac{a_{3,n}(\eps)}{n^{7/4}} + \frac{a_{4,n}(\eps)}{n^2}
\end{equation}
\begin{align*}
    a_{1,n}(\eps) & = 
    \frac{1.0253 \times 2^{(6-1)/2}\Gamma(7/2) }{12(1-3\eps)^2\pi} K_{4,n}^{5/4} \\
    a_{1,n}(\eps) & 
    \overset{\text{if } \eps = 0.1}{\approx} 
    1.0435 K_{4,n}^{5/4} \\
    %
    a_{2,n}(\eps) & = 
    \frac{1.0253 \Kquatren^{3/2}}{48(1-3\eps)^2 \pi} 
    2^{(7-1)/2} \Gamma(8/2) +
     \frac{1.0253 \Kquatren^{3/2}}{ 72(1-3\eps)^2 \pi}
    2^{(7-1)/2} \Gamma(8/2) \\
    & \qquad + \frac{1.0253 e_{1,n}(\eps) |\lambdatroisn| \Kquatren}{6\pi}
    \bigg(\frac{1}{24}
    + \frac{P_{1,n}(\eps)}{2(1-3\eps)^2}\bigg)
    2^{(9-1)/2} \Gamma(10/2) \\
    a_{2,n}(\eps) & 
    \overset{\text{if } \eps = 0.1}{\approx} 
    1.1101 \Kquatren^{3/2}
    + 8.2383 |\lambdatroisn| \times \Kquatren \\
    %
    a_{3,n}(\eps) & =
     \frac{1.0253 \Kquatren^{7/4}}{144(1-3\eps)^2 \pi}
    2^{(8-1)/2} \Gamma(9/2) \\
    a_{3,n}(\eps) & 
    \overset{\text{if } \eps = 0.1}{\approx}
    0.6087 \Kquatren^{7/4} \\    
    a_{4,n}(\eps) & =
    \frac{1.0253 \Kquatren^2}{24^2 \times 2(1-3\eps)^2 \pi} 
    2^{(9-1)/2} \Gamma(10/2) \\
    & \;\;\; + 
     \frac{1.0253 e_{1,n}(\eps) \Kquatren^2}{2\pi}
    \bigg(\frac{1}{24} + \frac{P_{1,n}(\eps)}{2(1-3\eps)^2}\bigg)^2
    2^{(9-1)/2} \Gamma(10/2) \\
    a_{4,n}(\eps) & 
    \overset{\text{if } \eps = 0.1}{\approx}
    9.8197\Kquatren^{2}.
\end{align*}

When \(\E[X_i^3] = 0\) for every~$i$, which implies $\lambda_{3,n} = 0$, simplifications occur so that we get
\begin{align}
    \RinidInt(\eps) = \frac{a_{1,n}(\eps)}{n^{3/2}} + \frac{a_{2,n}(\eps)}{n^2}
    \label{eq:simplified_R1n_int_noskewed}
\end{align}
\begin{align*}
    a_{1,n}(\eps) & = 
    \frac{1.0253}{48(1-3\eps)^2 \pi} K_{4,n}^{3/2}
    \times 2^{(7-1)/2} \Gamma(8/2), \\
    a_{1,n}(\eps) & \overset{\text{if } \eps = 0.1}{\approx}
    0.6661 \Kquatren^{3/2} \nonumber \\
    a_{2,n}(\eps) &= 
    \frac{1.0253}{24^2 \times 2(1-3\eps)^2 \pi} 
    K_{4,n}^2
    2^{(9-1)/2} \Gamma(10/2) \nonumber \\
    & + \frac{1.0253 e_{1,n}(\eps)}{2\pi}
    K_{4,n}^2
    \bigg(\frac{1}{24} \frac{P_{1,n}(\eps)}{2(1-3\eps)^2}\bigg)^2
    2^{(9-1)/2} \Gamma(10/2) \nonumber \\
    a_{2,n}(\eps) & \overset{\text{if } \eps = 0.1}{\approx}
    6.1361 \Kquatren^{2}.
\end{align*}

\subsubsection{Bound on integrated \texorpdfstring{$\Riid$}{Riid}}
\label{sec_bound_R2n}

Our goal in this section is to compute a bound on
\begin{align*}
    &\RiidInt(\eps) := \frac{1.0253}{\pi} \int_0^\infty u e^{-u^2/2} \Riid(u,\eps) du
    = \widetilde{A}_1(n , \eps) + \cdots + \widetilde{A}_7(n , \eps),
\end{align*}
where
\begingroup \allowdisplaybreaks
\begin{align*}
    \widetilde{A}_1(n, \eps)
    &:= \frac{1.0253}{2(1-3\eps)^2 \pi}
    \int_0^\infty u e^{-u^2/2} \frac{u |\lambda_{3,n}|}{6n^{3/2}} du , \\
    \widetilde{A}_2(n, \eps)
    &:= \frac{1.0253}{2(1-3\eps)^2 \pi}
    \int_0^\infty u e^{-u^2/2} \frac{u^6 K_{4,n}}{24n^2} du , \\
    \widetilde{A}_3(n, \eps)
    &:= \frac{1.0253}{2(1-3\eps)^2 \pi}
    \int_0^\infty u e^{-u^2/2} \frac{u^6\lambda_{3,n}^2}{36n^2} du , \\
    \widetilde{A}_4(n, \eps)
    &:= \frac{1.0253}{2(1-3\eps)^2 \pi}
    \int_0^\infty u e^{-u^2/2} \frac{u^7 K_{4,n}|\lambda_{3,n}|}{72n^{5/2}} du , \\
    \widetilde{A}_5(n, \eps)
    &:= \frac{1.0253}{2(1-3\eps)^2 \pi}
    \int_0^\infty u e^{-u^2/2} \frac{u^8 K_{4,n}^2}{576n^3} du, \\
    \widetilde{A}_6(n, \eps)
    &:= \frac{1.0253}{\pi} \int_0^\infty u e^{-u^2/2} 
    e_{2,n}(\eps) \frac{u^8}{8n^2}\left(\frac{K_{4,n}}{12}
    + \frac{1}{4(1-3\eps)^2}
    + \frac{P_{2,n}(\eps)}{576 (1-3\eps)^2} \right)^2 du, \\
    \widetilde{A}_7(n, \eps)
    &:= \frac{1.0253}{\pi} \int_0^\infty u e^{-u^2/2} e_{2,n}(\eps) \frac{u^7|\lambda_{3,n}|}{12n^{3/2}}\left(\frac{K_{4,n}}{12}
    + \frac{1}{4(1-3\eps)^2}
    + \frac{P_{2,n}(\eps)}{576 (1-3\eps)^2}\right) du.
\end{align*}
\endgroup

By Lemma~\ref{lemma:int_up_expmu2}, we get
\begingroup \allowdisplaybreaks
\begin{align}
    \widetilde{A}_1(n, \eps)
    &= \frac{1.0253}{2(1-3\eps)^2 \pi}
    \frac{|\lambda_{3,n}|}{6n^{3/2}}
    2^{(2-1)/2} \Gamma \big(3/2 \big), \nonumber \\
    \widetilde{A}_2(n, \eps)
    &= \frac{1.0253}{2(1-3\eps)^2 \pi}
    \frac{K_{4,n}}{24n^2}
    2^{(7-1)/2} \Gamma (8/2), \nonumber \\
    \widetilde{A}_3(n, \eps)
    &= \frac{1.0253}{2(1-3\eps)^2 \pi} \frac{\lambda_{3,n}^2}{36n^2}
    2^{(7-1)/2} \Gamma (8/2), \nonumber \\
    \widetilde{A}_4(n, \eps)
    &= \frac{1.0253}{2(1-3\eps)^2 \pi}
    \frac{K_{4,n}|\lambda_{3,n}|}{72n^{5/2}}
    2^{(8-1)/2} \Gamma (9/2), \nonumber \\
    \widetilde{A}_5(n, \eps)
    &= \frac{1.0253}{2(1-3\eps)^2 \pi}
    \frac{K_{4,n}^2}{576n^3}
    2^{(9-1)/2} \Gamma (10/2), \nonumber \\
    \widetilde{A}_6(n, \eps)
    &= \frac{1.0253}{\pi} 
    e_{2,n}(\eps) \frac{1}{8n^2}\left(\frac{K_{4,n}}{12}
    + \frac{1}{4(1-3\eps)^2}
    + \frac{P_{2,n}(\eps)}{576 (1-3\eps)^2} \right)^2
    2^{(9-1)/2} \Gamma (10/2), \nonumber \\
    \widetilde{A}_7(n, \eps)
    &= \frac{1.0253}{\pi} 
    e_{2,n}(\eps) \frac{|\lambda_{3,n}|}{12n^{3/2}}\left(\frac{K_{4,n}}{12}
    + \frac{1}{4(1-3\eps)^2}
    + \frac{P_{2,n}(\eps)}{576 (1-3\eps)^2}\right)
    2^{(8-1)/2} \Gamma (9/2).
    \label{eq:simplified_R2n_int}
\end{align}
\endgroup

\medskip

When skewness is not ruled out and \(\Kquatren = O(1)\), the previous equalities show that \(\RiidInt(\eps)\) is of order \(n^{-3/2}\) for any \(\eps \in (0,1/3)\).
When \(\lambdatroisn = 0\), we get an improved rate equal to \(n^{-2}\).

\medskip

In our main theorems, we set \(\eps = 0.1\).
In that case, we can get two explicit bounds\footnote{Bounds instead of equalities in the sense that we round up to the fourth digit the obtained numerical constants.} on \(\RiidInt(0.1)\).
When skewness is not ruled out, the bound \(\RiidIntSkew\) can be written as {\color{black}in Equation~\eqref{eq:definition_R2n_int_skew_bound_when_eps_01}.}
Absent skewness, the bound \(\RiidIntNoskew\) is defined
{\color{black}by Equation~\eqref{eq:definition_R2n_int_noskew_bound_when_eps_01}.}
The quantity \(e_{2,n}(0.1)\) that appears in the two previous expressions can be upper bounded by 
\begin{equation*}
    e_{2,n}(0.1) \leq \exp\big(
    0.0119 + 
    0.000071 \times P_{2,n}(0.1) \big),
\end{equation*}
where \(P_{2,n}(0.1)\) itself satisfies
\begin{equation*}
    P_{2,n}(0.1)
    \leq \frac{42.9326|\lambda_{3,n}|}{(K_{4,n}^{1/4}n^{1/4})}
    + 4.8 \left(\frac{K_{4,n}}{n}\right)^{1/2}
    + \frac{3.2\lambda_{3,n}^2}{(K_{4,n}n)^{1/2}}
    + \frac{0.7156K_{4,n}^{1/4}|\lambda_{3,n}|}{n^{3/4}}
    + \frac{0.04K_{4,n}}{n}.
\end{equation*}

\subsection{Bounding incomplete Gamma-like integrals}

For every $p \geq 1$, $0 \leq l, m \leq q$ and $T > 0$, we define $J_1$, $J_2$, and $J_3$ by
\begin{align}
    &J_{1}(p, l, m, T)
    := \frac{1}{T} \int_l^m
    \left| \Psi(u/T) \right| u^p e^{-u^2/2} du 
    \label{eq:def_J1_p} \\
    &J_{2}(p,l,m,q,T)
    := \frac{1}{T} \int_l^m
    |\Psi(u/T)|u^p \exp \Bigg( -\frac{u^2}{2}
    \bigg( 1 - \frac{4\chi_1}{q} u
    - \sqrt{\frac{K_{4,n}}{n}} \bigg)
    \Bigg) du.
    \label{eq:def_J2_p} \\
    &J_{3}(p,l,m,q,T)
    := \frac{1}{T} \int_l^m
    |\Psi(u/T)|u^p \exp \Bigg( -\frac{u^2}{2}
    \bigg( 1 - \frac{4\chi_1}{q} u
    - \frac{1}{n} \bigg)
    \Bigg) du.
    \label{eq:def_J3_p}
\end{align}

We show now that all these integrals can be bounded by differences of incomplete Gamma functions.

\begin{lemma}
    We have
    \begin{align*}
        \big| J_{1}(p, l, m, T) \big|
        &\leq \frac{1.0253 \times 2^{p/2-2} 
        \big| \Gamma(p/2, m^2/2) - \Gamma(p/2, l^2/2) \big|}{\pi}
    \end{align*}
    \label{lemma:bound_J1p}
\end{lemma}

\noindent {\it Proof of Lemma \ref{lemma:bound_J1p}.}
Without loss of generality, we assume $l \leq m$.
By the first inequality in~(\ref{eq:properties_Psi}), we get
\begin{align*}
    J_{1}(p, l, m, T)
    &\leq \frac{1.0253}{2 \pi} \int_l^m u^{p-1}
    e^{-u^2/2}du
    = \frac{1.0253}{2 \pi} \int_{l^2/2}^{m^2/2} \sqrt{2v}^{p-1} e^{-v} \frac{dv}{\sqrt{2v}} \\
    &= \frac{1.0253 \times 2^{p/2-2}}{\pi} \int_{l^2/2}^{m^2/2} v^{p/2 - 1} e^{-v} dv,
\end{align*}
where we used the change of variable $v=u^2/2$, $dv = u du$, so that $du = dv / \sqrt{2v}$. The proof is completed by recognizing that the last integral can be written as a difference of two incomplete Gamma functions.
$\Box$

\begin{lemma}
    Let $\Delta := (1 - 4 \chi_1 - \sqrt{K_{4,n}/n}) / 2$ and $\gamma(a, x) := \int_0^x |v|^{a-1} \exp(-v) dv$. We have
    \begin{align*}
        \big| J_{2}(p,l,m,q,T) \big|
        &\leq \frac{1.0253}{4\pi} \times
        \begin{cases}
            |\Delta|^{-p/2}
            \big|\gamma(p/2, \Delta m^2) - \gamma(p/2, \Delta l^2) \big|,
            & \text{ if } {\color{black}\Delta > 0
            \text{ or } \Delta < 0}, \\
            (2/p) \cdot (m^p - l^p),
            & \text{ if } \Delta = 0. 
        \end{cases}
    \end{align*}
    \label{lemma:bound_J2p}
\end{lemma}

\noindent {\it Proof of Lemma \ref{lemma:bound_J2p}.}
Without loss of generality, we assume $l \leq m$.
Using the first inequality in~(\ref{eq:properties_Psi}) and the fact that $0 \leq u/q \leq 1$ when $u \in [l,m]$, we get
\begin{align*}
    J_{2}(p,l,m,q,T)
    &\leq \frac{1.0253}{2\pi}
    \int_l^m u^{p-1}
    \exp \Bigg( - \frac{u^2}{2}\bigg( 1
    - 4 \chi_1
    - \sqrt{\dfrac{K_{4,n}}{n}} \bigg)
    \Bigg) du.
\end{align*}

We can then write
\begin{align*}
    J_{2}(p,l,m,q,T)
    &\leq \frac{1.0253}{2\pi}
    \int_l^m u^{p-1}
    \exp \Big( - u^2 \Delta \Big) du \\
    & = \frac{1.0253}{2\pi}
    \int_l^m u^{p-1}
    \exp \Big( - u^2 |\Delta| \sign(\Delta) \Big) du.
\end{align*}
If $\Delta \neq 0$, we do the change of variable $v = u^2 \Delta$,
$dv = 2 \Delta u du$,
$u = \sqrt{v / \Delta}$,
$du
= (2 \sqrt{v \Delta})^{-1} dv$,
and get
\begin{align*}
    J_{2}(p,l,m,q,T)
    &\leq \frac{1.0253}{2\pi}
    \int_{[\Delta l^2 \, , \, \Delta m^2]} (v/\Delta)^{(p-1)/2}
    \exp \Big( - v \Big)
    (2 \sqrt{v \Delta})^{-1} dv \\
    &= \frac{1.0253}{4\pi}
    \int_{[\Delta l^2 \, , \, \Delta m^2]} (|v|/|\Delta|)^{(p-1)/2}
    \exp \Big( - v \Big)
    (\sqrt{|v| |\Delta|})^{-1} dv \\
    &= |\Delta|^{-p/2} \frac{1.0253}{4\pi}
    \int_{[\Delta l^2 \, , \, \Delta m^2]} |v|^{p/2-1}
    e^{-v} dv,
\end{align*}
where we remarked that $v / \Delta > 0$ in the sense that either $\Delta > 0$ and in this case $v > 0$ as well; or $\Delta < 0$ and $v < 0$ as well.
Finally, we get
\begin{align*}
    J_{2}(p,l,m,q,T)
    &\leq \frac{1.0253}{4\pi} \times
    \begin{cases}
        |\Delta|^{-p/2} \times
        \int_{\Delta l^2}^{\Delta m^2} |v|^{p/2-1} e^{-v} dv
        & \text{ if } \Delta > 0 \\
        |\Delta|^{-p/2} \times
        \int_{\Delta m^2}^{\Delta l^2} |v|^{p/2-1} e^{-v} dv
        & \text{ if } \Delta < 0  \\
        2 \int_{l}^{m} v^{p-1} dv
        & \text{ if } \Delta = 0 
    \end{cases}
\end{align*}
If $\Delta \neq 0$, the bound can be rewritten as
\begin{align*}
    J_{2}(p,l,m,q,T)
    &\leq |\Delta|^{-p/2} \frac{1.0253}{4\pi} \big|\gamma(p/2, \Delta m^2) - \gamma(p/2, \Delta l^2) \big|.
    \quad \Box
\end{align*}

\begin{lemma}
    If $n \geq 3$, then
    \begin{align*}
        \big| J_{3}(p,l,m,q,T) \big|
        &\leq \frac{1.0253 \times 2^{3p/2-2}
        \big| \Gamma(p/2, m^2/8) - \Gamma(p/2, l^2/8) \big|}{\pi}.
    \end{align*}
    \label{lemma:bound_J3p}
\end{lemma}

\noindent {\it Proof of Lemma \ref{lemma:bound_J3p}.}
Without loss of generality, we assume $l \leq m$.
Using the first inequality in~(\ref{eq:properties_Psi}), we get
\begin{align*}
    J_{3}(p,l,m,q,T)
    &\leq \frac{1.0253}{2\pi}
    \int_l^m u^{p-1}
    \exp \Bigg( - \frac{u^2}{2}
    \bigg( 1 - \frac{4\chi_1}{q} u - \dfrac{1}{n} \bigg)
    \Bigg) du.
\end{align*}
We bound $u/q$, by $1$, so that
\begin{align*}
    J_{3}(p,l,m,q,T)
    &\leq \frac{1.0253}{2\pi}
    \int_l^m u^{p-1}
    \exp \Bigg( - \frac{u^2}{2} \bigg( 1 - 4 \chi_1 - \dfrac{1}{n} \bigg)
    \Bigg) du.
\end{align*}
Note that 
$1 - 4 \chi_1 - 1/n > 1/4$ when $n \geq 3$.
When this is the case, using the same change of variable and computations, we get the same result as for the previous lemma.
$\Box$

\subsection{Statement and proof of Proposition~\ref{prop:charac_bound_differentiable}}
\label{proof:prop:prop:charac_bound_differentiable}

A bound on the tail of the characteristic function is nearly equivalent to a regularity condition on the density. We detail this in the following proposition. 
The first part of this proposition is taken from \cite[Theorem 2.5.4]{ushakov2011} (see also \cite{ushakov1999some}).

\begin{proposition}
Let $p \geq 1$ be an integer, $Q$ be a probability measure that admits a density $q$ with respect to Lebesgue's measure, and $f_Q$ its corresponding characteristic function.

\begin{enumerate}
    \item If $q$ is $(p-1)$ times differentiable and $q^{(p-1)}$ is a function with bounded variation, then
    \begin{equation*}
        |f_Q(t)| \leq \frac{\mathrm{Vari}[q^{(p-1)}]}{|t|^p},
    \end{equation*}
    where $\mathrm{Vari}[\psi]$ denotes the total variation of a function $\psi$.
    
    \item If $t \mapsto |t|^{p-1} |f_Q(t)|$ is integrable on a neighborhood of $+ \infty$,
    then $q$ is $(p-1)$ times differentiable.
\end{enumerate}
\label{prop:charac_bound_differentiable}
\end{proposition}

Remark that the existence of $C>0$ and $\beta > 1$ such that $|f_Q(t)| \leq C / \big( |t|^p \log(|t|)^\beta \big)$ is sufficient to satisfy the integrability condition in the second part of Proposition~\ref{prop:charac_bound_differentiable}.

\bigskip

\noindent {\it Proof of Proposition \ref{prop:charac_bound_differentiable}.2.} The assumed integrability condition implies that $f_Q$ is absolutely integrable, and therefore we can apply the inversion formula \citep[Theorem 1.2.6]{ushakov2011}
so that for any $x \in \Rb$,
\begin{align*}
    q(x) = \int_{-\infty}^{+ \infty} r(x,t) dt.
\end{align*}
where $r(x,t) := \dfrac{1}{2 \pi} e^{-itx} f_P(t)$.
Note that $r$ is infinitely differentiable with respect to $x$, and that
\begin{align*}
    \left| \frac{\partial r(x,t)}{\partial x^{p-1}} \right|
    = \left| \dfrac{1}{2 \pi} (-it)^{p-1} e^{-itx} f_Q(t) \right|
    =  \dfrac{1}{2 \pi} |t|^{p-1} \big|f_Q(t)\big| ,
\end{align*}
which is integrable with respect to $t$, by assumption.
This concludes the proof that $q$ is $(p-1)$ times differentiable, as $r$ is measurable.




\end{document}